\documentclass[a4paper, nonatbib]{elsarticle}
\usepackage[a4paper, scale=0.8]{geometry}

\usepackage[T1]{fontenc}
\usepackage{babel}
\usepackage{lmodern,textcomp,latexsym}
\usepackage{amsthm,amssymb,amsmath,amsfonts,dsfont,mathrsfs,bbm,bm,mathtools}
\usepackage{cite}
\usepackage{color}
\usepackage{booktabs}
\usepackage{paralist} 
\usepackage[footnotesize]{caption}
\usepackage{pgfplots} 
\pgfplotsset{
        compat=newest,
        my axis style/.style={
            xlabel style={
                font=\footnotesize,
            },
            ylabel style={
                font=\footnotesize,
            },
            legend style={
                font=\footnotesize,
            },
            ticklabel style={
                font=\footnotesize,
            },
            every axis plot/.append style={
                line width=1pt,
            },
            legend pos=north west,
            legend columns = 3
        },
    }
\newlength\fheight
\newlength\fwidth
\newlength\fheightwide
\newlength\fwidthwide
\newlength\fheightthree
\newlength\fwidththree
\usepackage{subcaption}
\usetikzlibrary{external}
\tikzset{external/only named=true}
\usetikzlibrary{fit}
\newcommand{\graybox}[3][20]{%
  \tikz[baseline=(X.base)] {
    \node[inner sep=2pt, outer sep=3pt, fill=gray!#1, draw=white, minimum size=2em, anchor=mid] (X) {\phantom{$#3$}};
    \node[anchor=mid, overlay] {$#2$};
    }
}

\newcommand{\dashedbox}[1]{%
  \tikz[baseline=(X.base)] 
    \node[draw=white, dashed, inner sep=2pt, minimum size=2em] (X) {$#1$};
    }

\usepackage[title]{appendix}
\usepackage{enumitem}

\setdefaultenum{1)}{\theenumi.1)}{}{}

\theoremstyle{plain}
\newtheorem{theorem}{Theorem}
\newtheorem{lemma}[theorem]{Lemma}
\newtheorem{cor}[theorem]{Corollary}

\theoremstyle{definition}
\newtheorem{assu}{Assumption}
\newtheorem{defin}{Definition}

\theoremstyle{remark}
\newtheorem*{remark}{Remark}
\newtheorem{exmp}{Example}[section]

\usepackage[hidelinks]{hyperref}

\newcommand{\diff}[1][]{\mathrm{d}#1}

\newcommand{\pd}[2]{\frac{\partial #1}{\partial #2}}
\newcommand{\pdat}[3]{\left. \pd{#1}{#2} \right\vert_{#3}}
\newcommand{\td}[2]{\frac{\diff #1}{\diff #2}}

\newcommand{\norm}[1]{\left\Vert #1 \right\Vert}
\newcommand{\abs}[1]{\left\vert #1 \right\vert}
\newcommand{\T}{^{\mathop{\mathrm{T}}}}

\newcommand{\diag}{{\mathop{\mathrm{diag}}}}

\AtBeginDocument{\colorlet{defaultcolor}{.}}

\newcommand{\Rspace}{\mathbb{R}}
\newcommand{\Cspace}{\mathbb{C}}
\newcommand{\Nspace}{\mathbb{N}}
\newcommand{\Zspace}{\mathbb{Z}}

\newcommand{\re}{\mathrm e}

\newcommand{\ri}{\mathrm i}

\newcommand{\ic}{\ri} 
\newcommand{\ex}{\re} 

\newcommand{\val}{\bm{\alpha}}

\newcommand{\vPh}{\mathbf \Phi}

\newcommand{\vf}{\mathbf f}
\newcommand{\vg}{\mathbf g}

\newcommand{\vn}{\mathbf n}

\newcommand{\vp}{\mathbf p}

\newcommand{\vx}{\mathbf x}
\newcommand{\vy}{\mathbf y}
\newcommand{\vz}{\mathbf z}

\newcommand{\vzero}{\mathbf{0}}

\newcommand{\vA}{\mathbf A}

\newcommand{\vC}{\mathbf C}
\newcommand{\vD}{\mathbf D}
\newcommand{\vE}{\mathbf E}

\newcommand{\vH}{\mathbf H}
\newcommand{\vI}{\mathbf I}
\newcommand{\vJ}{\mathbf J}

\newcommand{\vM}{\mathbf M}

\newcommand{\vP}{\mathbf P}
\newcommand{\vQ}{\mathbf Q}

\newcommand{\vW}{\mathbf W}

\newcommand{\cA}{\mathcal{A}}

\newcommand{\cD}{\mathcal{D}}

\newcommand{\cF}{\mathcal{F}}

\newcommand{\cJ}{\mathcal{J}}
\newcommand{\cK}{\mathcal{K}}
\newcommand{\cL}{\mathcal{L}}

\newcommand{\cP}{\mathcal{P}}

\newcommand{\vPhsub}[1][]{\vPh_{\mathrm{subh}#1}}

\newcommand{\vCsub}[1][]{\tilde{\vC}_{\mathrm{subh}#1}}

\newcommand{\cPsub}{\cP_{\mathrm{subh}}}

\begin{document}
\begin{frontmatter}
	\journal{}
	\title{Explicit error bounds and guaranteed convergence of the Koopman-Hill projection stability method for linear time-periodic dynamics}				
	\author[unistutt]{Fabia Bayer\texorpdfstring{\corref{cor1}}{}}
	\ead{bayer@inm.uni-stuttgart.de}
	\author[unistutt]{Remco I. Leine}
	\ead{leine@inm.uni-stuttgart.de}
	\affiliation[unistutt]{organization={University of Stuttgart, Institute for Nonlinear Mechanics},addressline={Pfaffenwaldring 9},city={70569 Stuttgart},country={Germany}}
	\cortext[cor1]{Corresponding author}
	\begin{abstract}
		The Koopman-Hill projection method offers an efficient approach for stability analysis of linear time-periodic systems, and thereby also for the Floquet stability analysis of periodic solutions of nonlinear systems. However, its accuracy has previously been supported only by numerical evidence, lacking rigorous theoretical guarantees. This paper presents the first explicit error bound for the truncation error of the Koopman-Hill projection method, establishing a solid theoretical foundation for its application. The bound applies to linear time-periodic systems whose Fourier coefficients decay exponentially with a sufficient rate, and is derived using constructive series expansions. The bound quantifies the difference between the true and approximated fundamental solution matrices, clarifies conditions for guaranteed convergence, and enables conservative but reliable inference of Floquet multipliers and stability properties. Additionally, the same methodology applied to a subharmonic formulation demonstrates improved convergence rates of the latter. Numerical examples, including the Mathieu equation and the Duffing oscillator, illustrate the practical relevance of the bound and underscore its importance as the first rigorous theoretical justification for the Koopman-Hill projection method.
	\end{abstract}
	
\end{frontmatter}
\section{Introduction}

The Koopman-Hill projection method was introduced in~\cite{Bayer2023} to approximate the fundamental solution matrix of linear time-periodic systems. However, the justification for the method in prior work, which is based on the Koopman framework, does not provide any insight into the magnitude of the error incurred by truncation. In principle, this error could become very large and render the method useless. Prior works reported only numerical evidence for the accuracy of the method~\cite{Bayer2023,Bayer2024,McGurk2025}. The central contribution of the present work is to strengthen the theoretical foundation of the Koopman-Hill method by providing a closed-form bound for the truncation error for a certain class of linear time-periodic systems. 

Floquet theory is the mathematical field concerning the evolution and the local stability analysis of linear time-periodic (LTP) ordinary differential equations (ODEs) of the form 
\begin{align}\label{eq:background:ode}
	\dot{\vy}(t) = \vJ(t) \vy(t) \;,
\end{align}
where the system matrix $\vJ(t) = \vJ(t + T) \in \Rspace ^{n \times n}$ is periodic with period $T$.
This field of study, pioneered by Gaston Floquet in the 1880s~\cite{Floquet1883}, remains of prime interest in various fields of engineering and applied mathematics. While parametrically excited mechanical systems such as simulation models for interacting gears~\cite{Abboud2021}, MEMS devices~\cite{Moran2013}, or helicopter ground resonance problems~\cite{IgnaciodaSilva2019} can intrinsically showcase LTP dynamics, arguably the largest field of application of Floquet theory is the study of stability of periodic orbits or periodic solutions of nonlinear dynamical systems. Periodic solutions arise naturally in many applications, such as structural dynamics~\cite{Geradin2014} and electrical circuits~\cite{Schubert2016}. While only stable periodic orbits can be attained in practical experiments, unstable periodic solutions characterize the global behavior of a system, and knowledge of them can be critical to safe operation~\cite{Horvath2021}. 
If a nonlinear ordinary differential equation with a periodic orbit is linearized around that periodic orbit, the resulting dynamical system is LTP and of the form of Equation~\eqref{eq:background:ode}. Due to the Hartman-Grobman theorem, the stability properties of the periodic solution are equivalent to those of the trivial solution of Equation~\eqref{eq:background:ode} in the hyperbolic case~\cite{Nayfeh1995}. 

In contrast to linear time-invariant systems, closed-form solutions for LTP systems with more than one state are usually not available. The stability of the LTP system in Equation~\eqref{eq:background:ode} is not immediately determined by the spectrum of the system matrix $\vJ(t)$, so sophisticated numerical methods are required. Methods for numerical stability analysis of LTP systems or, equivalently, hyperbolic periodic orbits of nonlinear ODEs, can classically be categorized into two families, with their corresponding operations taking place in the time domain or in the frequency domain, respectively.

The time domain approaches rely on numerical integration of the ODE~\eqref{eq:background:ode} over one period to obtain the system's monodromy matrix, whose eigenvalues govern the stability. While effective, these methods can be computationally expensive, especially for systems with long periods or many states. The approximation error that these approaches incur is determined by the chosen numerical integration scheme~\cite{Peletan2013, Hairer2008}. Global error bounds depending on the local error for fixed or variable-step integration methods are available for Lipschitz continuous dynamics~\cite{Hairer2008}, but are not often explicitly computed in practice. Integration schemes that specifically measure the global error are computationally expensive~\cite{Skeel1986}. 

Frequency domain methods are based on Fourier series representations of the LTP system of Equation~\eqref{eq:background:ode} and its solutions. This leads to the formulation of an eigenvalue problem involving the bi-infinite Hill matrix, which captures the system's characteristics in the frequency domain~\cite{Magnus1966,Lazarus2010}. This approach is especially appealing for stability analysis of periodic solutions computed numerically by the harmonic balance method (HBM), where a truncated variant of this Hill matrix is immediately available as a byproduct of HBM. The spectrum of the truncated matrix consists of physically meaningful Floquet exponents, as well as spurious eigenvalues without physical meaning due to truncation effects.

Sorting criteria \cite{Moore2005,Lazarus2010,Detroux2015} are commonly used to identify the physically meaningful eigenvalues. The imaginary-part-based sorting criterion proposed by Zhou et al.~\cite{Zhou2004} is guaranteed to converge as the truncation order tends to infinity, but no bounds are provided for any finite truncation order~$N$. 
Other eigenvalue sorting methods can exhibit better convergence rates in some numerical examples~\cite{Guillot2020, Wu2022}, but come with no convergence guarantee whatsoever. In addition, in the vicinity of period-doubling bifurcations, the non-uniqueness of Floquet exponents poses a challenge for all sorting-based methods~\cite{Colaitis2022}.

The Koopman-Hill method, introduced by the authors in 2023 \cite{Bayer2023}, offers a novel approach combining aspects of both frequency and time domains. By reinterpreting the truncated Hill matrix as the system matrix of a linear time-\emph{in}variant differential equation, the fundamental solution matrix of the original system, Equation~\eqref{eq:background:ode}, can be obtained using a single matrix exponential involving the Hill matrix, bypassing direct numerical integration. Previous numerical studies have shown that this method provides accurate ap\-proxi\-ma\-tions of the mo\-no\-dro\-my matrix, with the approximation error decreasing as the truncation order of the Hill matrix increases~\cite{Bayer2023,Bayer2024,McGurk2025, Bayer2025a}.

Despite the promising numerical performance, the previously available justification for the Koopman-Hill method based on the Koopman framework~\cite{Bayer2023}, recounted and discussed in Section~\ref{sec:background:KoopLift}, has lacked a rigorous error analysis. The central contribution of the present work is to strengthen the Koopman-Hill method by providing a closed-form error bound for the difference between the Koopman-Hill approximation and the fundamental solution matrix. The approach taken here does not rely on methodology from the Koopman framework. Rather, the idea is to find exact series expressions for the considered matrices, compare them term by term, and bound the difference. The considered series expression is novel to the authors' knowledge. The resulting bound decays exponentially with the truncation order of the Hill matrix. It depends only on two scalar parameters which govern the decay behavior of the Fourier coefficients of the original LTP dynamics and is applicable whenever these Fourier coefficients decay exponentially at a sufficient rate. 

Previous works of the authors~\cite{Bayer2023, Bayer2024} observed that a modified formulation of the Koopman-Hill method that includes additional, sub\-har\-monic frequencies can lead to a significant improvement in the approximation accuracy.  
However, these works did not provide an explanation for this improved performance. Using the formalism developed in this paper, we provide also an error bound for the subharmonic formulation. Compared to the error bound of the direct formulation, this error bound decays at twice the rate, indicating improved accuracy.

The structure of this paper is as follows. 
Section~\ref{sec:background} provides the necessary theoretical background on Floquet theory and summarizes the previous motivation for the Koopman-Hill projection method based on the Koopman framework. Section~\ref{sec:overview} introduces our main results. After stating the involved series expressions, we derive the closed-form error bound in Section~\ref{sec:overview}. The subharmonic formulation is introduced and bounded subsequently in Section~\ref{sec:subharmonics}. The proofs for the series expressions which were stated in Section~\ref{sec:overview} are carried out in Section~\ref{sec:series}. Section~\ref{sec:examples} illustrates the theoretical findings with three numerical examples. Finally, Section~\ref{sec:conclusion} concludes the paper with a summary and discussion of the main results and an outlook on future research.
\section{Theoretical background and notation}\label{sec:background}
This section provides the necessary background for the developments to follow. First, we revisit Floquet theory. Next, the arguments of~\cite{Bayer2023}, justifying the Koopman-Hill projection method for computing the fun\-da\-men\-tal solution matrix of Equation~\eqref{eq:background:ode}, are summarized. Further, the multi-index notation that is used throughout the remainder of this paper is introduced.
\subsection{Floquet Theory}
This section revisits some well-known facts from Floquet theory that are the basis for the developments of the following sections. For a more comprehensive treatment of the concepts of this section, the interested reader is referred to~\cite{Teschl2012,Chicone2006,Nayfeh1995}. Floquet theory is concerned with the study of linear time-periodic (LTP) dynamical systems given by Equation~\eqref{eq:background:ode}. We restrict the class of considered LTP systems to ones where the system matrix has a pointwise convergent Fourier series
\begin{align}\label{eq:background:fourier}
	\vJ(t) = \vJ(t + T) = \sum_{k \in \Zspace} \vJ_k \ex^{\ic k \omega t}
\end{align}
with the complex-valued Fourier coefficient matrices $\dots, \vJ_{-1}, \vJ_0, \vJ_1, \dots \in \Cspace^{n \times n}$. A sufficient criterion for the existence of this Fourier series is that $\vJ(t)$ is piecewise continuous with piecewise continuous derivative~\cite{Krack2019}, which is the case in many engineering applications.

Due to the linearity of Equation~\eqref{eq:background:ode}, any initial condition $\vy(t_0)$ is mapped to its solution $\vy(t) = \vPh(t, t_0) \vy(t_0)$ by the fundamental solution matrix $\vPh(t, t_0)$, which is the unique solution of the matrix initial value problem
\begin{align}\label{eq:background:ode:matrix}
	\dot{\vPh}(t, t_0) = \vJ(t) \vPh(t, t_0) && \vPh(t_0, t_0) = \vI \;.	
\end{align}
In the literature, any full-rank $n \times n$ matrix that satisfies $\dot{\vPh} = \vJ \vPh$ is often referred to as a fundamental solution matrix, regardless of its specific initial condition. The solution of~\eqref{eq:background:ode:matrix}, i.e., the specific fundamental solution matrix with the identity initial condition, is then termed the \emph{principal} fundamental solution matrix~\cite{Teschl2012}. However, for brevity, we will refer to the unique solution of Equation~\eqref{eq:background:ode:matrix} simply as the fundamental solution matrix. Setting $t_0 = 0$ without loss of generality, we denote the fundamental solution matrix as $\vPh(t) = \vPh(t, 0)$.

Floquet's theorem~\cite{Teschl2012} states that the fundamental solution matrix can be expressed as
\begin{align}\label{eq:background:floquet}
	\vPh(t) = \vP(t) \ex^{\vQ t} \;,
\end{align}
where $\vQ \in\Cspace^{n \times n}$ is a constant matrix and $\vP(t) =\vP(t + T)$ is a $T$-periodic complex-valued matrix with $\vP(0) = \vI$. The long-term behavior of the solution is solely determined by the eigenvalues $\alpha_1, \dots, \alpha_n$ of $\vQ$, which are referred to as the Floquet exponents. If all Floquet exponents have a negative real part, then the equilibrium of the system given by Equation~\eqref{eq:background:ode} is asymptotically stable. If the LTP system~\eqref{eq:background:ode} was obtained by linearization of a nonlinear ODE around a periodic solution, then the Floquet exponents also govern the stability of that periodic solution.

The fundamental solution matrix evaluated after one period, $\vPh_T :=\vPh(T) = \ex^{\vQ T}$, is called the monodromy matrix. It is of special interest as it provides a discrete mapping of states from one period to the next. If all eigenvalues of the monodromy matrix are inside the unit circle, the solution will always decay from one period to the next, indicating asymptotic stability of the equilibrium at zero. In contrast, if at least one eigenvalue is outside the unit circle, the solution in the corresponding eigendirection will grow from one period to the next, indicating instability. The eigenvalues $\lambda_1, \dots, \lambda_n$ of the monodromy matrix are called Floquet multipliers and it can be easily seen from Floquet's theorem, Equation~\eqref{eq:background:floquet}, that they are related to the Floquet exponents by $\lambda_i = \ex^{\alpha_i T}$.

Since the monodromy matrix is usually not available in closed form, numerical methods are needed to approximate the Floquet multipliers. An overview over various methods to obtain the monodromy matrix by numerical integration of Equation~\eqref{eq:background:ode:matrix} is given in~\cite{Peletan2013}. These numerical integration methods will not be treated further in this work. Alternatively, the periodic nature of the differential equation can be exploited to determine the stability based on the Fourier coefficients of $\vJ$. These so-called Hill methods rely on the Hill matrix, which is constituted of by the Fourier coefficients of the system matrix $\vJ(t)$. The Hill methods classically consider the full spectral decomposition of the Hill matrix to assert stability~\cite{Peletan2013,Lazarus2010}. The following section introduces the Koopman-Hill projection method~\cite{Bayer2023}, which combines aspects of the classical Hill method with (closed-form) time-domain integration over a period.

\subsection{The Koopman-Hill projection method motivated through the Koopman framework}\label{sec:background:KoopLift}
Homogeneous linear time-invariant (LTI) systems have a closed-form solution given by the matrix ex\-po\-nen\-tial and are therefore straightforward to solve, in contrast to LTP systems such as Equation~\eqref{eq:background:ode}. The Koopman-Hill projection method, introduced in~\cite{Bayer2023}, is based on the idea of approximating the nonautonomous dynamics of Equation~\eqref{eq:background:ode} with an LTI system in a higher-di\-men\-sio\-nal state space. This approach is motivated by the Koopman framework, which represents nonlinear autonomous systems as linear autonomous systems of larger, potentially infinite, dimension. The following exposition summarizes the arguments from~\cite{Bayer2023} that motivate the use of a matrix exponential to approximate the fundamental solution matrix of the LTP system~\eqref{eq:background:ode}.

The Koopman operator was originally developed to extract global information from nonlinear autonomous dynamics via its spectral decomposition, known as Koopman eigenfunctions~\cite{Koopman1931,Mauroy2016,Mauroy2020b}. More recently, the Koopman operator has been applied in data-driven modeling to construct linear systems that capture nonlinear behavior~\cite{Brunton2022,Bruder2021,Budisic2012}. For the purposes of the Koopman-Hill projection, only the matrix representation of the Koopman generator is required. Readers interested in further concepts are referred to the extensive literature on the Koopman operator, such as~\cite{Mauroy2020b} and references therein.

Consider a nonlinear autonomous dynamical system
\begin{align}\label{eq:background:ode:nonlin}
	\dot{\vx} = \vf(\vx)
\end{align}
in $\Rspace^n$, with the flow map $\vx(t) = \varphi(t, \vx(0))$. Let $\cF$ denote a Banach space of observable functions $g: \Rspace^n \rightarrow \Cspace$. The semigroup of Koopman operators $\cK^t: g \mapsto g \circ \varphi(t, \cdot)$ describes the evolution of observables along system trajectories~\cite{Mauroy2020b}. Under suitable regularity conditions (see~\cite{Engel2008}), this semigroup is strongly continuous and generated by the infinitesimal Koopman generator
\begin{align}
	\cL = \lim_{t \rightarrow 0} \frac{\cK^t - \vI}{t} \;,
\end{align}
which acts as the Lie derivative of the observable along the flow:
\begin{align}\label{eq:background:infgen}
	(\cL g)(\vx) = \dot{g}(\vx) = \pd{g}{\vx} \vf(\vx) \;.
\end{align}
For example, the derivative of a Lyapunov function $V$ along system trajectories is given by $\dot{V}= \cL V$. The infinitesimal Koopman generator $\cL$ is linear on the space of observables, even though the original dynamics is nonlinear.

To apply the Koopman framework to the linear time-periodic system~\eqref{eq:background:ode}, it is necessary to reformulate the nonautonomous ODE as an autonomous system. This is achieved by introducing an extended state space with dynamics
\begin{subequations}\label{eq:background:ode:ext}
	\begin{align}
		\dot{t} &= 1\\
		\dot{\vy} &= \vJ(t) \vy \;,
	\end{align}
\end{subequations}
where the extended state is $\vx\T = (t, \vy\T)$. As function space $\cF$, consider specifically the span of the complex-valued basis functions
\begin{align}\label{eq:background:observable}
	g_{k, l}(t, \vy) = y_l \ex^{-\ic k \omega t} 
\end{align}
for $k \in \Zspace$ and $l \in \left\{ 1, \dots, n \right\}$, equipped with the $L_2$ norm.

Define the collection $\vg_k := (y_1 \ex ^{-\ic k \omega t}, \dots, y_n \ex^{- \ic k \omega t})\T$ of all $n$ states at one specific frequency as the $k$-th block of basis functions. Using the Fourier series expansion of $\vJ(t)$ from Equation~\eqref{eq:background:fourier}, the derivative of the $k$-th block is
\begin{align}\label{eq:background:observable_deriv}
	\dot{\vg}_k(t, \vy) &= 
	\left(-\ic k \omega \vI + \sum_{l \in \Zspace} \vJ_l \ex^{\ic l \omega t}\right) \ex^{- \ic k \omega t} \vy = -\ic k \omega \vg_k + \sum_{\tilde{l} \in \Zspace} \vJ_{k - \tilde{l}} \, \vg_{\tilde{l}} \;,
\end{align} 
where the Fourier series expression of Equation~\eqref{eq:background:fourier} for $\vJ(t)$ is used.

Stacking all blocks of basis functions by decreasing frequency yields the bi-infinite vector
\begin{align}\label{eq:background:liftstate}
	\vg(t, \vy)= 
	\left(
	\begin{array}{l}
		\vdots\\
		\vy \ex^{2 \ic \omega t} \\
		\vy \ex^{\ic \omega t} \\
		\vy \\
		\vy \ex^{-\ic \omega t} \\
		\vy \ex^{-2 \ic \omega t} \\
		\vdots
	\end{array} 
	\right) = \ex^{- \ic \vD_{\infty} \omega t} \vW_{\infty} \vy\;.
\end{align}
Here, $\vD_{\infty} = \diag(..., -2\vI, -\vI, \vzero, \vI, 2\vI, \dots)$ is a diagonal bi-infinite matrix whose diagonal entries are the integers in ascending order, each repeated $n$ times. The operator $\vW_{\infty} = (\dots, \vI, \vI, \vI, \dots)\T$ is a bi-infinite stack of $n \times n$ identity matrices, stacking the vector $\vy$ repeatedly. To describe the evolution of the observables for an arbitrary initial condition $\vy_0$ at $t_0 = 0$, the action of the Koopman operator on $\vg$ can be written as 
\begin{align}\label{eq:background:KOP:Phi}
	\left(\cK^{t} \vg\right)(0, \vy_0) 
	= \vg(t, \vy(t))
	= 
	\left(
	\begin{array}{r}
		\vdots \\
		\ex^{\ic \omega t}\vPh(t) \vy_0 \\
		\vPh(t) \vy_0 \\
		\ex^{- \ic \omega t}\vPh(t) \vy_0\\
		\vdots
	\end{array}
	\right)
	= \ex^{-\ic \omega \vD_{\infty} t} 
	\underbrace{
	\left(
	\begin{array}{c}
		\vdots\\
		\vPh(t)\\
		\vPh(t)\\
		\vPh(t)\\
		\vdots
	\end{array}
	\right) 
	}_{=: \vQ_{\infty}(t)}
	\vy_0\;.
\end{align}
The matrix $\vQ_{\infty}(t)$ collects a bi-infinite stack of fundamental solution matrices, each of which evolves the initial condition pertaining to a certain frequency.
From Equation~\eqref{eq:background:observable_deriv}, the Lie derivative of the bi-infinite basis vector $\vg$ is 
\begin{align}\label{eq:background:Hillmat}
	\dot{\vg} = \begin{pmatrix}
		\ddots & \\
		\dots & \vJ_0 + \ic \omega \vI & \vJ_{-1} & \vJ_{-2} & \dots \\
		\dots & \vJ_1 & \vJ_0 & \vJ_{-1} & \dots \\
		\dots & \vJ_2 & \vJ_1 & \vJ_0 - \ic \omega \vI & \dots \\
		& & & & \ddots
	\end{pmatrix} \vg =: \vH_{\infty} \vg \;,
\end{align}
which is linear in $\vg$. The bi-infinite block matrix $\vH_{\infty}$ of the above derivative expression, constructed from the Fourier coefficients of~$\vJ(t)$, is well-known in the study of LTP systems as the bi-infinite Hill matrix~\cite{Lazarus2010,Magnus1966}. Its point spectrum consists of the Floquet exponents, while its continuous and residual spectra are either empty or contain only~$\left\{ \pm \infty \right\}$~\cite{Zhou2004}. The operator $\vH_{\infty} $ is unbounded in~$\ell_2$ due to the $\ic \omega$ terms on the diagonal, a common feature of Koopman lifts not expressed in a finite basis of eigenfunctions~\cite{Brunton2021,Budisic2012}.

Consider a bi-infinite initial condition $\vz_0 \in \ell_{\infty}$ that is consistent with $\vg$, meaning that there exist~$t_0$ and~$\vy_0$ such that $\vz_0 = \vg(t_0, \vy_0)$. Substituting $\vg$ with $\vz \in \ell_{\infty}$ in Equation~\eqref{eq:background:Hillmat} yields the Cauchy problem
\begin{align}\label{eq:background:koopman:zdot}
	\dot{\vz} = \vH_{\infty} \vz && \vz(t_0) = \vz_0 = \vg(t_0, \vy_0) \;.
\end{align}

However, for every consistent initial condition $\vz_0\neq\vzero$ one has $\norm{\vH_{\infty}\vz_0}_{\infty}=\infty$, so the Cauchy problem~\eqref{eq:background:koopman:zdot} is ill-posed. 
This issue is typically addressed in the Koopman framework by projecting the Koopman generator onto a finite-dimensional subspace, which necessarily introduces truncation errors~\cite{Mauroy2020a}. In general cases, this truncation may completely change the problem at hand and there is no guarantee that the result is accurate. However, there exists a multitude of examples in literature involving Dynamic Mode Decomposition that show good accuracy despite the truncation, e.g.~\cite{Brunton2021}.

Specifically for the linear time-periodic case discussed here, the truncation entails that only the $2N+1$ central observable blocks $\vg_{-N}, \dots, \vg_N$ are retained, and all other blocks in Equation~\eqref{eq:background:observable_deriv} are discarded. The bi-infinite Hill matrix $\vH_{\infty}$ is thus replaced by its finite-dimensional truncation $\vH \in \Cspace^{n(2N+1) \times n(2N+1)}$, given by its $(2N+1)$ central row and column blocks. For instance, for $N = 2$, the truncated Hill matrix is 
\begin{align}\label{eq:background:H:finite}
	\vH = \left(\begin{array}{lllll}
		  \graybox[75]{\vJ_0}{\vJ_0} + 2 \ic \omega \vI 
		& \graybox[55]{\vJ_{-1}}{\vJ_0} 
		& \graybox[35]{\vJ_{-2}}{\vJ_0} 
		& \graybox[15]{\vJ_{-3}}{\vJ_0}
		& \graybox[0]{\vJ_{-4}}{\vJ_0} 
		\\
		  \graybox[55]{\vJ_1}{\vJ_0} 
		& \graybox[75]{\vJ_0}{\vJ_0} + \dashedbox{\ic \omega \vI} 
		& \graybox[55]{\vJ_{-1}}{\vJ_0} 
		& \graybox[35]{\vJ_{-2}}{\vJ_0} 
		& \graybox[15]{\vJ_{-3}}{\vJ_0} 
		\\
		  \graybox[35]{\vJ_2}{\vJ_0}
		& \graybox[55]{\vJ_1}{\vJ_0}
		& \graybox[75]{\vJ_0}{\vJ_0} \vphantom{+ \dashedbox{\ic \omega \vI}}
		& \graybox[55]{\vJ_{-1}}{\vJ_0} 
		& \graybox[35]{\vJ_{-2}}{\vJ_0} 
		\\
		  \graybox[15]{\vJ_3}{\vJ_0} 
		& \graybox[35]{\vJ_2}{\vJ_0} 
		& \graybox[55]{\vJ_1}{\vJ_0} 
		& \graybox[75]{\vJ_0}{\vJ_0} - \dashedbox{\ic \omega \vI} 
		& \graybox[55]{\vJ_{-1}}{\vJ_0} 
		\\
		  \graybox[0]{\vJ_4}{\vJ_0} 
		& \graybox[15]{\vJ_3}{\vJ_0} 
		& \graybox[35]{\vJ_2}{\vJ_0} 
		& \graybox[55]{\vJ_1}{\vJ_0} 
		& \graybox[75]{\vJ_0}{\vJ_0} - \dashedbox{2 \ic \omega \vI}
	\end{array} \right) \;,
\end{align}
where Fourier coefficient matrices $\vJ_k$ up to $k = \pm 2N = \pm 4$ appear. The background boxes in Equation~\eqref{eq:background:H:finite} indicate repeated blocks of $\vH$ and highlight its block-Toeplitz-like structure. 
The consistently initialized truncated initial value problem, involving the truncated matrix $\vH$ instead of the bi-infinite operator $\vH_{\infty}$, thus becomes
\begin{align}\label{eq:background:koopman:zdot:finite}
	\dot{\vz} = \vH \vz && \vz(t_0) = [\vg_{-N}(t_0, \vy_0)\T, \dots, \vg_N(t_0, \vy_0)\T]\T \;.
\end{align}
By Equation~\eqref{eq:background:observable}, consistent truncated initial conditions are linear in $\vy_0$ and given by 
\begin{align}\label{eq:background:init}
	\vz(t_0) = \ex^{- \ic \omega \vD t_0} \vW \vy_0 =: \vz_0\;,
\end{align}
where $\vW$ is a vertical stack of $2N+1$ identity matrices and $\vD = \diag \left(-N, \dots, N \right) \otimes \vI$ is a diagonal matrix that incorporates the initial time. They are the finite truncations of $\vW_{\infty}$ and $\vD_{\infty}$, respectively. If $t_0 = 0$, then $\vz(0) = \vW \vy_0$. For the case $N = 2$, the $\vW$ and $\vD$ matrices associated with $\vH$ from Equation~\eqref{eq:background:H:finite} are
\begin{align}
	\vD = 
	\left(
	\begin{array}{lllll}
	 - 2 \vI \\
	 & - \vI \\
	 && \vzero \\
	 &&&  \vI \\
	 &&&& 2 \vI
	\end{array}
	\right) \;, && \vW = 
	\left(
	\begin{array}{l}
		\vI \\
		\vI \\
		\vI \\
		\vI \\
		\vI
	\end{array}
	\right) \;, && \vz(0)= \begin{pmatrix} 
	\vy_0\\
	\vy_0\\
	\vy_0\\
	\vy_0\\
	\vy_0\\
\end{pmatrix} \;.
\end{align}

Unlike the bi-infinite case, Equation~\eqref{eq:background:koopman:zdot:finite} is a well-defined finite linear homogeneous initial value problem with $n(2N+1)$ states. Its solution at time $t$ is given by the matrix exponential 
\begin{align}
	\vz(t) = \exp(\vH  (t - t_0)) \, \vW \vy_0 \;.
\end{align} 
In the absence of errors due to truncation, the evolution of $\vz(t)$ generated by the truncated infinitesimal Koopman generator would coincide with the evolution of the $2N+1$ innermost entries of $\vg$ under the Koopman operator, according to Equation~\eqref{eq:background:KOP:Phi}. In practice, one hopes that this relationship is approximately preserved for sufficiently large $N$~\cite{Budisic2012}. This idea is visualized in the left half of Figure~\ref{fig:background:flowchart}: Evolving the state using the LTP system of Equation~\eqref{eq:background:ode} and applying the observable function afterwards should yield approximately the same results as applying the observable function to the initial state and evolving the result under the truncated lifted LTI dynamics of Equation~\eqref{eq:background:koopman:zdot:finite}. Previous numerical studies have shown that, indeed, at least the difference between the centermost blocks of $\vz(t)$ and $\vg(t, \vy(t))$ does tend to zero as the truncation order~$N$ increases~\cite{Bayer2023,Bayer2024,McGurk2025}. 

In the right half of Figure~\ref{fig:background:flowchart}, both the vectors $\vg(t, \vy(t))$ in the upper row and $\vz(t)$ in the lower row are pre-multiplied by $\ex^{\ic \vD \omega t}$. This operation does not alter the truncation error, but it removes the exponential scaling factor from each block of $\vg$, cf.\ Equation~\eqref{eq:background:KOP:Phi}. Each block of the resulting matrix contains $\vy(t)$ (upper row) or an approximation of $\vy(t)$ (lower row). Hereby, all the individual approximations need not be equal and may have differing truncation errors. 

\begin{figure}
	\centering
	\includegraphics[width=\textwidth]{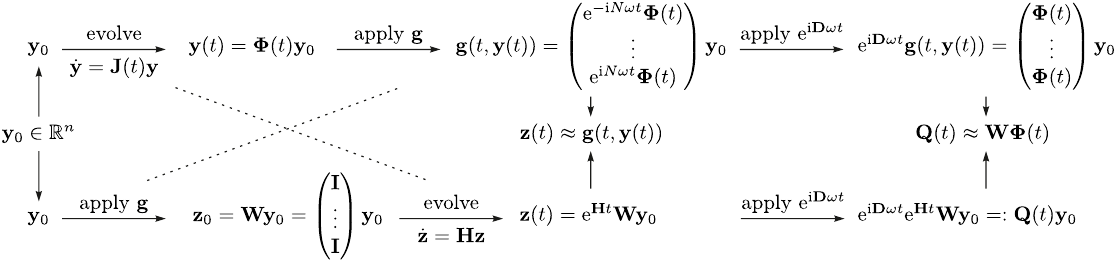}
	\caption{Schematic of the Koopman-Hill projection method. The top row illustrates the evolution of $\vy_0$ under the original LTP dynamics and subsequent application of $\vg$ to the evolved state, while the bottom row shows the evolution using the truncated lifted LTI dynamics after application of $\vg$ to the initial condition. If the truncation error is small, the bottom row approximates the top row.}
	\label{fig:background:flowchart}
\end{figure}

The final step is to generalize from the evolution of a single initial condition $\vy_0$ to the fundamental solution matrix. This generalization is straightforward. The fundamental solution matrix $\vPh(t)$ is defined as the solution of Equation~\eqref{eq:background:ode:matrix} with the initial condition $\vy(0) = \vI$. Therefore, $\vy_0$ can be replaced by $\vI$, or simply omitted from all equations in Figure~\ref{fig:background:flowchart}. The resulting matrix 
\begin{align}\label{eq:proof:KoopHill:Q_defin}
	\vQ(t) := \begin{pmatrix} \vQ_{-N}(t)\\ \vdots\\ \vQ_{N}(t) \end{pmatrix} = \ex^{\ic \omega \vD t}  \ex^{\vH t} \vW 
\end{align}
consists of $2N+1$ blocks $\vQ_j(t)$ for $j = -N, \dots, N$, each of size $n \times n$. When the truncation error is small and the top and bottom rows of Figure~\ref{fig:background:flowchart} approximate each other, $\vQ(t)$ is an approximation of the centermost blocks of $\vQ_{\infty}$, defined in Equation~\eqref{eq:background:KOP:Phi}. Hence, each $n \times n$ block of $\vQ(t)$ provides an approximation of the fundamental solution matrix $\vPh(t)$ with yet-undetermined approximation error. In particular, the central block $\vQ_0(t)$ can be extracted by matrix multiplication, yielding the following approximation for the monodromy matrix:
\begin{subequations}\label{eq:background:KoopHill}
	\begin{align}
		\vPh_T &\approx \vC \, \ex^{\vH T} \, \vW \\
		\vC & = \begin{pmatrix}\vzero & \dots & \vzero & \vI & \vzero & \dots & \vzero\end{pmatrix}\\
		\vW & = \begin{pmatrix}\vI & \dots & \vI \end{pmatrix}\T \;.
	\end{align}
\end{subequations}
This formulation was introduced in~\cite{Bayer2023} as the Koopman-Hill projection method. We emphasize that, in prior work, only numerical evidence was presented for the accuracy of this method and the truncation error has not been theoretically quantified. The present work addresses this gap by providing a closed-form bound for the truncation error for a certain class of linear time-periodic systems.

\subsection{Multi-index notation}\label{sec:background:multiindex}
For the sake of brevity, we introduce here two different types of index notation, which are used throughout this work. 
The natural numbers $\Nspace$ always include zero. 
We 
denote $1$-norms by $\abs{.}$ (with one vertical bar), even for tuples with more than one entry. 
By $\norm{.}$ (with two vertical bars) we denote an arbitrary vector $p$-norm and its induced matrix norm. 
\begin{defin}[Multi-index]\label{def:multiindex}
	A tuple $\val \in \Nspace^m$ of $m$ nonnegative integers is called \emph{multi-index (in the classical sense)}. The multi-index has a 1-norm
	\begin{align}
		\abs{\val} := \sum_{k = 1}^m \alpha_k 
	\end{align}
	given by the sum of its entries. For a vector $\vx \in \Cspace^m$, we denote by
	\begin{align}\label{eq:background:multiindex}
		\vx^{\val} := \prod_{k = 1}^m x_k^{\alpha_k}
	\end{align}
	monomials in $\vx$ of degree $\abs{\val}$. We define $0^0 := 1$ to avoid restricting the values that $\vx$ may admit.
\end{defin}
\begin{exmp}
	The multi-index $\val = [2, 0, 1, 0]$ has norm $\abs{\val}$ = $3$ and $\vx^{\val} = x_1^2 x_3$ for all $\vx = [x_1, x_2, x_3, x_4] \in \Cspace^4$. 
\end{exmp}

\begin{defin}[Products of Fourier coefficient matrices]\label{def:intindex}
	A tuple $\vp \in \Zspace^m$ of $m$ integers is called an \emph{integer index tuple}. It has a 1-norm
	\begin{align}
		\abs{\vp} := \sum_{k = 1}^m \abs{p_k} \;,
	\end{align}
	which is generally not equal to the sum of its entries. Given a bi-infinite sequence of Fourier coefficient matrices $\left( \vJ_k\right)_{k \in \Zspace}$, any integer index tuple $\vp \in \Zspace^m$ uniquely describes an (ordered) \emph{product of $m$ Fourier coefficient matrices}
	\begin{align}\label{eq:background:def_Jp}
		\cJ_\vp := \prod_{k = 1}^m \vJ_{p_k} = \vJ_{p_1} \vJ_{p_2} \dots \vJ_{p_m} \;.
	\end{align}
	The ordering of the entries of $\vp$ is important as the matrices $\vJ_k$ generally do not commute. 
\end{defin}
\begin{exmp}
	Consider as an example the integer index tuple $\vp = [-3, 0, 1, 1]$, which defines uniquely the product of Fourier coefficients
	\begin{align*}
		\cJ_{[-3, 0, 1, 1]} = \vJ_{-3} \,\vJ_0 \,\vJ_1\, \vJ_1 = \vJ_{-3} \,\vJ_0 \,\vJ_1^2 \;.
	\end{align*}
	As $\Zspace^m$ is a subset of $\Cspace^m$, integer index tuples can be exponentiated by multi-indices using Equation~\eqref{eq:background:multiindex}. With $\val = [2,0,1,0]$ and $\vp = [-3, 0, 1, 1]$ as before, we obtain
	\begin{align*}
		\vp^{\val} = [-3, 0, 1, 1]^{[2, 0, 1, 0]} = (-3)^2 \cdot 0^0 \cdot 1^1 \cdot 1^0 = 9
	\end{align*}
	and even
	\begin{align*}
		\vp^{\val} \cJ_{\vp} = 9 \, \vJ_{-3} \, \vJ_{0} \, \vJ_{1}^2 \;.
	\end{align*}
\end{exmp}
\begin{exmp}
	For an integer index tuple $\vp = [p_1, \dots, p_m] \in \Zspace^m$ and a multi-index $\val = [\alpha_1, \dots, \alpha_m] \in \Nspace^m$, a specific multi-index expression that will become important later is
	\begin{align}
		[p_1, p_1 + p_2, \dots, p_1 + \dots + p_m]^{\val} = \prod_{k = 1}^{m} \left(\sum_{l = 1}^k p_l \right)^{\alpha_k} = p_1^{\alpha_1} (p_1 + p_2)^{\alpha_2} \dots (p_1 + \dots + p_m)^{\alpha_m}\;,
	\end{align}
	consisting of a product of $m$ factors, where the $k$-th factor is given by the sum of the first $k$ entries of the integer index tuple, exponentiated by $\alpha_k$. 
\end{exmp}

\section{Main result: Error bound}\label{sec:overview}
Section~\ref{sec:background:KoopLift} motivated that the matrix $\vQ(t)$, defined in Equation~\eqref{eq:proof:KoopHill:Q_defin} as $\vQ(t) := \ex^{\ic \omega \vD t}\ex^{\vH t}\vW$ with $\vH$ being the truncated Hill matrix, provides approximations of the fundamental solution matrix $\vPh(t)$. More specifically, the matrix $\vQ(t)$ consists of $2N+1$ blocks $\vQ_j(t)$, for $j = -N, \dots, N$. Each block~$\vQ_j(t)$ of size $n \times n$ is expected to approximate the fundamental solution matrix $\vPh(t)$.
However, the motivation in Section~\ref{sec:background:KoopLift}, based on the Koopman framework, does not provide any insight into the magnitude of the error incurred by truncation. In principle, this error could become very large and render the method useless. Thus, the Koopman-based motivation for Equation~\eqref{eq:background:KoopHill} is not sufficient to prove that $\vQ_{j}(t) \approx \vPh(t)$ is indeed the case.

The central contribution of the present work is to strengthen the Koopman-Hill projection method, i.e., justify that Equation~\eqref{eq:background:KoopHill} approximates the fundamental solution matrix, by providing a closed-form bound for the difference between the central block $\vQ_0(t)$ of Equation~\eqref{eq:proof:KoopHill:Q_defin} and the fundamental solution matrix $\vPh(t)$. The approach taken here does not rely on methodology from the Koopman framework. Rather, the idea is to find exact series expressions for the considered matrices $\vQ_0(t)$ and $\vPh(t)$, compare them term by term, and bound the difference. For both the true fundamental solution matrix and the Koopman-Hill approximation, it is natural to consider Taylor series expansions around the initial time $t_0 = 0$.  

For the Taylor series of $\vPh(t)$, the values of the true fundamental solution and its first derivative at $t_0 = 0$ are immediately known from the matrix initial value problem of Equation~\eqref{eq:background:ode:matrix}. Further derivatives at $0$ can be found inductively by the product rule as polynomials of derivatives of $\vJ(t)$, evaluated at zero. With $\vJ(t)$ and all its derivatives expressed as Fourier series, a combined Taylor-Fourier series expression for $\vPh(t)$, characterized by products of the Fourier coefficients of $\vJ(t)$, can be established. This constructive process is described in more detail in~\ref{sec:app:construction}. The absolute convergence of the resulting series expression can be ensured by the following assumption, which in some sense characterizes the dynamical system under consideration.
\begin{assu} \label{assu:b}
	There are constants $a, b > 0$ and a matrix $p$-norm $\norm{\cdot}$ such that the Fourier coefficient matrices of $\vJ(t)$ in Equation~\eqref{eq:background:ode} decay exponentially with
	\begin{align}
		\norm{\vJ_k} \leq a \ex^{-b \abs{k}}&& \text{for all~} k \in \Zspace \;. \label{eq:proof:assu_b}
	\end{align}		
\end{assu}
By the Paley-Wiener theorem, Assumption~\ref{assu:b} holds if $\vJ(t)$ is analytic~\cite[Lemma 5.6]{Broer2011}. 
In the context of stability of periodic solutions of nonlinear dynamics, analyticity of a nonlinear differential equation in all its arguments is sufficient to guarantee analyticity of its periodic solution and thus analyticity of the linearization~$\vJ(t)$~\cite[Thm.~4.1]{Teschl2012}. 
In the special case where the Fourier coefficients of $\vJ(t)$ have finite support, Assumption~\ref{assu:b} is satisfied for any $b > 0$ with $a = \max_k \norm{\vJ_k} \ex^{b \abs{k}}$.

Our first theorem establishes the Taylor-Fourier series for the fundamental solution matrix $\vPh(t)$. 

	\begin{theorem}[Series formulation of~$\vPh(t)$]\label{thm:proof:Phi_series:xi}
	If the Fourier coefficient matrices of $\vJ(t)$ in Equation~\eqref{eq:background:ode} fulfill Assumption~\ref{assu:b}, then the fundamental solution matrix is given by the absolutely convergent series
	\begin{align}
		\vPh(t) = \vI + \sum_{m = 1}^\infty \sum_{\vp \in \Zspace^{m}} \xi_{\vp}(t) \cJ_{\vp} \label{eq:proof:Phi_series:xi}
	\end{align}
	with the scalar, complex-valued factor
	\begin{align}\label{eq:lem:proof:xi_p:xi}
		\xi_{\vp}(t) :=  \sum_{\val \in  \Nspace^m} \frac{t^{m + \abs{\val}}}{(m + \abs{\val})!} (\ic \omega)^{\abs{\val}}\left[ p_1, p_1 + p_2, \dots, p_1 + \dots + p_m\right]^{\val} 
	\end{align}
	that is independent of the specific dynamical system and bounded by  
	\begin{align}\label{eq:lem:xi_p:bound}
		\abs{\xi_{\vp}(t)} \leq \frac{\abs{t}^m}{m!} \;.
	\end{align}
	
\end{theorem}
\noindent Theorem~\ref{thm:proof:Phi_series:xi} is proven in Section~\ref{sec:series:proof}. The fundamental solution matrix is expressed here in terms of ordered monomials $\cJ_{\vp}$ of Fourier coefficients of the system matrix $\vJ(t)$, with each such monomial multiplied by a scalar factor $\xi_{\vp}(t)$ that only depends on the time~$t$ and the degree~$\vp$ of the monomial, and not on the considered system. Figure~\ref{fig:proof:xi_p} presents the values of the scalar factor $\xi_{\vp}$ for three different values of $\vp$, exemplifying three central properties: $\xi_{\vp}$ is  complex-valued, of oscillatory nature, and polynomially bounded.  For different values of $\vp$, the values of $\xi_{\vp}$ can be of vastly differing magnitude. 

\begin{figure}[hbt]
	\centering
	\begin{subfigure}[t]{0.325\textwidth}
		\centering
		\includegraphics{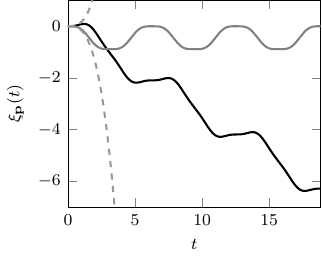}
		\caption{$\vp = [-1, -2, 3]$}
		\label{fig:proof:xi_p:1}
	\end{subfigure}
	\begin{subfigure}[t]{0.325\textwidth}
		\centering
		\includegraphics{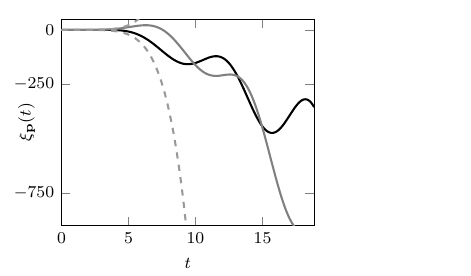}
		\caption{$\vp = [1, -1, 0, 1, -1, 1]$}
		\label{fig:proof:xi_p:2}
	\end{subfigure}
	\begin{subfigure}[t]{0.325\textwidth}
		\centering
		\includegraphics{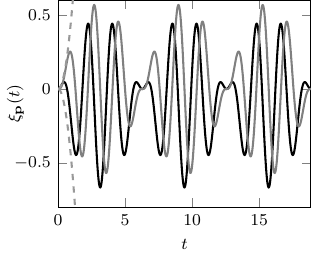}
		\caption{$\vp = [4, -1]$}
		\label{fig:proof:xi_p:3}
	\end{subfigure}
	\caption{Real part (black), imaginary part (gray), and polynomial upper bound (dashed) of the scalar factor $\xi_{\vp}(t)$ over three periods for three exemplary values of $\vp$. Data was computed using Lemma~\ref{lem:proof:xi_p:deriv} and numerical quadrature with trapezoidal rule.}
	\label{fig:proof:xi_p}
\end{figure}

The proof of Theorem~\ref{thm:proof:Phi_series:xi} in Section~\ref{sec:series:proof} relies on derivative properties of the scalar factor $\xi_{\vp}$ to show that the series in Equation~\eqref{eq:proof:Phi_series:xi} uniquely solves the matrix initial value problem~\eqref{eq:background:ode:matrix} and must thus be the fundamental solution matrix. However, this approach is not constructive and does not explain why the series is of its specific structure. For additional insight,~\ref{sec:app:construction} sketches an alternative, more illustrative derivation of Equations~\eqref{eq:proof:Phi_series:xi} and~\eqref{eq:lem:proof:xi_p:xi} based on the Taylor expansion of $\vPh(t)$ and the Fourier series of $\vJ(t)$.

Similarly, a series expression for the blocks of the matrix $\vQ(t)$, which are expected to approximate the fundamental solution matrix according to Figure~\ref{fig:background:flowchart} and Equation~\eqref{eq:background:KOP:Phi}, can be derived by expressing Equation~\eqref{eq:proof:KoopHill:Q_defin} as a Taylor series. Since the matrix $\vH$ consists of the Fourier coefficients of $\vJ(t)$ (cf.\ Equation~\eqref{eq:background:Hillmat}), the summands of the resulting Taylor series are again characterized by polynomials of these Fourier coefficients. 

\begin{theorem}[Series expression for the Koopman-Hill approximations $\vQ(t)$]\label{thm:proof:Q}
	For any index $j \in \left\{ -N, \dots, N\right\}$, the $j$-th $n\times n$-block of the matrix $\vQ(t)$ defined in Equation~\eqref{eq:proof:KoopHill:Q_defin} is given by the series
	\begin{align}\label{eq:proof:Q:series}
		\vQ_j(t) = \vI + \sum_{m = 1}^\infty \sum_{\vp \in \cP_j^{(m)}} \xi_{\vp}(t) \cJ_{\vp} \;,
	\end{align}
	where
	\begin{align}\label{eq:proof:Pj}
		\cP_j^{(m)} = \left\{ \vp \in \Zspace^m : \abs{j - \sum_{l = 1}^w p_l} \leq N \text{~for all~} w = 1, \dots, m\right\}
	\end{align} 
	are sets of eligible integer index tuples (i.e., tuples which appear in the series) and $\xi_{\vp}(t)$ is defined in Equa\-tion~\eqref{eq:lem:proof:xi_p:xi} of Theorem~\ref{thm:proof:Phi_series:xi}. The series~\eqref{eq:proof:Q:series} is absolutely convergent for all $t \in \Rspace$ if all Fourier coefficients $\left\{\vJ_k\right\}_{k = -2N, \dots, 2N}$ that occur in $\vH$ are bounded.
\end{theorem}
\noindent The proof of Theorem~\ref{thm:proof:Q} is carried out in Section~\ref{sec:series:KoopHill}. 
The series expression for the direct Koopman-Hill projection of Equation~\eqref{eq:background:KoopHill} follows immediately: 
\begin{cor}[Direct Koopman-Hill projection]\label{cor:KHP}
	The direct Koopman-Hill projection approximation as de\-fined in Equation~\eqref{eq:background:KoopHill}, $\vC \ex^{\vH t} \vW = \vQ_0(t)$, is given by the series
	\begin{align}\label{eq:proof:KoopHill:direct}
		\vC \ex^{\vH t} \vW= \vI + \sum_{m = 1}^\infty \sum_{\vp \in \cP_0^{(m)}} \xi_{\vp}(t) \cJ_{\vp} 
	\end{align}
	with the index parallelotopes
	\begin{align}\label{eq:proof:KoopHill:P}
		\cP_0^{(m)} = \left\{ \vp \in \Zspace^m : \abs{\sum_{l = 1}^w p_l} \leq N \text{~for all~} w = 1, \dots, m\right\} \;.
	\end{align} 
\end{cor}
\begin{proof}
	This is a direct consequence of Theorem~\ref{thm:proof:Q} for $j = 0$.
\end{proof}

Geometrically, the eligible sets $\cP_j^{(m)}$ form parallelotopes (generalized parallelograms) in $\Zspace^m$, centered at the point $(j, 0, \dots, 0)$. In essence, the true fundamental solution sums over all possible integer index tuples, while the Koopman-Hill approximation only includes those tuples that lie within the corresponding $m$-dimensional parallelotope. Figure~\ref{fig:proof:Pj} illustrates these parallelotopes for the case $m = 2$. 
\begin{figure}[hbtp]
	\centering
	\begin{subfigure}[t]{0.325\textwidth}
		\centering
		\includegraphics{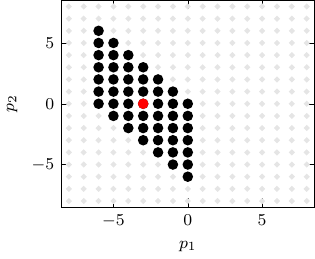}
		\caption{$\cP^{(2)}_{-3}$}
		\label{fig:proof:Pj:m4}
	\end{subfigure}
	\begin{subfigure}[t]{0.325\textwidth}
		\centering
		\includegraphics{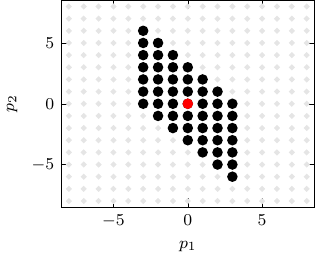}
		\caption{$\cP^{(2)}_{0}$}
		\label{fig:proof:Pj:0}
	\end{subfigure}
	\begin{subfigure}[t]{0.325\textwidth}
		\centering
		\includegraphics{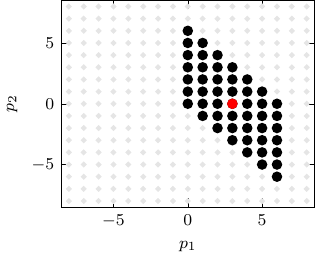}
		\caption{$\cP^{(2)}_{3}$}
		\label{fig:proof:Pj:4}
	\end{subfigure}
	\caption{Visualization of the sets $\cP_j^{(2)}$ with $N = 3$ for different values of $j$. Center element located at $(j, 0)$ is indicated by a red dot.}
	\label{fig:proof:Pj}
\end{figure}

Both series expressions of Equations~\eqref{eq:proof:Phi_series:xi} and~\eqref{eq:proof:KoopHill:direct} for the true fundamental solution matrix and its Koopman-Hill approximation are characterized by the same monomials $\cJ_{\vp}$ and identical scalar factors $\xi_{\vp}(t)$. The only difference is that the Koopman-Hill approximation of Equation~\eqref{eq:proof:KoopHill:direct} is a partial sum of the true fundamental solution matrix of Equation~\eqref{eq:proof:Phi_series:xi} due to the truncation effects, only summing over entries of the eligible parallelotopes $\cP_0^{(m)}$ instead of over the whole $\Zspace^m$. Consequently, the truncation error between both can be expressed as the series over all multi-indices that are \emph{not} included in the eligible parallelotopes, i.e.,
\begin{align}\label{eq:conv:E_def}
	\vE(t) = \vPh(t) - \vC \ex^{\vH t} \vW = \sum_{m = 1}^\infty \sum_{\vp \in \Zspace^m \setminus \cP_0^{(m)}} \xi_{\vp}(t) \cJ_{\vp} \;.
\end{align}
Our main result bounds Equation~\eqref{eq:conv:E_def} to provide a convergence guarantee for the Koopman-Hill ap\-proxi\-ma\-tion. 
\begin{theorem}[Error bound and convergence of the direct Koopman-Hill projection approximation] \label{thm:proof:error}
	Let As\-sump\-tion~\ref{assu:b} hold with $b > \ln 2$. For a fixed truncation order $N \in \Nspace$, the approximation error between the true fundamental solution matrix and the direct Koopman-Hill projection defined in Equation~\eqref{eq:background:KoopHill} is bounded by
	\begin{align}\label{eq:conv:bound}
		\norm{\vPh(t) - \vC \ex^{\vH t} \vW} \leq (2 \ex^{-b})^{N} \left( \ex^{\abs{4 a t}} - 1 \right)\;.
	\end{align}
	This error bound decays exponentially with $N$, providing a convergence guarantee of the Koopman-Hill pro\-jec\-tion. Satisfaction of a desired accuracy $\norm{\vPh(t) - \vC \ex^{\vH t} \vW} \leq E_{\mathrm{des}}$ is guaranteed if the truncation order fulfills
	\begin{align}\label{eq:conv:Nmin}
		N \geq N^* = \frac{\abs{4 a t} + \ln \left( 1 - \ex^{-\abs{4 a t}}\right) - \ln(E_{\mathrm{des}})}{b - \ln 2} \;.
	\end{align}
\end{theorem}

\begin{proof}
With Assumption~\ref{assu:b}, a product $\cJ_{\vp}$ of Fourier coefficients is bounded by 
\begin{align}\label{eq:assu}
	\norm{\cJ_{\vp}} \leq \norm{\vJ_{p_1}} \dots \norm{\vJ_{p_m}} \leq a^m \ex^{-b \abs{\vp}} \;.
\end{align}
This bound is coarse when the Fourier coefficient matrices are not diagonally dominant, but it reduces the involved matrix norms on $\cJ_{\vp}$ to a $1$-norm in the integer index tuple $\vp$. Using this bound and Equation~\eqref{eq:lem:xi_p:bound} on each summand of the approximation error given by Equation~\eqref{eq:conv:E_def} individually yields
\begin{align}\label{eq:proof:bound:summands}
	\norm{\vE(t)}  
	\leq \sum_{m = 1}^\infty \sum_{\vp \in \Zspace^{m} \setminus \cP_0^{(m)}} \norm{\xi_{\vp}(t) \cJ_{\vp} }
	\leq \sum_{m = 1}^\infty \sum_{\vp \in \Zspace^{m} \setminus \cP_0^{(m)}}  \frac{\abs{a t}^m}{m!} \ex^{- b \abs{\vp}} \;.
\end{align}
Equation~\eqref{eq:proof:bound:summands} already contains two majorly conservative simplifications. Firstly, bounding every summand individually cancels any interaction between terms which could potentially annihilate each other. Secondly, every scalar factor $\xi_{\vp}(t)$ is bounded by its polynomial upper bound of Equation~\eqref{eq:lem:xi_p:bound}. From Figure~\ref{fig:proof:xi_p} it is obvious that this bound is not particularly sharp as $t$ increases. In fact, for the majority of values of $\vp$, $\xi_{\vp}$ is even periodic and bounded by a constant. This is examined in more detail in~\ref{sec:app:periodicity}. However, using any more sophisticated bounds for $\xi_{\vp}$ or $\norm{\cJ_{\vp}}$ complicates Equation~\eqref{eq:proof:bound:summands} so much that it is unclear how to proceed to a simple closed-form expression. 

The benefit of Equation~\eqref{eq:proof:bound:summands} is that only the norm $\abs{\vp}$ of each integer index tuple appears in the summand.  The sum over all integer index tuples can thus be replaced by a sum over their norm. This reformulation yields the triple sum
\begin{align}\label{eq:conv:counting}
	\norm{\vE(t)}  \leq \sum_{m = 1}^\infty \frac{\abs{a t}^m}{m!} \sum_{M = 0}^\infty \ex^{- b M}  \sum_{\substack{\vp \in \Zspace^{m} \setminus \cP^{(m)}_0 \\ \abs{\vp} = M}} 1    \;,
\end{align}
where the rightmost sum counts the number of integer $m$-tuples of a certain norm $M$ that lie outside the parallelotope $\cP^{(m)}_0$. This sum has only finitely many summands: The total number of integer index tuples of norm $M$ is bounded by Lemma~\ref{lem:proof:starsbars}\ref{lem:proof:starsbars:item:p} of~\ref{sec:app:combi}, and we only sum over a subset of them.

\begin{figure}[hbtp]
	\centering
	\begin{subfigure}[t]{0.475\textwidth}
		\centering
		\includegraphics{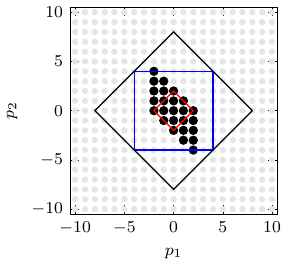}
		\caption{$N = 2$: rhombus $\abs{\vp} \!=\! 2\subset \cP_0^{(2)}$.}
		\label{fig:conv:norms:2}
	\end{subfigure}
	\hfill
	\begin{subfigure}[t]{0.475\textwidth}
		\centering
		\includegraphics{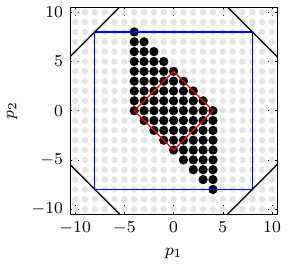}
		\caption{$N = 4$: rhombus $\abs{\vp} = 4\subset \cP_0^{(2)}$.}
		\label{fig:conv:norms:4}
	\end{subfigure}
	\caption{Visualization of integer index set $\cP_0^{(2)}$ (large dots). Inner rhombus (red), outer square (blue) and outer rhombus (black) are indicated.}
	\label{fig:conv:norms}
\end{figure} 

In Figure~\ref{fig:conv:norms}, the integer index tuples are visualized for $m = 2$ and and three different values of $N$. Visually, the space $\Zspace^2$ in the figure can be separated into three regions, depending on the norm $\abs{\vp}$. Tuples within the $N$-rhombus shown in red in Figure~\ref{fig:conv:norms}, i.\,e., $\vp$ with small norm $\abs{\vp} \leq N$, are always elements of $\cP_0^{(2)}$. In contrast, tuples $\vp$ are never elements of $\cP_0^{(2)}$ when $\max\{\abs{p_1}, \abs{p_2}\}>2N$, i.e., when $\vp$ lies outside the $2N$-square shown in blue in Figure~\ref{fig:conv:norms}. In the intermediate region outside the $N$-rhombus but within the $2N$-square, some tuples are elements of $\cP_0^{(2)}$ and the others are not. It can be easily verified from the definition of $\cP_0^{(m)}$ in Equation~\eqref{eq:proof:Pj} that these three observations generalize to arbitrary~$m$: All tuples $\vp \in \Zspace^m$ inside the $N$-rhombus with norm $\abs{\vp} \leq N$ are elements of $\cP_0^{(m)}$, while no tuple outside the $2N$-hypercube with an entry $\abs{p_k} > 2N$ can be an element of $\cP_0^{(m)}$. In the intermediate region $N < \abs{\vp}$ and $\abs{p_k} \leq 2N$ for all $k$, some tuples are elements of $\cP_0^{(m)}$ and the others are not.

This observation simplifies the rightmost summation in Equation~\eqref{eq:conv:counting} depending on $M$. For $M \leq N$, in the inner region, all possible tuples $\vp$ are elements of the eligible set $\cP_0^{(m)}$ and the set over which to sum in Equation~\eqref{eq:conv:counting} is empty. Hence, the summation over $M$ in Equation~\eqref{eq:conv:counting} can equivalently be initialized at $M = N+1$ instead of $M = 0$. For $\abs{\vp} > 2mN$, shown in black in Figure~\ref{fig:conv:norms}, at least one entry must have $\abs{p_k} > 2N$ and no such tuples can be elements of $\cP_0^{(m)}$. Hence, for $M > 2mN$, the rightmost sum of Equation~\eqref{eq:conv:counting} must count all integer index tuples with norm $M$. Using Lemma~\ref{lem:proof:starsbars}\ref{lem:proof:starsbars:item:p} from~\ref{sec:app:combi}, this number is bounded by $2^m \binom{M + m - 1}{m - 1}$. In the intermediate region $N < M \leq 2mN$, counting is not so easy. As another very conservative simplification step, we upper-bound the rightmost sum in Equation~\eqref{eq:conv:counting} by the binomial coefficient $2^m \binom{M + m - 1}{m - 1}$ also in this region. Essentially, we count integer index tuples in the intermediate region $N < M \leq 2mN$, irrespective of whether they are elements of $\cP_0^{(m)}$ or not. These arguments simplify Equation~\eqref{eq:conv:counting} to
\begin{align}\label{eq:conv:remainder}
	\norm{\vE(t)} \leq \sum_{m = 1}^\infty \frac{\abs{a t}^m}{m!} \sum_{M = N+1}^\infty \ex^{- b M} \sum_{\substack{\vp \in \Zspace^{m} \\ \abs{\vp} = M}}  1 \leq \sum_{m = 1}^\infty  \frac{\abs{2at}^m}{m! } \sum_{M = N + 1}^\infty \binom{M+m-1}{m-1} \ex^{-bM} \;.
\end{align}
The coarseness introduced by this step increases as $N$ and $m$ increase and the intermediate region becomes larger. Being able to resolve the number of considered tuples more accurately in the intermediate region would likely improve the convergence rate of the error bound, possibly removing the factor $2^m$ which is, in part, responsible for the condition $b > \ln 2$ in Theorem~\ref{thm:proof:error}. However, this counting problem is non-trivial and beyond the scope of this paper.

The rightmost series in Equation~\eqref{eq:conv:remainder} is the remainder $R_N(x)$ of the Taylor series of $(1-x)^{-(k+1)}$, treated in Lemma~\ref{lem:series_eval} of~\ref{sec:app:Taylor}, with $x = \ex^{-b} < 1$ and $k = m-1$. Using the expression~\eqref{eq:lem:series_eval:remainder} for the remainder and keeping~$x$ for the sake of brevity, we obtain
\begin{align}\label{eq:conv:binom}
	\norm{\vE(t)} \leq x^{N} \sum_{m = 1}^\infty  \frac{\abs{2at}^m}{m!}  \sum_{w = 1}^{m} \binom{N + m}{N + w} \left( \frac{x}{1-x}\right)^{w}  \;.
\end{align}
Due to the infinite series over $m$, Equation~\eqref{eq:conv:binom} can still not be evaluated with ease and further bounding steps are necessary. The assumption $b > \ln 2$ guarantees that $x < 0.5$, and thus $\frac{x}{1-x} < 1$ can be neglected. This simplification does not introduce very much additional slack as $b$ will often be chosen close to $\ln 2$ in practice, see Section~\ref{sec:examples}. However, this still does not solve the issue of the infinite series. Summing the binomial coefficients with lower argument $0$ to $N+m$ instead of from $N+1$ to $N+m$ yields the bound 
\begin{align}\label{eq:proof:conv:bound}
	\norm{\vE(t)} &\leq x^{N} \sum_{m = 1}^\infty  \frac{\abs{2at}^m}{m!} \sum_{w = 0}^{m + N} \binom{m + N}{w} \nonumber \\
	&= x^{N} \sum_{m = 1}^\infty  \frac{\abs{2at}^m}{m!} 2^{m + N}\nonumber \\
	&= (2x)^{N} \sum_{m = 1}^\infty  \frac{\abs{4at}^m}{m!} \nonumber\\
	&= (2 \ex^{-b})^{N} \left( \ex^{\abs{4 a t}} - 1  \right)\;.
\end{align}
This is the desired error bound of Equation~\eqref{eq:conv:bound}. The expression~\eqref{eq:conv:Nmin} for the minimum truncation order~$N^*$ which guarantees an error smaller than $E_{\mathrm{des}}$ follows immediately from Equa\-tion~\eqref{eq:proof:conv:bound} by isolating $N^*$ in $E_{\mathrm{des}} = (2 \ex^{-b})^{N^*} \left( \ex^{\abs{4 a t}} -1 \right)  $.
\end{proof}

The bounding steps employed in the above proof to go from the series expression of Equation~\eqref{eq:conv:E_def} to the closed-form bound of Equation~\eqref{eq:proof:conv:bound} are conservative in multiple individual instances. On the one hand, every summand of Equation~\eqref{eq:conv:E_def} is bounded individually and coarsely using the exponential decay condition. This leads to the~$\ex^{a t}$ term in Equation~\eqref{eq:proof:conv:bound}. Numerical experiments in Section~\ref{sec:examples} indicate that this term is not very sharp and might be an artifact of the specific proof strategy employed here, as the actual error does not seem to grow with $t$ after a short initial phase. A more sophisticated bound could potentially leverage the periodicity of $\xi_{\vp}$ for most values of $\vp$, see also~\ref{sec:app:periodicity}. 

On the other hand, counting the relevant integer tuples in the rightmost series in Equation~\eqref{eq:conv:binom} accurately is a challenging combinatorial problem. The very conservative upper bound chosen here leads to the binomial coefficient expression of Equation~\eqref{eq:conv:remainder}. This expression is beneficial because it can be identified with a certain Taylor series, but the subsequent handling of binomial coefficients leads to the technical requirement $b > \ln 2$. It is unclear how the remaining summations could be eliminated if a sharper bound for the number of considered integer indices were employed. 

The condition $b > \ln 2$ for the applicability of the closed-form bound of Equation~\eqref{eq:conv:bound} requires the Fourier coefficients of~$\vJ(t)$ to decay exponentially at a rate of at least $\ln 2$. This is a stricter requirement than the arbitrary exponential decay needed for absolute convergence of $\vPh(t)$, according to Theorem~\ref{thm:proof:Phi_series:xi}. In practice, whether the condition $b > \ln 2$ is satisfied can be checked numerically, as demonstrated in Section~\ref{sec:examples} with three numerical examples. Despite being not particularly sharp, the bound is the first of its kind and thus integral in providing theoretical backing for the Koopman-Hill projection method, which the Koopman framework cannot provide.

\section{Improved convergence rate using subharmonic projection}\label{sec:subharmonics}
According to Theorem~\ref{thm:proof:Q}, the matrix $\vQ(t)$ contains $2N+1$ blocks, each offering an individual approximation of the fundamental solution matrix~$\vPh(t)$. Theorem~\ref{thm:proof:error} gives an explicit error bound for the centermost block~$\vQ_0$. The other blocks differ from $\vQ_0$ only by a shift in the center of their respective index parallelotopes~$\cP_j^{(m)}$ (see Figure~\ref{fig:proof:Pj}). In principle, the proof of Theorem~\ref{thm:proof:error} can be adapted to these other blocks, but counting the eligible indices in Equation~\eqref{eq:conv:counting} becomes even more complex when the parallelotope is not centered at the origin. Specifically, the maximum norm $M$ for which all integer indices are inside $\cP_j^{(m)}$ decreases as $\abs{j}$ increases, so the summation in Equation~\eqref{eq:conv:remainder} must start at $N + 1-\abs{j}$ instead of $N +1$. As a result, the error bound for each block decays with $N-\abs{j}$ rather than $N$, making the centermost block $\vQ_0$ the one with the smallest error bound. 

However, previous works~\cite{Bayer2023,Bayer2024} have observed that better accuracy can be achieved in some cases by approximating the fundamental solution matrix by a linear combination over all blocks of $\vQ(t)$. 
In the present section, we formalize this approach, which we call the \emph{subharmonic} approach, making it applicable to arbitrary dynamics fulfilling Assumption~\ref{assu:b}. We show that this approach is not only beneficial numerically, but also comes with a lower error bound compared to the direct approach.

\subsection{Subharmonic formulation and series expression}\label{sec:subharmonics:formulation}
Consider a fixed truncation order $N \in \Nspace$ and an arbitrary $T$-periodic system matrix $\vJ(t)$ fulfilling As\-sump\-tion~\ref{assu:b} with $b > \ln 2$. We make no further assumption on the original Fourier series given in Equation~\eqref{eq:background:fourier} with $\omega = \frac{2\pi}{T}$. 
In particular, the Fourier coefficient matrices $\vJ_k$ may be nonzero for even and odd values of $k \in \Zspace$. 
As the system matrix $\vJ(t)$ is $T$-periodic, it can also be considered  $\tilde{T} = 2T$-periodic. The corresponding Fourier series with the halved base frequency $\tilde{\omega} = \frac{2\pi}{2T} = \frac{\omega}{2}$ is 
\begin{align}\label{eq:subh:fourierseries}
	\vJ(t) = \sum_{\tilde{k} = -\infty}^\infty \tilde{\vJ}_{\tilde{k}} \ex^{\ic \tilde{k} \frac{\omega}{2} t} \;.
\end{align}
Here and below, we denote quantities associated with the doubled period $\tilde{T}$ and the Fourier series in Equa\-tion~\eqref{eq:subh:fourierseries} with a tilde. Quantities without the tilde refer to the original Fourier series in Equation~\eqref{eq:background:fourier}. 
Comparison of the Fourier series expressions in Equations~\eqref{eq:background:fourier} and~\eqref{eq:subh:fourierseries} reveals that every odd subharmonic Fourier coefficient vanishes:
\begin{align}\label{eq:subh:J_subh}
	\tilde{\vJ}_{\tilde{k}} = \begin{cases}
		\vJ_{\frac{\tilde{k}}{2}} & \text{if~}\tilde{k} \text{~even}\\
		\vzero & \text{if~} \tilde{k} \text{~odd} \;.
	\end{cases}
\end{align}
The Fourier coefficients of Equation~\eqref{eq:subh:J_subh} in the Fourier series of Equation~\eqref{eq:subh:fourierseries} fulfill Assumption~\ref{assu:b} with $\tilde{a} = a$ and $\tilde{b} = \frac{b}{2}$. The subharmonic Hill matrix $\tilde{\vH}$ of truncation order $\tilde{N}$ can be constructed from the subharmonic Fourier series of Equation~\eqref{eq:subh:fourierseries} in the usual way.  According to Equation~\eqref{eq:background:Hillmat}, the subharmonic Fourier coefficients $\tilde{\vJ}_{\tilde{k}}$ with ${\tilde{k} \in \left\{-2\tilde{N}, \dots, 2\tilde{N} \right\}}$ are placed on the cor\-res\-pon\-ding block diagonals and purely imaginary terms of the form $\ic \tilde{k} \tilde{\omega}$ are added in the main diagonal. 

The (classical) Hill matrix $\vH$ of truncation order $N$ associated with the original Fourier series in Equa\-tion~\eqref{eq:background:fourier} and the subharmonic Hill matrix $\tilde{\vH}$ of truncation order $\tilde{N} = 2N$ associated with the subharmonic Fourier series in Equation~\eqref{eq:subh:fourierseries} both contain the same set of Fourier coefficient matrices $\left\{ \vJ_{-2N}, \dots, \vJ_{2N}\right\}$ and are thus, in a sense, related. The difference between $\vH$ and $\tilde{\vH}$, however, is that the subharmonic Hill matrix is of twice the size of the original Hill matrix and the blocks of nontrivial Fourier coefficient matrices are separated by blocks of zeros ($\tilde{\vJ}_{\tilde{k}}$ for odd $\tilde{k}$). Explicitly for $N = 1$, this yields
\begin{align}
	\vH &= 
	\left(\begin{array}{lll}
		  \graybox[55]{\vJ_0}{\vJ_0} + \ic \omega \vI 
		& \graybox[30]{\vJ_{-1}}{\vJ_0} 
		& \graybox[10]{\vJ_{-2}}{\vJ_0} 
		\\
		\graybox[30]{\vJ_1}{\vJ_0} 
		& \graybox[55]{\vJ_0}{\vJ_0} \phantom{+ \dashedbox{\ic \omega \vI}}
		& \graybox[30]{\vJ_{-1}}{\vJ_0} 
		\\
		\graybox[10]{\vJ_2}{\vJ_0} 
		& \graybox[30]{\vJ_1}{\vJ_0} 
		& \graybox[55]{\vJ_0}{\vJ_0} - \dashedbox{\ic \omega \vI}
	\end{array} \right)  \label{eq:subh:H:N1}
	\\
	\tilde{\vH} &= 
	\left(\begin{array}{lllll}
		  \graybox[55]{\vJ_0}{\vJ_0} + \ic \omega \vI 
		& \graybox[0]{\vzero}{\vJ_0} 
		& \graybox[30]{\vJ_{-1}}{\vJ_0} 
		& \graybox[0]{\vzero}{\vJ_0}
		& \graybox[10]{\vJ_{-2}}{\vJ_0} 
		\\
		  \graybox[0]{\vzero}{\vJ_0} 
		& \graybox[55]{\vJ_0}{\vJ_0} + \dashedbox{\frac{1}{2}\ic \omega \vI} 
		& \graybox[0]{\vzero}{\vJ_0} 
		& \graybox[30]{\vJ_{-1}}{\vJ_0} 
		& \graybox[0]{\vzero}{\vJ_0} 
		\\
		  \graybox[30]{\vJ_1}{\vJ_0}
		& \graybox[0]{\vzero}{\vJ_0}
		& \graybox[55]{\vJ_0}{\vJ_0} \phantom{+ \dashedbox{\ic \omega \vI}}
		& \graybox[0]{\vzero}{\vJ_0} 
		& \graybox[30]{\vJ_{-1}}{\vJ_0} 
		\\
		  \graybox[0]{\vzero}{\vJ_0} 
		& \graybox[30]{\vJ_1}{\vJ_0} 
		& \graybox[0]{\vzero}{\vJ_0} 
		& \graybox[55]{\vJ_0}{\vJ_0} - \dashedbox{\frac{1}{2}\ic \omega \vI} 
		& \graybox[0]{\vzero}{\vJ_0} 
		\\
		  \graybox[10]{\vJ_2}{\vJ_0} 
		& \graybox[0]{\vzero}{\vJ_0} 
		& \graybox[30]{\vJ_1}{\vJ_0} 
		& \graybox[0]{\vzero}{\vJ_0} 
		& \graybox[55]{\vJ_0}{\vJ_0} - \dashedbox{2 \ic \omega \vI}
	\end{array} \right) \;. \label{eq:subh:Hsub:N1}
\end{align}

Note that both Equations~\eqref{eq:subh:H:N1} and~\eqref{eq:subh:Hsub:N1} are expressed in terms of the \emph{original} Fourier coefficients pertaining to the Fourier series of Equation~\eqref{eq:background:fourier}. In the even row blocks of $\tilde{\vH}$ (counted starting from the centermost row block with index $0$), only the even column blocks are nonzero. If all odd row and column blocks are discarded, exactly the original Hill matrix~$\vH$ with $2N+1 \times 2N+1$ blocks remains. Similarly, only the odd column blocks are nonzero in the odd row blocks of $\tilde{\vH}$. If all even row and column blocks are discarded, what remains is a matrix $\hat{\vH}$ with $2N$ blocks. Explicitly for $N = 1$, the matrix $\hat{\vH}$ is
\begin{align}
	\hat{\vH} &= 
	\left(\begin{array}{llll}
		  \graybox[55]{\vJ_0}{\vJ_0} + \dashedbox{\frac{1}{2}\ic \omega \vI} 
		& \graybox[30]{\vJ_{-1}}{\vJ_0} 
		\\
		  \graybox[30]{\vJ_1}{\vJ_0}
		& \graybox[55]{\vJ_0}{\vJ_0} - \dashedbox{\frac{1}{2}\ic \omega \vI} 
	\end{array} \right) \;.
\end{align} For all values of $N$, the matrix $\hat{\vH}$ can be constructed from $\vH$ by discarding the last block row and block column and subtracting $0.5 \ic \omega$ from the diagonal. 

By also defining the diagonal matrix $\tilde{\vD} = \diag\left\{-N, -N + \frac{1}{2}, N-1, \dots, N - \frac{1}{2}, N\right\} \otimes \vI$ of the corresponding subharmonic size, we establish the subharmonic equivalent to the matrix $\vQ(t)$ of Equation~\eqref{eq:proof:KoopHill:Q_defin}, 
\begin{align}\label{eq:subh:Q_def}
	\tilde{\vQ}(t) = \ex^{\ic \omega t \tilde{\vD}} \ex^{\tilde{\vH} t} \tilde{\vW} \;,
\end{align}
which has $2\tilde{N}+1 = 4N+1$ blocks of size $n \times n$ that approximate the fundamental solution matrix due to Theorem~\ref{thm:proof:Q}. Naively, with Theorem~\ref{thm:proof:error}, the approximation error of the centermost block of $\tilde{\vQ}$ is
\begin{align}\label{eq:subh:Qtilde}
	\abs{\tilde{\vQ}_0 - \vPh(t)} \leq (2 \ex ^{- \tilde{b}})^{\tilde{N}} \left( \ex^{\abs{4 a t}} - 1 \right) = (2 \ex ^{- b})^{N} \left( \ex^{\abs{4 a t}} - 1\right) \;,
\end{align}
so no improvements in accuracy are expected compared to the original, non-subharmonic formulation if only the centermost block is considered.  The following lemma provides a series expression for all blocks of $\tilde{\vQ}(t)$ specifically when $\tilde{\vH}$ is of subharmonic structure, showing in particular that $\vQ_0(t)$ and $\tilde{\vQ}_0(t)$ are indeed identical.

\begin{lemma}[Series expression for $\tilde{\vQ}(t)$ in subharmonic formulation]\label{thm:subh:Q}
	Let $\vJ(t)$ be a $T$-periodic system matrix, whose original (non-subharmonic) Fourier series as defined in Equation~\eqref{eq:background:fourier} has the Fourier coefficients $\vJ_{-2N}, \dots, \vJ_{2N}$ which may all be nonzero. Consider the matrix $\tilde{\vQ}(t)$ constructed via Equation~\eqref{eq:subh:Q_def} using the corresponding subharmonic Hill matrix of the form of Equation~\eqref{eq:subh:Hsub:N1}. For any $\tilde{j} \in \left\{ -2N, \dots, 2N\right\}$, the $\tilde{j}$-th $n\times n$-block of~$\tilde{\vQ}(t)$ is given by the series
	\begin{align}\label{eq:subh:Q:series}
		\tilde{\vQ}_{\tilde{j}}(t) = \vI + \sum_{m = 1}^\infty \sum_{\vp \in \cP_{\frac{\tilde{j}}{2}}^{(m)}} \xi_{\vp}(t) \cJ_{\vp} \;,
	\end{align}
	where
	\begin{align}\label{eq:subh:Q:p}
		\cP_{\frac{\tilde{j}}{2}}^{(m)} &= \left\{ \vp \in \Zspace^m : \abs{\frac{\tilde{j}}{2} - \sum_{i = 1}^w p_i} \leq N \text{~for all~}  w = 1, \dots m\right\} 
	\end{align} 
	are the sets of eligible integer index tuples, $\cJ_{\vp}$ refers to products of the original (non-subharmonic) Fourier coefficient matrices occurring in Equation~\eqref{eq:background:fourier}, and $\xi_{\vp}(t)$ is defined in Equation~\eqref{eq:lem:proof:xi_p:xi} of Theorem~\ref{thm:proof:Phi_series:xi}. The series in Equation~\eqref{eq:subh:Q:series} is absolutely convergent for all $t \in \Rspace$ if all Fourier coefficients $\left\{\vJ_k\right\}_{k = -2N, \dots, 2N}$ of the original (non-subharmonic) Fourier series are bounded.
\end{lemma}

\noindent The proof of Theorem~\ref{thm:subh:Q} is deferred to Section~\ref{sec:proof:subh:Q}, as with the other series proofs.

The index set defined in Equation~\eqref{eq:subh:Q:p} generalizes the eligible index set $\cP_j^{(m)}$ from Theorem~\ref{thm:proof:Q} (cf.\ Equation~\eqref{eq:proof:Pj}) to non-integer values of $j$. Any block with even $\tilde{j}$, i.e., $\tilde{\vQ}_{\tilde{j}}$ with $\tilde{j} = 2j$ and $j \in \Zspace$, is identical to the corresponding block $\vQ_j$ of the non-subharmonic $\vQ(t)$ from Theorem~\ref{thm:proof:Q}. In particular, the central blocks~$\vQ_0$ and $\tilde{\vQ}_0$ coincide.

For odd $\tilde{j}$, the corresponding set $\cP_{\frac{\tilde{j}}{2}}^{(m)}$ is the intersection of the index sets $\cP_{\frac{\tilde{j}+1}{2}}^{(m)}$ and $\cP_{\frac{\tilde{j} - 1}{2}}^{(m)}$, evaluated at the integer values adjacent to $\frac{\tilde{j}}{2}$. This follows from the fact that equality in the ``$\leq$'' condition of Equation~\eqref{eq:subh:Q:p} can never be attained when $\frac{\tilde{j}}{2}$ is not integer, as visualized in Figure~\ref{fig:subh:setdiff} for the case $m = 2$. The set difference between $\cP_{\frac{\tilde{j}}{2}}^{(m)}$ and $\cP_{\frac{\tilde{j}+1}{2}}^{(m)}$ then forms an ``outer edge'' of the larger set $\cP_{\frac{\tilde{j}+1}{2}}^{(m)}$. 
The difference $\tilde{\vQ}_{\tilde{j}+1} - \tilde{\vQ}_{\tilde{j}}$ between two neighboring blocks of $\tilde{\vQ}$ sums exactly over this set difference for all $m$. In other words, it isolates all integer index tuples in the series that lie on a specific edge of a parallelotope.

\begin{figure}
	\centering
	\begin{subfigure}[t]{0.475\textwidth}
		\centering
		\includegraphics{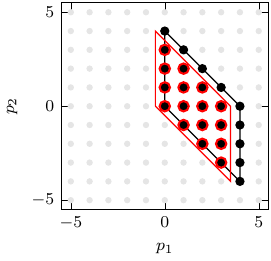}
		\caption{Sets $\cP^{(2)}_2$ (black) and $\cP^{(2)}_{1.5}$ (red) for $N = 2$ and their boundaries. Set difference is an edge of $\cP^{(2)}_2$.}
		\label{fig:subh:setdiff}
	\end{subfigure}
	\hfill
	\begin{subfigure}[t]{0.475\textwidth}
		\centering
		\includegraphics{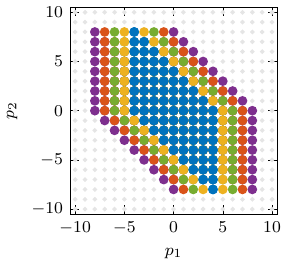}
		\caption{Set $\cP_{\mathrm{subh}}^{(2)}$ (large dots) for $N=4$, constructed by $\cP_0^{(2)}$ (blue) and edges of $\cP_{\pm j }^{(2)}$, $j = 1, \dots, 4$ (yellow, green, red, purple).}
		\label{fig:subh:construct}
	\end{subfigure}
	\caption{Construction of the subharmonic index set $\cP_{\mathrm{subh}}^{(2)}$ as the union of sets $\cP_j^{(2)}$, $j = -N, \dots, N$.}
	\label{fig:subh}
\end{figure}

As a result, an approximation of the fundamental solution matrix with significantly enlarged index sets can be constructed by taking an alternating superposition of the blocks in $\tilde{\vQ}(t)$: 
\begin{subequations}\label{eq:subh:KoopHill}
\begin{align}
	\vPh_{\mathrm{subh}}(t):= \vCsub \tilde{\vQ}(t) &= \vCsub \ex^{\ic \omega \tilde{\vD} t}\ex^{\tilde{\vH} t} \tilde{\vW} \\
	\vCsub & = \left( \vI, -\vI, \vI, \dots, -\vI, \vI\right) & &\in \Rspace^{n \times n(4N+1)}\\
	\tilde{\vW} & = \left( \vI, \dots \vI \right) \T & &\in \Rspace^{n(4N+1) \times n} \;.
\end{align}
\end{subequations}
Equation~\eqref{eq:subh:KoopHill} is termed the subharmonic Koopman-Hill projection formula, which has been shown numerically to provide improved accuracy~\cite{Bayer2024}. The series expression for this formulation is formalized in the following theorem.

\begin{theorem}[Series expression for $\vPhsub$]\label{thm:proof:Phsub}
	The series expression for the subharmonic Koopman-Hill projection formulation, with $\tilde{\vQ}_{\tilde{j}}$ denoting the blocks of $\tilde{\vQ}$ in Equation~\eqref{eq:subh:Q_def}, is given by
	\begin{align}\label{eq:subh:series}
		\vPh_{\mathrm{subh}}(t) :=  \sum_{\tilde{j} = -2N}^{2N} (-1)^{\tilde{j}} \tilde{\vQ}_{\tilde{j}}(t) = \vI + \sum_{m = 1}^\infty \sum_{\vp \in \cPsub^{(m)}}\xi_{\vp} \cJ_{\vp}
	\end{align}
	with the integer index sets
	\begin{align}\label{eq:subh:Psub}
		 \cPsub^{(m)} = \bigcup_{j = -N}^N \cP_j^{(m)} = \left\{ \vp \in \Zspace^m : \abs{\sum_{l = v}^w p_l } \leq 2N \text{~for all~}v, w = 1, \dots, m \right\} \;.
	\end{align}
	The series converges absolutely for all $t \in \Rspace$. $\cJ_{\vp}$ refers to products of the original (non-subharmonic) Fourier coefficient matrices occurring in Equation~\eqref{eq:background:fourier}.
\end{theorem}

\begin{proof}
By reformulating Equation~\eqref{eq:subh:series}, the series $\vPh_{\mathrm{subh}}(t)$ is given by
	\begin{align}\label{eq:proof:subh:CQ_tilde}
		\vPh_{\mathrm{subh}}(t) = \tilde{\vQ}_0(t) + \sum_{j = 1}^{N} \left( \tilde{\vQ}_{2j}(t) - \tilde{\vQ}_{2j - 1}(t) \right) + \left( \tilde{\vQ}_{-2j}(t) - \tilde{\vQ}_{-2j + 1}(t) \right)\;.
	\end{align}
	As discussed above, the integer index set $\cP^{(m)}_{\pm (j - \frac{1}{2})}$ corresponding to the block $\tilde{\vQ}_{\pm (2j - 1)}$ is a strict subset of $\cP^{(m)}_{\pm j}$, only excluding the ``outer'' edge. Therefore, the difference between blocks $\tilde{\vQ}_{\pm 2j}$ and $\tilde{\vQ}_{\pm (2j-1)}$ provides the summation over only the outer edge of the sets $\cP^{(m)}_{\pm j}$, $m = 1, \dots, \infty$. This idea is visualized in Figure~\ref{fig:subh:construct}, with the set $\cP_{\mathrm{subh}}^{(2)}$ constituted of the set $\cP_0^{(2)}$ due to the block $\tilde{\vQ}_0$, and then additional edges to the left and right due to the summands $(\tilde{\vQ}_{\pm 2j} - \tilde{\vQ}_{\pm (2j - 1)})$ for $j = 1, \dots, N$. All these subsets are disjoint, and the whole summation in Equation~\eqref{eq:proof:subh:CQ_tilde} is given by summing over the union of the edges, or equivalently over the union of the $\cP_j ^{(m)}$ for $j = -N, \dots, N$.  

	It remains to show that this union of the sets $\cP_j^{(m)}$ is described by the set expression in Equation~\eqref{eq:subh:Psub}. First, consider $\vp \in \cP^{(m)}_j$ for an arbitrary $j \in \left\{ -N, \dots, N\right\}$. With
	\begin{align}
		\abs{\sum_{l = v}^w p_l} = \abs{ (j - \sum_{l = 1}^{v-1} p_l) - (j - \sum_{l = 1}^{w} p_l)} \leq N + N \;, 
	\end{align}
	we see immediately that all sets $\cP_j^{(m)}$, $j = -N, \dots, N$, are subsets of the expression in Equation~\eqref{eq:subh:Psub}. For the inverse direction, let $\vp \in \Zspace^m$ satisfy $\abs{\sum_{l = v}^w p_l} \leq 2N$ for all natural $v < w \leq m$.  

	Consider first the case where 
	\begin{align}
		\min_{w} \sum_{l = 1}^w p_l \geq 0 \;.
	\end{align}
	Then, the chain of inequalities
	\begin{align}
		-N \leq -N + \sum_{l = 1}^w p_l \leq -N + 2N = N
	\end{align}
	holds for all $w = 1, \dots, m$, asserting that $\vp \in \cP^{(m)}_N$. Otherwise, if the minimal sum is negative, choose
	\begin{align}\label{eq:proof:subh:min}
		j = N + \min_{w} \sum_{l = 1}^w p_l
	\end{align} and denote the minimizer of Equation~\eqref{eq:proof:subh:min} by $w^*$. Clearly, $-N \leq j < N$.
	The chain of inequalities
	\begin{align}
		N =  j - \sum_{l = 1}^{w^*} p_l \geq j - \sum_{l = 1}^w p_l = N - \left( \sum_{l = 1}^w p_l - \sum_{l = 1}^{w^*}p_l \right) \geq -N
	\end{align}
	holds for all $w = 1, \dots, m$, asserting that $\vp \in \cP^{(m)}_j$. 
\end{proof}

\subsection{Implementation of the subharmonic formulation}
Before calculating the error bound of the series in Equation~\eqref{eq:subh:series}, we make a few remarks about the efficient computation of $\vPh_{\mathrm{subh}}(t)$.  
The matrix $\tilde{\vH}$ is twice as large as the equivalent non-subharmonic matrix~$\vH$ (cf. Equa\-tions~\eqref{eq:subh:H:N1} and~\eqref{eq:subh:Hsub:N1}). Hence, if Equation~\eqref{eq:subh:Q_def} were implemented directly as written, the matrix exponential should approximately require the $2^3 = 8$-fold computation time compared to the one in Equation~\eqref{eq:background:KoopHill}. However, due to the many zeros introduced by construction of the subharmonic Hill matrix, it can be decoupled into two sub-matrices of sizes $n(2N+1)$ and $2nN$ by the orthogonal similarity transform $\vM \tilde{\vH} \vM \T$, where $\vM$ is a permutation matrix that collects first all uneven row blocks of the identity matrix, and afterwards all even ones. 
This similarity transform results in a block diagonal matrix. Exemplarily with $N = 1$, the transformed subharmonic Hill matrix is (all empty entries are zero):
\begin{footnotesize}
	\begin{align}
			\tilde{\vH} &= \vM \T \begin{pmatrix}
				\vJ_{0} + \ic \omega & \vJ_{-1} &  \vJ_{-2} \\
				\vJ_{1} & \vJ_{0} & \vJ_{-1}  \\
				\vJ_2 & \vJ_{1} & \vJ_{0} - \ic \omega& \\
				& & & \vJ_{0} + 0.5 \ic \omega & \vJ_{-1} \\
				& & & \vJ_{1} & \vJ_{0} - 0.5\ic \omega 
				\end{pmatrix} \vM \;.
	\end{align}
\end{footnotesize}
By close inspection, the first block of size $n(2N+1) \times n(2N+1)$ is exactly $\vH$, i.e., the ``original'' Hill matrix associated with the Fourier series in Equation~\eqref{eq:background:fourier}. The second block $\hat{\vH}$ of size $2nN \times 2nN$ can easily be constructed from the original matrix $\vH$ by discarding the last row and column block and then subtracting $0.5 \ic \omega$ from the diagonal elements. Similarly, reordering the elements of $\tilde{\vD}$ using $\vM$ separates the larger diagonal matrix into the original $\vD$ in the first half and a diagonal matrix $\hat{\vD}$ in the second half, which is constructed from $\vD$ by removing the last $n$ rows and columns and subtracting $0.5 \ic \omega$ on the diagonal. The Koopman-Hill approximation then becomes
\begin{align}
	\tilde{\vQ}(t) = \vM \T \begin{pmatrix} \ex^{\ic \omega \vD t} \ex^{\vH t} & \vzero \\ \vzero & \ex^{\ic \omega \hat{\vD} t} \ex^{\hat{\vH}t} \end{pmatrix}   \vM \T \tilde{\vW}  \;.
\end{align}
The matrix exponentials of the two blocks, which are each (at most) the size of the original Hill matrix, can be evaluated independently. 
	Reordering the block rows of $\tilde{\vW}$ using $\vM$ does not have any effect as all block rows are the identity matrix. Summing the blocks of $\tilde{\vQ}$ with alternating sign is equivalent to adding all entries of the first exponential block, while subtracting the others pertaining to the second block. Thus, we obtain the formula for the monodromy matrix with subharmonic Koopman-Hill projection
	\begin{align}\label{eq:subh:KoopHill:blocks}
	 	\vPh(t) \approx \vPh_{\mathrm{subh}}(t) = \vW \T \ex^{-\ic \omega \vD t} \ex^{\vH t} \vW + \hat{\vW} \T \ex^{-\ic \omega \hat{\vD} t} \ex^{\hat{\vH} t} \hat{\vW} \;,
	\end{align}
	where $\vW$ is a stack of $2N+1$ identity matrices as before, and $\hat{\vW}$ is a stack of $2N$ identity matrices. 
	As two matrix exponentials of almost same size must be evaluated, this procedure is about twice as costly as evaluating the naive projection. 

\subsection{Convergence of the subharmonic formulation}
In this section, we analyze the convergence of the subharmonic Koopman-Hill projection of Equation~\eqref{eq:subh:KoopHill:blocks} using its series expression~\eqref{eq:subh:series} and the tools of the previous sections. As in Section~\ref{sec:overview}, we want to bound
\begin{align}
	\tilde{\vE}(t) = \vPh(t) - \left( \vI + \sum_{m= 1}^\infty \sum_{ \vp \in {\cPsub}^{(m)}} \xi_{\vp}(t) \cJ_{\vp} \right) = \sum_{m= 1}^\infty \sum_{ \vp \in \Zspace^m \setminus \cPsub^{(m)}} \xi_{\vp}(t) \cJ_{\vp}\;.
\end{align}
\begin{figure}[hbtp]
	\centering
	\tikzsetnextfilename{conv_norms_subh_2}
	\begin{subfigure}[t]{0.475\textwidth}
		\centering
		\includegraphics{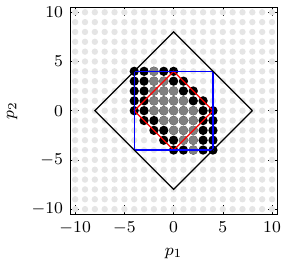}
		\caption{$N = 2$.}
		\label{fig:conv:norms:subh:2}
	\end{subfigure}
	\hfill
	\tikzsetnextfilename{conv_norms_subh_4}
	\begin{subfigure}[t]{0.475\textwidth}
		\centering
		\includegraphics{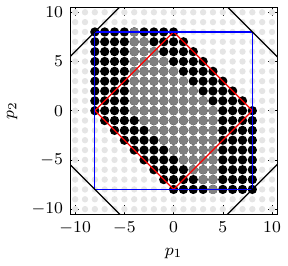}
		\caption{$N = 4$.}
		\label{fig:conv:norms:subh:4}
	\end{subfigure}
	\caption{Visualization of integer index set $\cPsub^{(2)}$ (large dots) for the subharmonic projection. $\cP_0^{(2)}$ is indicated by gray dots and additional indices only present in the subharmonic formulation by black dots. Inner rhombus (red), outer square (blue) and outer rhombus (black) are indicated.}
	\label{fig:conv:norms:subh}
\end{figure} 

Analogously to Figure~\ref{fig:conv:norms}, Figure~\ref{fig:conv:norms:subh} shows the sets $\cPsub^{(2)}$ for different truncation orders $N$ as large dots. Hereby, tuples $\vp$ that are also contained in $cP_{0}$, and thus in the direct projection of Corollary~\ref{cor:KHP}, are shown in gray while tuples that are additionally included due to the subharmonic formulation are shown in black. In the figures, the largest rhombus that is contained within $\cPsub^{(2)}$, shown in red, is of size $2N$, twice as large as in the direct projection (cf. Figure~\ref{fig:conv:norms}). This also follows immediately from the definition~\eqref{eq:subh:Psub} for arbitrary $m$. The size of the outer rhombus remains the same, compared to the direct, non-subharmonic case.
We can thus repeat the bounding steps of Section~\ref{sec:overview}, but with the size $N$ of the inner rhombus replaced by $2N$ during the counting of nonzero summands. This yields an analogous error bound, formulated in the following theorem.
\begin{theorem}[Convergence and error bound of the subharmonic Koopman-Hill projection]\label{thm:subh:error}
	Let Assumption~\ref{assu:b} hold with $b > \ln 2$. For a fixed truncation order $N \in \Nspace$, the approximation error between the true fundamental solution matrix and the subharmonic Koopman-Hill projection defined in Equation~\eqref{eq:subh:KoopHill} is bounded by
	\begin{align}\label{eq:subh:bound}
		\norm{\vPh(t) - \vCsub \ex^{\ic \omega \tilde{\vD} t}\ex^{\tilde{\vH} t} \tilde{\vW}} 
		\leq (2 \ex^{-b})^{2N} \left( \ex^{\abs{4 a t}} - 1\right)  \;.
	\end{align}
	This error bound decays exponentially with $N$, providing a convergence guarantee of the subharmonic Koopman-Hill projection. Satisfaction of a desired accuracy $\norm{\vPh(t) - \vCsub \ex^{\ic \omega \tilde{\vD} t}\ex^{\tilde{\vH} t} \tilde{\vW}} \leq E_{\mathrm{des}}$ is guaranteed if the truncation order fulfills
	\begin{align}\label{eq:subh:Nmin}
		N \geq \tilde{N}^* = \frac{ \abs{4 a t} + \ln \left( 1 - \ex^{-\abs{4 a t}}\right) - \ln(E_{\mathrm{des}})}{2 (b - \ln 2)} \;.
	\end{align}
\end{theorem}

\begin{proof}
	The proof follows by replacing $N$ by $2N$ in Equations~\eqref{eq:conv:remainder} --~\eqref{eq:proof:conv:bound} within the proof of Theorem~\ref{thm:proof:error}. 
\end{proof}
\begin{remark}
	The bound in Equation~\eqref{eq:subh:bound} decays exponentially with twice the decay rate as the corresponding bound~\eqref{eq:conv:bound} of the direct formulation. 
\end{remark}

\section{Derivative-based proofs for the series expressions}\label{sec:series}
This section is dedicated to proving Theorems~\ref{thm:proof:Phi_series:xi} and~\ref{thm:proof:Q}, which were stated but left unproven in Section~\ref{sec:overview}. Both proofs follow the same idea: Showing that the series expressions of Equations~\eqref{eq:proof:Phi_series:xi} and~\eqref{eq:proof:Q:series} solve associated linear matrix initial value problems whose solutions are given by $\vPh(t)$ and $\vQ(t)$, respectively. Once absolute convergence is established, the series expressions can be differentiated summand by summand. In  each summand of Equation~\eqref{eq:proof:Phi_series:xi} or~\eqref{eq:proof:Q:series}, the only time-dependent term is the scalar factor $\xi_{\vp}(t)$. Hence, an expression for $\dot{\xi}_{\vp}(t)$ is required in the left-hand side of each matrix initial value problem. Section~\ref{sec:series:xi_p} provides a recursive relationship between $\dot{\xi}_{\vp}(t)$ and $\xi_{[p2, \dots, p_m]}(t)$. Leveraging this  relationship, Sections~\ref{sec:series:proof} and~\ref{sec:series:KoopHill} provide proofs for Theorems~\ref{thm:proof:Phi_series:xi} and~\ref{thm:proof:Q}, respectively. Finally, Section~\ref{sec:proof:subh:Q} provides the proof of Theorem~\ref{thm:subh:Q}, using Theorem~\ref{thm:proof:Q}.

\subsection{Bound and derivatives of the scalar factor \texorpdfstring{$\xi_{\vp}(t)$}{}}\label{sec:series:xi_p}
As a preliminary step for the analysis of the series expressions in Equations~\eqref{eq:proof:Phi_series:xi} and~\eqref{eq:proof:Q:series} of Theorems~\ref{thm:proof:Phi_series:xi} and~\ref{thm:proof:Q}, useful properties of the scalar factor~$\xi_{\vp}(t)$ are established in this section. First, the following lemma isolates and proves the  polynomial bound on the absolute value of the scalar factor that is depicted in Figure~\ref{fig:proof:xi_p} and mentioned in Theorem~\ref{thm:proof:Phi_series:xi}. Afterwards, a recursive differential relationship for the $\xi_{\vp}$ with varying lengths of $\vp$ is established.

\begin{lemma}[Bound of $\xi_{\vp}(t)$]\label{lem:proof:xi_p}
	For all $m \in \Nspace \setminus \left\{ 0\right\}$,  $\vp \in \Zspace^m$, and $t \in \Rspace$, the series in Equation~\eqref{eq:lem:proof:xi_p:xi} over all multi-indices of length $m$ is absolutely convergent and bounded by Equation~\eqref{eq:lem:xi_p:bound}.
\end{lemma}
\begin{proof}
	First we establish absolute convergence. Using for all $k = 1, \dots, m$ the bound
	\begin{align}
		\left(p_1 + \dots + p_k\right)^{\alpha_k} \leq 	\left(\abs{p_1} + \dots + \abs{p_k}\right)^{\alpha_k} \leq \left(\abs{p_1} + \dots + \abs{p_k} + \dots + \abs{p_m}\right)^{\alpha_k} = \abs{\vp}^{\alpha_k} \;,
	\end{align}
	the multi-index expression in Equation~\eqref{eq:lem:proof:xi_p:xi} is bounded by
	\begin{align}\label{eq:proof:phi_p:bound}
		\left[ p_1, p_1 + p_2, \dots, p_1 + \dots + p_m\right]^{\val} \leq \abs{\vp}^{(\alpha_1 + \dots + \alpha_m)} = \abs{\vp}^{\abs{\val}} \;.
	\end{align}
	We denote summands of Equation~\eqref{eq:lem:proof:xi_p:xi} by
	\begin{align}
		s_{\val} := \frac{t^{m + \abs{\val}}}{(m + \abs{\val})!} (\ic \omega)^{\abs{\val}}\left[ p_1, p_1 + p_2, \dots, p_1 + \dots + p_m\right]^{\val} \;.
	\end{align}
	With Equation~\eqref{eq:proof:phi_p:bound}, the summands are bounded via
	\begin{align}\label{eq:proof:xi:boundp}
		\abs{s_{\val}}	& \leq \frac{\abs{t}^{m + \abs{\val}}}{(m + \abs{\val})!} \omega^{\abs{\val}} \abs{\vp}^{\abs{\val}} \;.
	\end{align}
	As the right-hand side only depends on the $1$-norm $\abs{\val}$ and not on the actual value of $\val$ itself, the sum over all $\val \in \Nspace^{m}$ can be expressed by a sum over all norms $M := \abs{\val} \in \Nspace$, where we count for each $M$ the number of tuples $\val \in \Nspace^m$ with that norm. Using Lemma~\ref{lem:proof:starsbars}\ref{lem:proof:starsbars:item:alpha} from~\ref{sec:app:combi}, there are exactly $\binom{M + m - 1}{m - 1}$ such tuples for a given norm $M$. This simplifies Equation~\eqref{eq:proof:xi:boundp} to
	\begin{align}
		\sum_{\val \in \Nspace^{m}} \abs{s_{\val}} &\leq \sum_{M = 0}^\infty \sum_{\substack{\val \in \Nspace^m \\ \abs{\val} = M}}  \frac{\abs{t}^{m + M}}{(m + M)!} \omega^{M} \abs{\vp}^{M}\\
		& = \sum_{M = 0}^\infty \frac{(M + m - 1)!}{M! (m-1)!} \frac{\abs{t}^{m + M}}{(m + M)(m + M - 1)!} \omega^M \abs{\vp}^M\\
		&= \sum_{M = 0}^\infty \frac{\abs{t}^m}{(m-1)! (m+M)} \frac{(\abs{\vp} \abs{\omega t})^M}{M!} \;.
	\end{align}
	After using $(m+M) \geq m$ to pull out the first factor, and identifying the exponential series in the second factor, the bound
	\begin{align}\label{eq:proof:series:s:bound}
		\sum_{\val \in \Nspace^{m}} \abs{s_{\val}} \leq \frac{\abs{t}^m}{m!} \ex^{\abs{\vp} \abs{\omega t}} < \infty
	\end{align}
	follows. The bound of Equation~\eqref{eq:proof:series:s:bound} is very coarse since bounding each summand $s_{\val}$ individually ignores potential cancellations due to $(\ic \omega) ^{\val}$ for different values of $\val$, but it suffices to establish absolute convergence. We can hence derive the sharper bound in Equation~\eqref{eq:lem:xi_p:bound} by considering the actual values of $s_{\val}$ rather than just their magnitudes. Using 
	$(m + \abs{\val})! \geq (m! \, \alpha_1! \, \dots \alpha_m!)$, we obtain
	\begin{align}\label{eq:proof:xi:factorials}
		\abs{\xi_{\vp}(t)} & 
		\leq \abs{\sum_{\val \in  \Nspace^m} \frac{t^{m + \abs{\val}}}{m! \, \alpha_1! \, \dots \alpha_m!} (\ic \omega)^{\abs{\val}}\left[ p_1, p_1 + p_2, \dots, p_1 + \dots + p_m\right]^{\val} }\nonumber \\
		&= \abs{ \frac{t^{m}}{m!} 
		\left( \sum_{\alpha_1 = 0}^\infty \frac{(\ic \omega p_1 t)^{\alpha_1}}{\alpha_1!} \right)
		\left( \sum_{\alpha_2 = 0}^\infty \frac{(\ic \omega (p_1 + p_2) t)^{\alpha_2}}{\alpha_2!} \right)
		\dots
		\left( \sum_{\alpha_m = 0}^\infty \frac{(\ic \omega (p_1 + \dots + p_m) t)^{\alpha_m}}{\alpha_m!} \right) }\;.
	\end{align}
	where the sums were reorganized in the second step into factors that each only depend on one $\alpha_k$.  In every round bracket of Equation~\eqref{eq:proof:xi:factorials}, the power series of the exponential function can be identified. This simplifies the expression to
	\begin{align}
		\abs{\xi_{\vp}(t)} \leq \abs{\frac{t^m}{m!} \ex^{\ic \omega p_1 t} \ex^{\ic \omega (p_1 + p_2)t }  \dots \ex^{\ic \omega (p_1 + \dots + p_m)t}} = \frac{\abs{t}^m}{m!} \;.
	\end{align}
\end{proof}
Due to Lemma~\ref{lem:proof:xi_p}, the expression for $\xi_{\vp}$ may be differentiated summand by summand. This leads after some reorganization to a recursive differential relationship between the scalar factors for various lengths of $\vp$, covered in the following lemma. 
\begin{lemma}[Derivative of $\xi_{\vp}(t)$]\label{lem:proof:xi_p:deriv}
	The derivative of $\xi_{\vp}(t)$ fulfills the following relations:
	\begin{enumerate}
		\item For $m = 1$, $p \in \Zspace$ it holds that $\dot{\xi}_p(t)= \ex^{\ic p \omega t}$. \label{enum:proof:vareq:xidiff:scalar}
		\item For $m > 1, \vp = [p_1, \dots, p_m] \in \Zspace^m$ it holds that $\dot{\xi}_{\vp}(t) = \xi_{[p_2, \dots, p_m]} (t) \, \ex^{\ic p_1 \omega t}$. \label{enum:proof:vareq:xidiff:tuple}
	\end{enumerate}
\end{lemma}
\begin{proof}
	For $m = 1$, the derivative of Equation~\eqref{eq:lem:proof:xi_p:xi} becomes
	\begin{align}\label{eq:proof:xi_p:deriv}
		\dot{\xi}_p(t) = \td{}{t} \left(\sum_{\alpha \in \Nspace} \frac{t^{\alpha + 1}}{(\alpha + 1)!} \left( \ic \omega p\right)^\alpha \right) = \sum_{\alpha = 0}^\infty \frac{t^\alpha}{\alpha!} (\ic \omega p)^\alpha =  \ex^{\ic p \omega t}\;.
	\end{align}
	
	For $m \geq 2$, consider a fixed $\vp \in \Zspace^m$. We define partial sums of $\vp$, excluding $p_1$, by $b_k := \sum_{l = 2}^k p_l$ for all $k = 2, \dots, m$. With the power series expression for the exponential
	$\ex^{\ic p_1 \omega t} = \sum_{\alpha_1 = 0}^\infty \frac{(\ic p_1 \omega t)^{\alpha_1}}{\alpha_1!}$ and the explicit expression for $\xi_{[p_2, \dots, p_m]}(t)$ of Equation~\eqref{eq:lem:proof:xi_p:xi}, the right-hand side of the differential relationship in Lemma~\ref{lem:proof:xi_p:deriv}-\ref{enum:proof:vareq:xidiff:tuple}) becomes
	\begin{align}\label{eq:proof:vareq:xidiff:p2pm}
		\xi_{[p_2, \dots, p_m]}(t) \ex^{\ic p_1 \omega t} = \sum_{\alpha_1 = 0}^{\infty} \, \sum_{[\alpha_2, \dots, \alpha_m] \in \Nspace^{m-1}} \, \frac{t^{m-1} (\ic \omega t)^{\alpha_1 + \alpha_2 + \dots + \alpha_m }}{(m-1+\alpha_2 + \dots + \alpha_m)! \, \alpha_1!} \, p_1^{\alpha_1} b_2^{\alpha_2} \dots b_m^{\alpha_m} \;.
	\end{align}
	In contrast, differentiation of the power series in Equation~\eqref{eq:lem:proof:xi_p:xi} summand by summand yields the left-hand side
	\begin{align}\label{eq:proof:vareq:xidiff:diff}
		\dot{\xi}_{\vp}(t) = \sum_{\val \in \Nspace^m} \frac{t^{m-1} (\ic \omega t)^{\abs{\val}}}{(m-1+\abs{\val})!} p_1^{\alpha_1} (p_1 + b_2)^{\alpha_2} \dots (p_1 + b_m)^{\alpha_m} \;.
	\end{align}
	The crucial difference between these two expressions is that the first integer index $p_1$ occurs in every factor in Equation~\eqref{eq:proof:vareq:xidiff:diff}, while it only occurs isolated in Equation~\eqref{eq:proof:vareq:xidiff:p2pm}. In addition, $\alpha_1$ is part of the large factorial in the denominator of Equation~\eqref{eq:proof:vareq:xidiff:diff}, but is isolated in Equation~\eqref{eq:proof:vareq:xidiff:p2pm}.
	To prove equivalence of these two series expressions, we will proceed to pull the index $p_1$ outside of every binomial in Equation~\eqref{eq:proof:vareq:xidiff:diff} and reorder the summations.
	
	Applying the binomial theorem to every factor in Equation~\eqref{eq:proof:vareq:xidiff:diff} yields
	\begin{align}\label{eq:proof:vareq:xidiff:binom:unsorted}
		\dot{\xi}_{\vp}(t) &= \sum_{\val \in \Nspace^m} \frac{t^{m-1} (\ic \omega t)^{\abs{\val}}}{(m-1+\abs{\val})!} p_1^{\alpha_1} \left(\sum_{n_2 = 0}^{\alpha_2} \binom{\alpha_2}{n_2}p_1^{\alpha_2 - n_2} b_2^{n_2 }\right) \dots \left(\sum_{n_m = 0}^{\alpha_m} \binom{\alpha_m}{n_m}p_1^{\alpha_m - n_m} b_m^{n_m }\right)
		\nonumber\\
		&= \sum_{\val \in \Nspace^m} \sum_{n_2 = 0}^{\alpha_2 } \dots \sum_{n_m = 0}^{\alpha_m } \frac{t^{m-1} (\ic \omega t)^{\abs{\val}}}{(m-1+\abs{\val})!} p_1^{\left(\abs{\val} - \sum_{k = 2}^m n_k\right)} \, b_2^{n_2} \dots b_m^{n_m} \, \binom{\alpha_2}{n_2} \dots \binom{\alpha_m}{n_m} \;.
	\end{align}
	Using $\sum_{\alpha_k = 0}^\infty \sum_{n_k = 0}^{\alpha_k} = \sum_{n_k = 0}^{\infty} \sum_{\alpha_k = n_k}^\infty$ for $k = 2, \dots, m$, every corresponding pair of sums can be swapped. We introduce the new summation multi-index $\vn = [n_2, \dots, n_m]$, which starts with $n_2$. Correspondingly, the shifted index $\tilde{\val} = \left[ \tilde{\alpha}_2, \dots, \tilde{\alpha}_m\right] \in \Nspace^{m-1}$ with $\tilde{\alpha}_k = \alpha_k - n_k$ also begins with $\tilde{\alpha}_2$. Substituted into the series after the sums are swapped, this yields
	\begin{align}
		\dot{\xi}_{\vp}(t) &=\sum_{\vn \in \Nspace^{m-1}} \sum_{\alpha_2 = n_2}^\infty \dots \sum_{\alpha_m = n_m}^{\infty}  \sum_{\alpha_1 = 0}^\infty  \frac{t^{m-1} (\ic \omega t)^{\abs{\val}}}{(m-1+\abs{\val})!} p_1^{\left(\abs{\val} - \sum_{k = 2}^m n_k\right)} \, b_2^{n_2} \dots b_m^{n_m} \, \binom{\alpha_2}{n_2} \dots \binom{\alpha_m}{n_m} \nonumber \\
		&= \sum_{\vn \in \Nspace^{m-1}} \sum_{\tilde{\val}\in \Nspace^{m-1}} \sum_{\alpha_1 = 0}^\infty  \frac{t^{m-1} (\ic \omega t)^{\alpha_1 + \abs{\tilde{\val}} + \abs{\vn}}}{(m-1+\alpha_1 + \abs{\tilde{\val}} + \abs{\vn})!} p_1^{\alpha_1 + \abs{\tilde{\val}}} \, b_2^{n_2} \dots b_m^{n_m} \, \binom{\tilde{\alpha}_2 + n_2}{n_2} \dots \binom{\tilde{\alpha}_m + n_m}{n_m} \;.
	\end{align}
	Next, the summation variable $\alpha_1$ is replaced by $M = \alpha_1 + \abs{\tilde{\val}}$ to yield
	\begin{align}
		\dot{\xi}_{\vp}(t) &= \sum_{\vn \in \Nspace^{m-1}}  \sum_{\tilde{\val}\in \Nspace^{m-1}} \sum_{M = \abs{\tilde{\val}}}^\infty  \frac{t^{m-1} (\ic \omega t)^{M + \abs{\vn}}}{(m-1+M + \abs{\vn})!} p_1^{M} \, b_2^{n_2} \dots b_m^{n_m} \, \binom{\tilde{\alpha}_2 + n_2}{n_2} \dots \binom{\tilde{\alpha}_m + n_m}{n_m} \;.
	\end{align}
	Note that only the binomial coefficients and the lower bound of $M$ depend on the multi-index~$\tilde{\val}$. In the next step, we intend to swap the summation over $M$ with the summation over~$\tilde{\val}$. For the sake of brevity, we temporarily collect all $\tilde{\val}$-independent terms of the summand into
	\begin{align}
		h(t, M, \vn, \vp) := \frac{t^{m-1} (\ic \omega t)^{M + \abs{\vn}}}{(m-1+M + \abs{\vn})!} p_1^{M} \, b_2^{n_2} \dots b_m^{n_m} \;.
	\end{align}
	For all summands it holds that $M \geq \abs{\tilde{\val}}$. Hence, the value of $M$ acts as an upper bound on~$\abs{\tilde{\val}}$. In other words, for any fixed value $M$, $\tilde{\alpha}_2$ can take any value that is not larger than~$M$. Then, for fixed $M$ and $\tilde{\alpha}_2$, the value of $\tilde{\alpha}_3$ can be at most $M - \tilde{\alpha}_2$, and this procedure can be continued to pull the summation over $M$ to the front and yield
	\begin{align}
		\dot{\xi}_{\vp}(t) &= \sum_{\vn \in \Nspace^{m-1}} \sum_{M = 0}^\infty  h(t, M, \vn, \vp) \sum_{\tilde{\alpha}_2 = 0}^M \sum_{\tilde{\alpha}_3 = 0}^{M - \tilde{\alpha}_2} \dots \sum_{\tilde{\alpha}_m = 0}^{M - \tilde{\alpha}_2 - \dots - \tilde{\alpha}_{m-1}}  \, \binom{\tilde{\alpha}_2 + n_2}{n_2} \dots \binom{\tilde{\alpha}_m + n_m}{n_m} \;.
	\end{align}
	The $m-1$ right-most summations are covered by Lemma~\ref{cor:proof:prelim:vandermonde:multsums} of \ref{sec:app:combi}, so all summation indices $\tilde{\alpha}_k$ for $ k = 2, \dots, m$ can be eliminated:
	\begin{align}
		\dot{\xi}_{\vp}(t) &= \sum_{\vn \in \Nspace^{m-1}} \sum_{M = 0}^\infty  h(t, M, \vn, \vp) \binom{M + (m-1) + \abs{\vn}}{M} \;.
	\end{align}
	Re-substituting $h(t, M, \vn, \vp)$ and writing the binomial coefficient in factorial form yields
	\begin{align}
		\dot{\xi}_{\vp}(t) &= \sum_{\vn \in \Nspace^{m-1}} \sum_{M = 0}^\infty  \frac{t^{m-1} (\ic \omega t)^{M + \abs{\vn}}}{(m-1+M + \abs{\vn})!} p_1^{M} \, b_2^{n_2} \dots b_m^{n_m} \frac{(M + (m-1) + \abs{\vn})!}{M! (m-1 + \abs{\vn})!} \nonumber \\
		&= \sum_{\vn \in \Nspace^{m-1}} \sum_{M = 0}^\infty  \frac{t^{m-1} (\ic \omega t)^{M + \abs{\vn}}}{(m-1+\abs{\vn} )! M !} p_1^{M} \, b_2^{n_2} \dots b_m^{n_m}\;. \label{eq:proof:vareq:xidiff:likeJxi}
	\end{align}
	Comparison of Equation~\eqref{eq:proof:vareq:xidiff:likeJxi} to Equation~\eqref{eq:proof:vareq:xidiff:p2pm} reveals $\dot{\xi}_{\vp}(t) = \xi_{[p_2, \dots, p_m]}(t) \ex^{\ic p_1 \omega t}$, concluding Case~\ref{enum:proof:vareq:xidiff:tuple}.
 \end{proof}
 \ref{sec:app:periodicity} examines a consequence of Lemma~\ref{lem:proof:xi_p:deriv}, namely that $\xi_{\vp}$ is $T$-periodic with finite support of Fourier coefficients for most values of $\vp$. This periodicity property is, however, not of immediate use for the further developments and is therefore deferred to the appendix.
 In the next subsection, we use Lemma~\ref{lem:proof:xi_p:deriv} directly as a tool to show that the series representation~\eqref{eq:proof:Phi_series:xi} fulfills the differential equation~\eqref{eq:background:ode:matrix}.
\subsection{Proof of Theorem~\ref{thm:proof:Phi_series:xi}: Series expression for \texorpdfstring{$\vPh(t)$}{true fundamental matrix}}\label{sec:series:proof}
In this section, we prove Theorem~\ref{thm:proof:Phi_series:xi} by verifying that the series expression of Equation~\eqref{eq:proof:Phi_series:xi} satisfies the matrix initial value problem of Equation~\eqref{eq:background:ode:matrix}, which uniquely defines the fundamental solution matrix $\vPh(t)$. The proof relies heavily on the derivative relationship for $\xi_{\vp}$ from Lemma~\ref{lem:proof:xi_p:deriv} of the previous section.
\begin{proof}[Proof of Theorem~\ref{thm:proof:Phi_series:xi}.]
	We begin by showing that the series in Equation~\eqref{eq:proof:Phi_series:xi} of Theorem~\ref{thm:proof:Phi_series:xi} is absolutely convergent, i.e., we want show that
	\begin{align}
		1 + \sum_{m = 1}^{\infty} \sum_{\vp \in \Zspace} \norm{\xi_{\vp}(t) \cJ_{\vp}} < \infty \;.
	\end{align} 
	In a second step we will show that the series in Equation~\eqref{eq:proof:Phi_series:xi} solves the matrix initial value problem of Equation~\eqref{eq:background:ode:matrix}, which defines the fundamental solution matrix. 

	By Assumption~\ref{assu:b}, the product $\cJ_{\vp}$ of Fourier coefficients is bounded by
	\begin{align}\label{eq:proof:series:Phi:abs}
		\norm{\cJ_{\vp}} = \norm{\vJ_{p_1}\dots \vJ_{p_m}} \leq a^m \ex^{b (p_1 + \dots+p_m)} \;.
	\end{align}
	Together with the bound for $\xi_{\vp}$ from Lemma~\ref{lem:proof:xi_p}, each summand of Equation~\eqref{eq:proof:series:Phi:abs} is bounded by
	\begin{align}\label{eq:proof:Phi_series:assu}
		\norm{\xi_{\vp}(t) \cJ_{\vp}} \leq \abs{\xi_{\vp}(t)} \norm{\cJ_{\vp}} \leq \frac{\abs{t}^m}{m!} a^m \ex^{-b\abs{\vp}} = \frac{\abs{a t}^m}{m!} \ex^{-b \abs{\vp}} \;.
	\end{align}
	Here, the matrix norm is reduced to a $1$-norm of the integer index tuple~$\vp$. As in the derivation of the error bound (Theorem~\ref{thm:proof:error}), the sum over all integer index tuples $\vp$ can be simplified to the sum over their norm $M = \abs{\vp}$, using Lemma~\ref{lem:proof:starsbars}\ref{lem:proof:starsbars:item:p} of~\ref{sec:app:combi} to count the number of tuples with a given norm $M$. This yields
	\begin{align}\label{eq:proof:Phi_series:replacep}
		\sum_{\vp \in \Zspace^m} \norm{\xi_{\vp}(t) \cJ_{\vp}} \leq \sum_{M = 0}^\infty \frac{\abs{a t}^m}{m!} \ex^{-b M}  \sum_{\substack{\vp \in \Zspace^m \\ \abs{\vp} = M} } 1
		\leq \sum_{M = 0}^\infty \frac{\abs{2 a t}^m}{m!} \binom{M+m-1}{m-1} \ex^{-b M} \;.
	\end{align}
	After pulling out the factor $\frac{\abs{2 a t}^m}{m!}$, the rightmost series in Equation~\eqref{eq:proof:Phi_series:replacep} is of the form $\sum_{M = 0}^\infty \binom{M+k}{k}M^q x^M$ with $q = 0$, $k = m-1$ and $x = \ex^{-b} < 1$. This is the Taylor series of $(1-x)^{-(1+k)}$, treated in Lemma~\ref{lem:series_eval} of~\ref{sec:app:Taylor}. Hence, Equation~\eqref{eq:proof:Phi_series:replacep} simplifies to
	\begin{align}
		\sum_{\vp \in \Zspace^m} \norm{\xi_{\vp}(t) \cJ_{\vp}} \leq \frac{\abs{2 a t}^m}{m!} \sum_{M = 0}^\infty \binom{M+m-1}{m-1} \ex^{-bM}
		=
		\frac{\abs{2 a t}^m}{m!} \left( 1 - \ex^{-b}\right)^{-m} \;.
	\end{align}
	Absolute convergence follows by 
	\begin{align}
		1 + \sum_{m = 1}^\infty \sum_{\vp \in \Zspace^{m}} \norm{\xi_{\vp}(t) \cJ_{\vp}} 
		\leq 
		1 + \sum_{m = 1}^\infty \frac{1}{m!} \left(\frac{\abs{2at}}{1 - \ex^{-b}}\right)^m 
		= \exp \left(\frac{\abs{2at}}{1 - \ex^{-b}}\right) < \infty \;.
	\end{align}
	
	With the absolute convergence established, we are now ready to show that Equation~\eqref{eq:proof:Phi_series:xi} satisfies the matrix initial value problem of Equation~\eqref{eq:background:ode:matrix}. 
	As $\xi_{\vp}(0) = 0$ for arbitrary $m \in \Nspace \setminus \left\{ 0 \right\}$ and $\vp \in \Zspace^m$, the initial condition $\vPh(0) = \vI$ is trivially satisfied. Differentiation of the series summand by summand yields the derivative
	\begin{align}\label{eq:proof:vareq:phidot}
		\dot{\vPh}(t) = \sum_{m = 1}^\infty \sum_{\vp \in \Zspace^m} \dot{\xi}_{\vp}(t) \cJ_{\vp} \;,
	\end{align}
	and the right-hand side of~the matrix initial value problem of Equation~\eqref{eq:background:ode:matrix} reads
	\begin{align}
		\vJ(t) \vPh(t) &= \left( \sum_{k = -\infty}^\infty \vJ_k \ex^{\ic k \omega t} \right) \left( \vI + \sum_{m = 1}^{\infty} \sum_{\vp \in \Zspace^m} \xi_{\vp}(t)  \cJ_{\vp} \right) \nonumber \\
		& = \sum_{p = -\infty}^\infty \ex^{\ic p \omega t} \vJ_p + \sum_{m = 1}^\infty \sum_{k = -\infty}^{\infty} \sum_{\vp \in \Zspace^m} \xi_{\vp}(t) \ex^{\ic k \omega t} \vJ_k \cJ_{\vp} \;.
	\end{align}
	With new integer index tuples $\tilde{\vp} := [k, \vp] \in \Zspace^{m+1}$ and the index shift $\tilde{m} = m+1$, this expression can be reformulated to
	\begin{align}\label{eq:proof:vareq:Jphi}
		\vJ(t) \vPh(t) &= \sum_{p = -\infty}^\infty \ex^{\ic p \omega t} \vJ_p  + \sum_{\tilde{m} = 2}^{\infty} \sum_{\tilde{\vp} \in \Zspace^{\tilde{m}} } \xi_{[\tilde{p}_2, \dots, \tilde{p}_{\tilde{m}]}}(t)\, \ex^{\ic \tilde{p}_1 \omega t}\cJ_{\tilde{\vp}} \;.
	\end{align} 
	Comparison of Equations~\eqref{eq:proof:vareq:phidot} and~\eqref{eq:proof:vareq:Jphi} reveals that the matrix initial value problem of Equation~\eqref{eq:background:ode:matrix} is satisfied if the following conditions on $\xi_{\vp}$ hold:
	\begin{enumerate}
		\item $\dot{\xi}_p(t) = \ex^{\ic p \omega t}$ if $p \in \Zspace$
		\item $\dot{\xi}_\vp(t) = \xi_{[p_2, \dots, p_m]} \ex^{\ic p_1 \omega t}$ if $\vp \in \Zspace^m$ with $m \geq 2$.
	\end{enumerate}
	As these conditions are true due to Lemma~\ref{lem:proof:xi_p:deriv}, the series of Equation~\eqref{eq:proof:Phi_series:xi} does indeed solve the matrix initial value problem for the fundamental solution matrix. Since the solution of a linear initial value problem is unique, the series of Equation~\eqref{eq:proof:Phi_series:xi} must thus be the fundamental solution matrix.
\end{proof}
\subsection{Proof of Theorem~\ref{thm:proof:Q}: Series expression for Koopman-Hill approximation \texorpdfstring{$\vQ(t)$}{Q(t)}}\label{sec:series:KoopHill}	
Differentiating Equation~\eqref{eq:proof:KoopHill:Q_defin} reveals the linear matrix initial value problem 
\begin{align} \label{eq:proof:Q:ode}
		\dot{\vQ}(t) &= \left( \ic \omega \vD + \ex^{\ic \omega \vD t} \vH \ex^{-\ic \omega \vD t} \right) \vQ(t) & 
		\vQ(0) &= \vW \;,
\end{align}
whose unique solution is $\vQ(t) = \ex^{\ic \omega \vD t} \ex^{\vH t} \vW$. Hereby, $\vQ(t)$ is constituted of $2N+1$ vertically stacked blocks of size $n \times n$, and correspondingly the matrix in round brackets in Equation~\eqref{eq:proof:Q:ode} can be separated into $2N+1$ row and column blocks each. In this section, we prove Theorem~\ref{thm:proof:Q} using the same approach as for Theorem~\ref{thm:proof:Phi_series:xi} in the previous section. Specifically, we show that the stack of blocks defined by the series expression in Equation~\eqref{eq:proof:Q:series}, for $j = -N, \dots, N$, satisfies the matrix initial value problem in Equation~\eqref{eq:proof:Q:ode}.

\begin{proof}[Proof of Theorem~\ref{thm:proof:Q}]
	We establish absolute convergence of the series expression of Equation~\eqref{eq:proof:Q:series} first. The series 
	\begin{align}
		\sum_{m = 1}^\infty \sum_{\vp \in \cP_j^{(m)}} \norm{\xi_{\vp}(t) \cJ_\vp} \label{eq:proof:Q:norm} 
	\end{align}
	is a partial sum of the series in Equation~\eqref{eq:proof:series:Phi:abs}, which was shown to converge absolutely if Assumption~\ref{assu:b} is satisfied. In the truncated case considered here, we can actually drop Assumption~\ref{assu:b}: For any multi-index $\vp$ that appears in Equation~\eqref{eq:proof:Q:norm}, each entry $p_k$ must satisfy $\abs{p_k} \leq 2N$, as discussed in Section~\ref{sec:overview}. 
	Any nonzero Fourier coefficients $\vJ_p$ with $\abs{p} > 2N$ do not affect Equation~\eqref{eq:proof:Q:norm}, so we may, without loss of generality, assume that the Fourier coefficients have finite support $2N$. This ensures absolute convergence of the series, automatically satisfying Assumption~\ref{assu:b}.
	
	It remains to show that the right-hand side of Equation~\eqref{eq:proof:Q:series} fulfills the matrix initial value problem given by Equation~\eqref{eq:proof:Q:ode}. The initial condition is true by definition. 
	For the differential equation, the approach is analogous to the proof for Theorem~\ref{thm:proof:Phi_series:xi}, only the summation limits have to be considered carefully now. 
	
	The $(jk)$-th block ($k,j = -N, \dots, N$) of the time-periodic matrix $\left(\ic \omega \vD + \ex^{\ic \omega  \vD t} \vH \ex^{-\ic \omega  \vD t} \right)$ from the initial value problem of Equation~\eqref{eq:proof:Q:ode} is
	\begin{align}
		\left(\ic \omega \vD + \ex^{\ic \omega \vD t} \vH \ex^{-\ic \omega \vD t}\right)_{jk} =  \vJ_{j-k} \ex^{\ic(j-k) \omega t} 
	\end{align}
	as the matrix $\ic \omega \vD$ removes the $\ic \omega$-terms in $\vH$ and pre-multiplication with the diagonal matrix $\ex^{\ic \omega \vD t}$ scales the rows of $\vH$, while post-mu\-lti\-pli\-cation with $\ex^{-\ic  \omega \vD t}$ scales the columns. For the purpose of this proof, we will call the right-hand side of Equation~\eqref{eq:proof:Q:ode} $\widehat{\dot{\vQ}}$ if the series expression for $\vQ$ is used, as we do not (yet) know that it actually coincides with $\dot{\vQ}$. Evaluating the matrix multiplication block by block with the series expression of Equation~\eqref{eq:proof:Q:series} from Theorem~\ref{thm:proof:Q} yields
	\begin{align}
		\widehat{\dot{\vQ}}_j(t) := \left[\left( \ic \omega \vD + \ex^{\ic \omega \vD t} \vH \ex^{-\ic \omega \vD t} \right) \vQ(t)\right]_{j} = \sum_{l = -N}^N \vJ_{j-l} \ex^{\ic (j-l) \omega t} \left( \vI + \sum_{m = 1}^\infty \sum_{\vp \in \cP_{l}^{(m)}} \xi_{\vp} (t) \cJ_{\vp} \right)
	\end{align}
	or, after the index shift $\tilde{p}_1 = j-l$ and some reorganization,
	\begin{align}
		\widehat{\dot{\vQ}}_j(t)  
		&= 
		\sum_{\tilde{p}_1 = j-N}^{j+N} \vJ_{\tilde{p}_1} \ex^{\ic \tilde{p}_1 \omega t} 
		+ \sum_{m = 1}^\infty \sum_{\tilde{p}_1= j-N}^{j+N} \sum_{[p_1, \dots, p_m] \in \cP_{j - \tilde{p}_1}^{(m)}} 
		\xi_{[p_1, \dots, p_m]}(t) \, \ex^{\ic \tilde{p}_1 \omega t} \vJ_{\tilde{p}_1} \cJ_{\vp} \;.
	\end{align}
	With the relabeling $[p_1, \dots, p_m] := [\tilde{p}_2, \dots, \tilde{p}_{m+1}]$ and the index shift in $m$ by one, the expression
	\begin{align}
		\widehat{\dot{\vQ}}_j (t)
		&= 
		\sum_{\tilde{p}_1 = j-N}^{j+N} \vJ_{\tilde{p}_1} \ex^{\ic \tilde{p}_1 \omega t} 
		+ \sum_{m = 2}^\infty \sum_{\tilde{p}_1= j-N}^{j+N} \sum_{[\tilde{p}_2, \dots, \tilde{p}_m] \in \cP_{j-\tilde{p}_1}^{(m-1)}} \xi_{[\tilde{p}_2, \dots, \tilde{p}_m]}(t) \, \ex^{\ic \tilde{p}_1 \omega t} \vJ_{\tilde{p}_1} \cJ_{[\tilde{p}_2, \dots, \tilde{p}_m]}\\ 
		&= \sum_{p = j-N}^{j+N} \vJ_p \dot{\xi}_p(t) + \sum_{m = 2}^\infty \sum_{\tilde{p}_1 = j-N}^{j+N} \sum_{[\tilde{p}_2, \dots, \tilde{p}_m] \in \cP_{j-\tilde{p}_1}^{(m-1)}} \dot{\xi}_{\tilde{\vp}}(t) \cJ_{\tilde{\vp}} \label{eq:proof:Q:mult}
	\end{align}
	follows, where the derivative properties of Lemma~\ref{lem:proof:xi_p:deriv} were used in the second step. This series is analogous to Equation~\eqref{eq:proof:vareq:Jphi} in Section~\ref{sec:series:proof}, but with nontrivial limits in the sum. 
	Evaluating the left-hand side of Equation~\eqref{eq:proof:Q:ode} by differentiating Equation~\eqref{eq:proof:Q:series} yields
	\begin{align}\label{eq:proof:Q:deriv}
		\dot{\vQ}_j(t) = \sum_{m = 1}^{\infty} \sum_{\vp \in \cP_j^{(m)}} \dot{\xi}_{\vp}(t) \cJ_{\vp} \;,
	\end{align}
	which is in turn analogous to Equation~\eqref{eq:proof:vareq:phidot}. The comparison of Equations~\eqref{eq:proof:Q:mult} and~\eqref{eq:proof:Q:deriv} reveals that the two expressions coincide if the summation sets $\cP_j^{(m)}$ and $\left\{ \vp: \abs{j - p_1} \leq N, [p_2, \dots, p_m] \in \cP^{(m-1)}_{j-p_1}\right\}$ are equal. With
	\begin{align}
		\vp  \in \cP^{(m)}_{j} &\iff \abs{j-p_1} \leq N \text{~and~} \abs{j - p_1 - \sum_{k = 1}^{w} p_{k+1}} \leq N \text{~for all~} w = 1, \dots, m-1 \\
		&\iff p_1 \in \left\{j-N, \dots, j+ N\right\} \text{~and~} [p_2, \dots, p_m] \in \cP^{(m-1)}_{j-p_1} \;,
	\end{align}
	this is indeed the case and, thu, the series expression of Equation~\eqref{eq:proof:Q:series} solves the matrix differential equation~\eqref{eq:proof:Q:ode}.
\end{proof}

\subsection{Proof of Theorem~\ref{thm:subh:Q}: Series expression for \texorpdfstring{$\tilde{\vQ}(t)$}{Q\textasciitilde(t)}}\label{sec:proof:subh:Q}
In this section, we prove Theorem~\ref{thm:subh:Q} by setting up the series according to Theorem~\ref{thm:proof:Q} for the subharmonic  Fourier series and exploiting the fact that every second Fourier coefficient vanishes.

\begin{proof}[Proof of  Theorem~\ref{thm:subh:Q}.]
	Consider a set $\left\{\vJ_k\right\}_{k =-\infty}^{\infty}$ of nontrivial (original) Fourier coefficient matrices with the corresponding Fourier series of Equation~\eqref{eq:background:fourier} with frequency $\omega$. Consider also the associated subharmonic Fourier series (cf. Equation~\eqref{eq:subh:fourierseries}) with  frequency $\tilde{\omega} = 0.5\omega$ and Fourier coefficient matrices  $\left\{\tilde{\vJ}_{\tilde{k}}\right\}_{\tilde{k} =-\infty}^{\infty}$. According to Equation~\eqref{eq:subh:J_subh}, $\tilde{\vJ}_{\tilde{k}}$ is zero for odd $\tilde{k}$. 

	Let $\tilde{\vQ}(t)$ be constructed according to Equation~\eqref{eq:subh:Q_def}, i.e., by the standard Koopman-Hill approach, but applied with even truncation order~$\tilde{N}$ to the subharmonic Fourier series of Equation~\eqref{eq:subh:fourierseries}. With Theorem~\ref{thm:proof:Q}, the $\tilde{j}$-th block of the series expression for $\tilde{\vQ}(t)$, expressed in the subharmonic Fourier coefficients $\left\{\tilde{\vJ}_{\tilde{k}}\right\}_{\tilde{k} =-\infty}^{\infty}$, is
	\begin{align}\label{eq:subh:Q:subh}
		\tilde{\vQ}_{\tilde{j}}(t) = \vI + \sum_{m = 1}^{\infty} \sum_{\tilde{\vp} \in \tilde{\cP}^{(m)}_{\tilde{j}}} \xi_{\vp}(t) \tilde{\cJ}_{\tilde{\vp}} \;.
	\end{align}
	The $\tilde{}$ symbols above $\tilde{\cJ}$ and $\tilde{\cP}$ indicate that they are evaluated with respect to the subharmonic quantities, i.e.,  $\tilde{\cJ}_{\tilde{\vp}}$ is a product of the \emph{subharmonic} Fourier coefficient matrices $\left\{\tilde{\vJ}_{\tilde{k}}\right\}_{\tilde{k} =-\infty}^{\infty}$ of which every odd one is zero, and $\tilde{\cP}_{\tilde{j}}^{(m)}$ is the integer index set of Equation~\eqref{eq:proof:Pj}, evaluated at $\tilde{j}$ and truncation order $\tilde{N}$. This proof will proceed by expressing Equation~\eqref{eq:subh:Q:subh} in terms of the original, non-subharmonic  Fourier coefficient matrices $\left\{\vJ_k\right\}_{k =-\infty}^{\infty}$ and the original integer index sets $\cP_j^{(m)}$, evaluated at truncation order $N = 0.5\tilde{N}$. 

	As the Fourier coefficient matrices $\tilde{\vJ}_{\tilde{k}}$ are zero for odd $\tilde{k}$, products $\tilde{\cJ}_{\tilde{\vp}}$ can only be nonzero if every individual entry of the integer index tuple $\tilde{\vp}$ is even. This can be ensured by requiring $\tilde{\vp} = 2 \vp$ with $\vp \in \Zspace ^m$. In other words, we seek an expression for the set of integer indices $\vp$ such that $2 \vp$ is an element of $\tilde{\cP}_{\tilde{j}} ^{(m)}$. Substituting $\tilde{\vp} = 2 \vp$ into the condition for the set $\tilde{\cP}_{\tilde{j}}^{(m)}$ yields
	\begin{align}
		\abs{\tilde{j} - \sum_{l = 1}^w 2p_l } \leq \tilde{N} \;,
	\end{align}
	or, by substituting $\tilde{N} = 2N$ and dividing by $2$, 
	\begin{align}
		\abs{\frac{\tilde{j}}{2} - \sum_{l = 1}^w p_l} \leq N \;.
	\end{align}
	This condition yields exactly the set $\cP_{\frac{\tilde{j}}{2}}$ as defined in Equation~\eqref{eq:subh:Q:p}. Thus, the summation over $\tilde{\cP}_{\tilde{j}}$ in Equation~\eqref{eq:subh:Q:subh} can equivalently be replaced by $\cP_{\frac{\tilde{j}}{2}}$, yielding
	\begin{align}
		\tilde{\vQ}_{\tilde{j}}(t) = \vI + \sum_{m = 1}^{\infty} \sum_{\vp \in \cP^{(m)}_{\frac{\tilde{j}}{2}}} \xi_{\vp}(t) \tilde{\cJ}_{2\vp} \;.
	\end{align} 
	As $\tilde{\vJ}_{2k} = \vJ_k$ implies $\tilde{\cJ}_{2\vp} = \cJ_{\vp}$, the proof is complete.
\end{proof}

\section{Examples}\label{sec:examples} 
In this section, the error bounds that were derived in Sections~\ref{sec:overview} and~\ref{sec:subharmonics} are illustrated and applied to three examples. First, we consider a scalar example with finite support of Fourier coefficients, whose fundamental solution is available in closed form. Afterwards, we discuss how our error bound can offer guaranteed stability determination of the well-known Mathieu equation. Finally, we leverage the Duffing oscillator as an example for the complete workflow including finding a periodic solution and its subsequent stability analysis.
\subsection{Scalar example} 
As a first simple example, consider the scalar linear time-periodic differential equation
\begin{align}\label{eq:example:scalar:dynamics}
		\dot{y} = J(t) y = (\beta + 2\gamma \cos t) y 
	\end{align}
with $\beta, \gamma \in\Rspace$. 
As the differential equation is scalar and linear, solutions are available in closed form. In particular, the fundamental solution matrix (here scalar), i.e., the solution initialized at $y(0) = 1$, is given by
\begin{align}\label{eq:example:scalar:phi}
	\phi(t, \beta, \gamma) = \ex^{(\beta t + 2 \gamma \sin t)} \;.
\end{align}

Using this closed-form solution, the actual error of the direct Koopman-Hill ap\-proxi\-ma\-tion with a given $N$ can be compared to the bound of Theorem~\ref{thm:proof:error}. 
The Fourier coefficients of the system matrix $J(t)$ (here scalar) can be immediately read off: $J_0 = \beta$, $J_1 = J_{-1} = \gamma$. All other Fourier coefficients are zero. Due to the finite support of the Fourier coefficients, the exponential decay condition of Assumption~\ref{assu:b} is satisfied for arbitrary $b > 0$ and the error bounds of Theorems~\ref{thm:proof:error} and~\ref{thm:subh:error} hold for all $b > \ln 2$.  Nonetheless, the choice $b$ does influence the numerical value of the error bound.
Writing $b(\varepsilon) = \ln 2 - \ln(\varepsilon)$ with $\varepsilon \in (0, 1)$, such that $b(\varepsilon)$ approaches $\ln 2$ as $\varepsilon \rightarrow 1$, the Koopman-Hill error bound in Theorem~\ref{thm:proof:error} simplifies to
\begin{align}\label{eq:example:bound}
	\abs{\vE(T)} \leq \varepsilon^N  \left(  \exp\left( 4 a  t \right) - 1 \right)\;.
\end{align}
The smaller $\varepsilon$, or equivalently the steeper the exponential decay $b$, the faster the error converges to zero. However, the lowest possible value for $a$ increases when $b$ is increased.  For a given $b(\varepsilon)$, the lowest possible value for $a$ in the scalar example with finite support considered here is $a(\varepsilon) = \max\left\{ \beta, \frac{2\gamma}{\varepsilon} \right\}$. In the first case $a=\beta$, the exponentially decaying bound $a \ex^{-b \abs{k}}$ on the Fourier coefficients is tight on the zero-th harmonic and overshoots the first harmonic, and vice versa for the other case. 

Figure~\ref{fig:example:scalar:E} shows the error bound of Equation~\eqref{eq:example:bound} for various values of $\varepsilon$ and the corresponding best values of $a(\varepsilon)$. The error bound reaches its minimum at a certain optimal value, indicated by a black cross. 
If the Fourier coefficient $\gamma$ of the first harmonic is dominant over $\beta$ (see Figure~\ref{fig:example:scalar:E:001}), the optimal decay is tight at the first harmonic with $a(\varepsilon) = \frac{2\gamma}{\varepsilon}$ and the minimum occurs approximately at $\varepsilon^* = \frac{8 t \gamma}{N}$ with $a(\varepsilon^*) = \frac{N}{4t}$. For the parameters used in this example, truncation orders $N$ between 113 and 142 guarantee a sufficiently small error. 

However, if the zero-th harmonic coefficient $\beta$ is large enough that $\beta > \frac{N}{4t}$, $a(\varepsilon^*) = \frac{N}{4t}$ is not a valid decay parameter. In this case, the error bound is minimized for $\varepsilon^{**} = \frac{2\gamma}{\beta}$, where the exponential decay is tight on both the zero-th and the first harmonic. This is visible in Figure~\eqref{fig:example:scalar:E:5} where, for smaller $N$, the minimum is located at the kink, where the expression for $a(\varepsilon)$ switches. Substituting these optimal values into the error bound of Equation~\eqref{eq:example:bound}, the best possible error bound over all choices of the exponential decay is, for the considered dynamical system with finitely supported Fourier coefficients, 
\begin{align}\label{eq:example:scalar:E*}
  E^* &= 
  \begin{cases}
    \left(\frac{8 \gamma t}{N}\right)^N\left(\ex^N - 1\right) & \text{if } \beta < \frac{N}{4t} \\
    \left(\frac{2\gamma}{\beta}\right)^N \left(\ex^{4 \beta t} - 1\right) & \text{otherwise } \;.
  \end{cases}
\end{align}

\begin{figure}[hbtp]
	\centering
	\begin{subfigure}[t]{0.48\textwidth}
		\centering
		\includegraphics{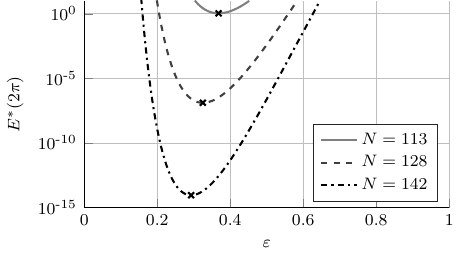}
		\caption{$\beta = 0.01$, $\gamma = 0.8$. First harmonic $\gamma$ dominates.}
		\label{fig:example:scalar:E:001}
	\end{subfigure}
	\hfill
	\begin{subfigure}[t]{0.48\textwidth}
		\centering
		\includegraphics{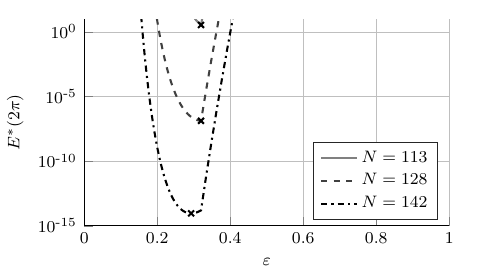}
		\caption{$\beta = 5$, $\gamma = 0.8$. Zeroth harmonic $\beta$ dominates.}
		\label{fig:example:scalar:E:5}
	\end{subfigure}
	\caption{Error bound of Equation~\eqref{eq:example:bound} over decay parameter $\varepsilon$ for $\gamma=0.8$, $t = 6.5$ and two values of $\beta$.}
	\label{fig:example:scalar:E}
\end{figure}

Equation~\eqref{eq:example:scalar:E*} can numerically be solved for $N$ to find the optimal truncation order guaranteed to satisfy a certain error bound. Figure~\ref{fig:example:bound:E} illustrates this truncation order $N^*$ across a range of desired errors for $\gamma = 0.8$. In addition, Figure~\ref{fig:example:bound:E} shows the value of $N$ that was needed to approximate the fundamental solution to this accuracy numerically. Figure~\ref{fig:example:bound:gamma} shows the truncation order $N^*$ needed to ensure an error smaller than $10^{-6}$ for various values of $\gamma$, together with its numerical equivalent. This comparison is also made for the subharmonic approach.

	\begin{figure}[hbtp]
	\centering
	\begin{subfigure}[t]{0.48\textwidth}
		\centering
		\includegraphics{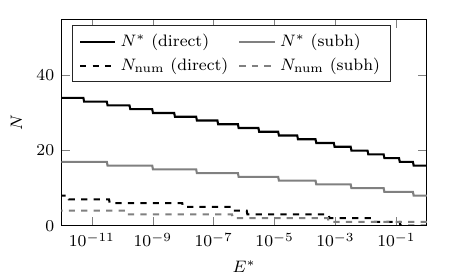}
		\caption{$N^*$ and $N_{\mathrm{num}}$ over $E^*$ for $\gamma = 0.1$.}
		\label{fig:example:bound:E}
	\end{subfigure}
	\hfill
	\begin{subfigure}[t]{0.48\textwidth}
		\centering
		\includegraphics{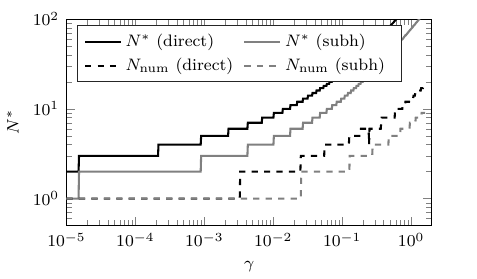}
		\caption{$N^*$ and $N_{\mathrm{num}}$ over $\gamma$ for $E^* = 10^{-6}$ .}
		\label{fig:example:bound:gamma}
	\end{subfigure}
	\caption{Guaranteed ($N^*$, solid) and actual ($N_{\mathrm{num}}$, dashed) truncation order needed to achieve an error smaller than $E^*$ in Equation~\eqref{eq:example:scalar:E*} with $\beta = 0.01, t = 6.5$ for direct (black)  and subharmonic (gray) approach.} 
	\label{fig:example:bound}
\end{figure} 

	It is immediately evident that the derived error bounds become rather conservative as $\gamma$ increases. While both the numerical and the guaranteed error grow exponentially with $\gamma$, even for this simple example, the required truncation order $N^*$ to guarantee sufficiently small errors reaches $N^* = 100$, while the truncation order that is actually needed remains around $N = 10$. This discrepancy arises from the conservative bounding steps employed throughout this work to enable closed-form expressions of the series. Nevertheless, the subharmonic approach consistently requires approximately half the truncation order of the direct approach, as predicted by the error bound. 

\subsection{Mathieu equation}\label{sec:exanple:mathieu}

The next considered example is the well-known Mathieu equation, given by
\begin{align}
	\ddot{x} + (\delta + \epsilon \cos \omega t) x = 0\;.
\end{align}
The Mathieu equation is a archetypical example of a linear time-periodic system and many applications in physics and engineering can be modeled by it, such as the interaction of gears, vibrations of membranes or the parametric roll of ships. See~\cite{Kovacic2018} for a detailed overview. In first-order form, the Mathieu equation reads
\begin{align}\label{eq:mathieu:firstorder}
	\dot{\vy} = \begin{pmatrix} 0 & 1 \\ -(\delta + \epsilon \cos \omega t) & 0 \end{pmatrix} \vy =: \vJ(t) \vy \;,
\end{align}
which is of the form of Equation~\eqref{eq:background:ode}. The Fourier coefficients of the system matrix $\vJ(t)$ are $\vJ_0 = \begin{pmatrix} 0 & 1 \\ -\delta & 0 \end{pmatrix}$, $\vJ_1 = \vJ_{-1} = \begin{pmatrix} 0 & 0 \\ -\frac{\epsilon}{2} & 0 \end{pmatrix}$, and all other Fourier coefficients are zero. Thus, the Fourier coefficients have finite support and the exponential decay condition of Assumption~\ref{assu:b} is satisfied for arbitrary $b > 0$. As in the previous example, values $a$ and $b$ that minimize the error bound of Theorem~\ref{thm:proof:error} can be identified. With $\beta = \norm{\vJ_0}$ and $\gamma = \norm{\vJ_1} = \frac{\epsilon}{2}$, the ideas of the previous section with respect to the optimal error bound of Equation~\eqref{eq:example:scalar:E*} carry over. 

\begin{figure}[hbt]
	\centering
	\begin{subfigure}[t]{0.48\textwidth}
		\centering
		\includegraphics{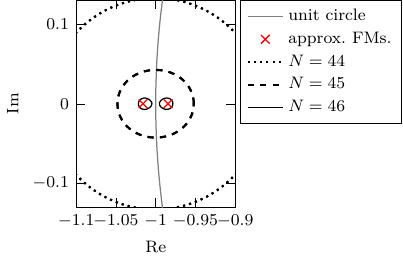}
		\caption{$\delta = -0.35485$.}
		\label{fig:matheieu:FMs:ustbl}
	\end{subfigure}
	\hfill
	\begin{subfigure}[t]{0.48\textwidth}
		\centering
		\includegraphics{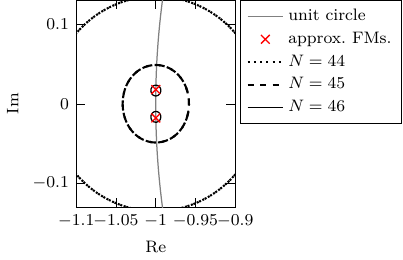}
		\caption{$\delta = -0.35490$.}
		\label{fig:mathieu:FMs:stbl}
	\end{subfigure}
	\caption{Floquet multipliers determined using subharmonic Koopman-Hill approximation and guaranteed regions for true Floquet multipliers for Mathieu equation with $\omega = 2$, $\epsilon = 2.4$.}
	\label{fig:mathieu:FMs}
\end{figure}

In contrast to the scalar example, the stability of the equilibrium of Equation~\eqref{eq:mathieu:firstorder} is not determined immediately by the monodromy matrix, but by its eigenvalues, the Floquet multipliers. However, the proposed error bounds of Theorems~\ref{thm:proof:error} and~\ref{thm:subh:error} only provide bounds for perturbations of the monodromy matrix itself. 
Bounds for the Floquet multipliers can be obtained with use of the pseudospectrum. When the 2-norm of the difference between the true and approximated monodromy matrix is bounded by $E$, the true Floquet multipliers are contained in the $E$-pseudospectrum $\Lambda_{E}$ of the approximated matrix~\cite{Trefethen1997}, given by
\begin{align}
	\Lambda_{E}(\vPh_{\mathrm{approx}}) = \left\{ z \in \Cspace : \sigma_{\mathrm{min}}(z \vI - \vPh_{\mathrm{approx}}) \leq E\right\} \;.
\end{align}
In words, for all (true) Floquet multipliers $\lambda$, the minimal singular value of $\lambda \vI - \vPh_{\mathrm{approx}}$ must be smaller than~$E$.

Figure~\ref{fig:mathieu:FMs} illustrates the Floquet multipliers determined using the subharmonic Koopman-Hill approximation for two values of $\delta$ that are close to each other. For $\delta = -0.35485$ (Figure~\ref{fig:matheieu:FMs:ustbl}), the Floquet multipliers are real with one lying outside the unit circle. For $\delta = -0.35490$ (Figure~\ref{fig:mathieu:FMs:stbl}), the Floquet multipliers are both complex with $\abs{\lambda_1} = \abs{\lambda_2} = 1$. The Floquet multipliers meet at $-1$ when $\delta$ is varied from the first to the latter value, inducing a stability change. The fundamental solution matrix is not well-conditioned around stability changes, where two Floquet multipliers coincide and the eigenspace deflates. 

For the considered values $N \in \left\{ 44, 45, 46\right\}$, the approximated Floquet multipliers are identical up to numerical precision. However, the error bound of Equation~\eqref{eq:example:scalar:E*} changes, leading to different pseudospectra. It is well known that the Floquet multipliers of the Mathieu equation either lie on the unit circle (stable case) or on the real axis (unstable case)~\cite{Kovacic2018}. For $N = 44$ and $N = 45$, the pseudospectra are large enough that they contain segments of the unit circle as well as segments of the real axis, so correct stability assertion could not be guaranteed using the error bound with these truncation orders. However, for $N = 46$, the error bound is small enough that the pseudospectra do not intersect both the unit circle and the real axis, allowing for a stability assertion that is guaranteed to be correct. 

Figure~\ref{fig:mathieu:traverse} shows the truncation orders $N$ that are required for various statements about the Mathieu equation with $\epsilon = 2.4$, $\omega = 2$ using the subharmonic Koopman-Hill approximation. The truncation orders are shown over a range of values of $a$ that traverse several stability changes of the Mathieu equation. The solid black line reports the truncation orders needed to guarantee an error of the fundamental solution matrix smaller than~$10^{-6}$, determined using Equation~\eqref{eq:example:scalar:E*}. The truncation orders needed to numerically achieve this error are shown in dashed black, determined by computing the actual error for increasing truncation orders until the desired accuracy is reached. The truncation orders needed to guarantee a correct stability assertion, shown in solid gray, are determined by increasing $N$ until the pseudospectra do not intersect both the unit circle and the real axis. The truncation orders needed to numerically achieve a correct stability assertion are reported in dashed gray. They are determined by increasing $N$ until the stability determined from the approximated Floquet multipliers matches that determined from a high-accuracy reference solution.

The truncation orders needed to guarantee or numerically achieve a certain error of the fundamental matrix are constant over the whole range of $\delta$. Consistently with the previous example and with the simplifications of the proof, the error bound is rather conservative, leading to truncation orders that are significantly larger than those needed to numerically achieve the desired accuracy. 

The truncation orders needed for correct stability assertion, however, vary over $\delta$. In the unstable region between $\delta \approx -0.3$ and $\delta \approx 2$, the correct stability assertion ``unstable'' is made immediately with truncation order $1$. This, however, is misleading because for such small truncation orders, the numerically determined Floquet multipliers are still complex conjugated, just far away from the unit circle. In the stable regions, the truncation orders needed for correct stability assertion are significantly larger, as the Floquet multipliers are close to the unit circle and a small error can lead to a wrong stability assertion. Around stability changes, the truncation orders needed for correct stability assertion peak, as the fundamental solution matrix is ill-conditioned there. However, even during these peaks, the truncation orders needed to guarantee correct stability assertion are significantly smaller than those needed to guarantee the error $10^{-6}$. This shows that ill-conditioning during stability changes is not an issue for the method. 

\begin{figure}[hbt]
	\centering
	\includegraphics{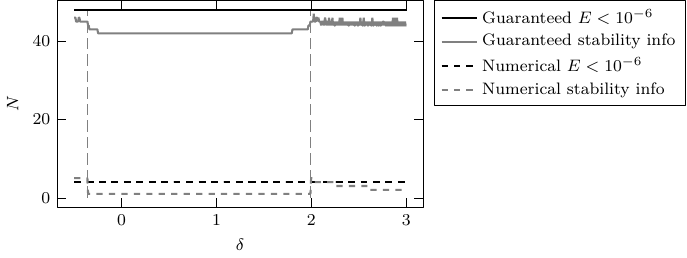}
	\caption{Truncation orders $N$ needed to guarantee (solid) or numerically achieve (dashed) error < $10^{-6}$ or correct stability assertion for Mathieu equations with $\epsilon = 2.4$, $\omega = 2$ using subharmonic Koopman-HIll approximation. Stability changes of the ODE indicated by vertical lines.}
	\label{fig:mathieu:traverse}
\end{figure}

\begin{figure}[hbt]
	\centering
	\includegraphics{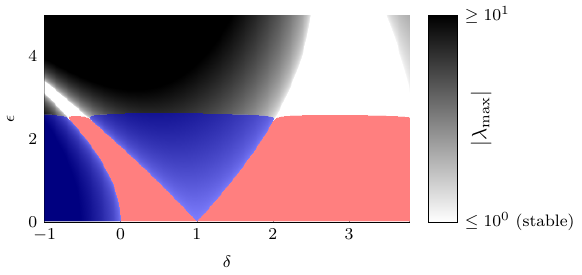}
	\caption{Ince-Strutt diagram for Mathieu equation, computed using subharmonic Koopman-Hill approximation with $N=45$. Grayscale color corresponds to maximum magnitude of computed Floquet multipliers, white ($1$) if stable. Red/blue overlay indicates parameter combinations guaranteed to be stable/unstable using the error bound.}
	\label{fig:mathieu:incestrutt}
\end{figure}

Figure~\ref{fig:mathieu:incestrutt} shows the so-called Ince-Strutt-diagram, i.e., the stability regions of the Mathieu equation over~$\delta$ and~$\epsilon$. The stability regions determined using the subharmonic Koopman-Hill approximation with truncation order $N = 45$ are shown. For every pixel, the stability was determined using the maximum magnitude of the computed Floquet multipliers. Afterwards, the unit circle and the real axis were sampled to assert whether they are both contained in the pseudospectrum of the approximate fundamental solution matrix at this point, determined using the error bound of Equation~\eqref{eq:example:scalar:E*}. 
If only one of the two sets is contained, the stability assertion is guaranteed to be correct and the pixel is colored red (guaranteed stable) or blue (guaranteed unstable). If both sets are contained, correct stability assertion is not guaranteed by the error bound and the pixel is colored in grayscale according to the maximum magnitude of the numerically determined Floquet multipliers. 

The resulting Ince-Strutt-diagram is in excellent agreement with the known stability regions of the Mathieu equation~\cite{Kovacic2018}. As expected, large values of $\epsilon$ increase the error bound, leading to large pseudospectra that contain both the real axis and the unit circle. While the numerical stability assertion is still correct, larger truncation orders $N$ could be chosen in these regions to guarantee correct stability insights. Only minor difficulties in guaranteeing stability are observed close to the stability changes, solidifying the observation that ill-conditioning of the fundamental solution matrix is not a primary issue for the method.
\subsection{Duffing oscillator}\label{sec:example:duffing}
\begin{table}[t]
	\centering
	\caption{Parameter values for  Duffing oscillator.}
	\label{tbl:duffing:params}
	\begin{tabular}{lrrrrrrr}
		\toprule
		& \multicolumn{2}{r}{stiffnesses} & damping & \multicolumn{2}{r}{excitation} & \multicolumn{2}{r}{$\norm{\vJ_k}$ decay}\\
		\cmidrule(lr){2-3}\cmidrule(lr){4-4}\cmidrule(lr){5-6}\cmidrule(lr){7-8} 
		Parameter & $\alpha$ & $\beta$ & $\delta$ & $F$ & $\omega$ & $a$ & $b$\\
		\midrule
		Configuration 1 & $5$ & $0.1$ & $0.02$ & $0.1$ & $5$ & $5.00$ & $7.40$ \\
		Configuration 2 & $0.5$ & $3$ & $0.05$ & $0.1$ & $0.3$ & $6.74$ & $1.12$ \\
		\bottomrule
	\end{tabular}
\end{table}
\begin{figure}[hbt]
	\centering
	\tikzsetnextfilename{duffing_perisol_case_1}
	\begin{subfigure}[t]{0.48\textwidth}
		\centering
		\includegraphics{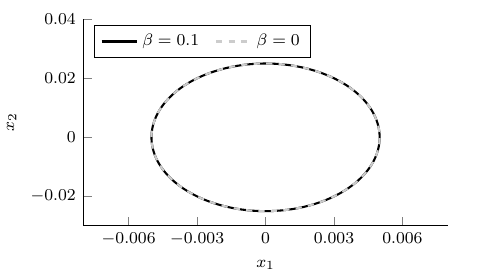}
		\caption{Periodic solution in Configuration $1$ (solid blue) and solution of the corresponding linear system (dashed black).}
		\label{fig:duffing:perisol:lin}
	\end{subfigure}
	\hfill
	\tikzsetnextfilename{duffing_perisol_case_2}
	\begin{subfigure}[t]{0.48\textwidth}
		\centering
		\includegraphics{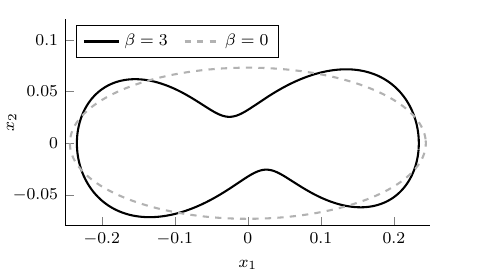}
		\caption{Periodic solution in Configuration $2$ (solid blue) and solution of the corresponding linear system (dashed black).}
		\label{fig:duffing:perisol:nonlin}
	\end{subfigure}
	\tikzsetnextfilename{duffing_spectrum_case_1}
	\begin{subfigure}[t]{0.48\textwidth}
		\centering
		\includegraphics{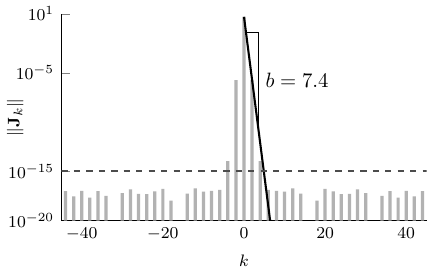}
		\caption{Norm of the Fourier coefficient matrices in Configuration 1.}
		\label{fig:duffing:spectrum:lin}
	\end{subfigure}
	\hfill
	\tikzsetnextfilename{duffing_spectrum_case_2}
	\begin{subfigure}[t]{0.48\textwidth}
		\centering
		\includegraphics{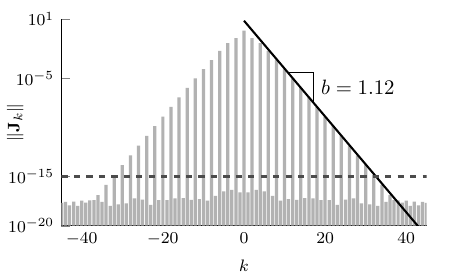}
		\caption{Norm of the Fourier coefficient matrices in Configuration 2.}
		\label{fig:duffing:spectrum:nonlin}
	\end{subfigure}
	\caption{Periodic solutions for the two considered configurations of the Duffing oscillator and norm of Fourier coefficient matrices.}
	\label{fig:duffing}
\end{figure}
In this example, we consider the forced Duffing oscillator to showcase the applicability of the error bound to stability determination of periodic solutions. 
The considered forced Duffing oscillator in first-order form and the corresponding system matrix for a perturbation around a solution $\vx(t)$ are
\begin{align}
	\dot{\vx} = \vf(\vx) = \begin{pmatrix}
		x_2\\
		-\alpha x_1 - \beta x_1^3 - \delta x_2 + F \cos \omega t
	\end{pmatrix} \;, &&
		\vJ(t) = \pdat{\vf}{\vx}{\vx(t)} = \begin{pmatrix}
			0& 1 \\ -\alpha - 3 \beta x_1(t)^2 & -\delta
		\end{pmatrix} \;.
	\end{align}
We consider two parameter settings, detailed in Table~\ref{tbl:duffing:params}.

The periodic solutions of the two configurations, depicted in Figure~\ref{fig:duffing}, were determined using the harmonic balance method with $N = 45$ Fourier coefficients. Configuration 1, where the forcing frequency is well beyond the resonance frequency and the cubic stiffness is low, could be labeled as an almost-linear case where the periodic solution visually coincides with that of the corresponding linear system ($\beta = 0$, all other parameters as in Table~\ref{tbl:duffing:params}), cf.\ Figure~\ref{fig:duffing:perisol:lin}. In Configuration 2, where the cubic stiffness is high and the forcing frequency is low, the higher-frequency components of the periodic solution due to the nonlinear effects are clearly visible in Figure~\ref{fig:duffing:perisol:nonlin}. 

The exponential decay of the Fourier coefficients is analyzed in Figures~\ref{fig:duffing:spectrum:lin} and~\ref{fig:duffing:spectrum:nonlin}. In contrast to the previous examples, we now do not have finite support of Fourier coefficients. The exponential parameters $a$ and $b$ were fitted to the norm of the Fourier coefficient matrices, where only Fourier coefficient matrices with a numerically determined norm larger than $10^{-15}$ have been taken into account due to finite numerical precision. The resulting parameters for both configurations are reported in Table~\ref{tbl:duffing:params}. 

Figure~\ref{fig:duffing:N} shows in solid lines the truncation orders $N^*$ that are expected to guarantee an error smaller than~$10^{-6}$ using the error bound~\eqref{eq:conv:Nmin}. 
At each time $t$, the fundamental solution matrices obtained using Equation~\eqref{eq:background:KoopHill} with increasing truncation orders~$N$ were compared against a reference fundamental solution matrix obtained by integrating Equation~\eqref{eq:background:ode:matrix} from $0$ to $t$ using Matlab's \texttt{ode45} numerical integrator with absolute and relative tolerances set to $10^{-10}$. 
For every considered time sample $t$, the lowest value of $N$ where the norm of the difference between these two fundamental matrices is lower than $10^{-6}$ is reported in dashed lines in Figure~\ref{fig:duffing:N}. 

As in the previous example, the error bound is not tight. In particular, the expected truncation orders $N^*$ increase linearly with $t$, while the numerically determined truncation orders $N_{\mathrm{num}}$ do not. This renders the error bound overly conservative for larger $t$, especially if $b$ is small.

\begin{figure}[hbt]
	\centering
	\begin{subfigure}[t]{0.48\textwidth}
		\centering
		\includegraphics{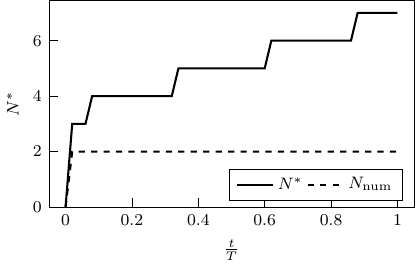}
		\caption{Configuration 1.}
		\label{fig:duffing:N:1}
	\end{subfigure}
	\hfill
	\begin{subfigure}[t]{0.48\textwidth}
		\centering
		\includegraphics{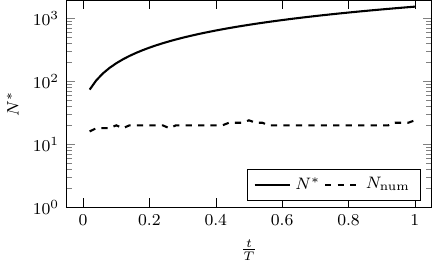}
		\caption{Configuration 2.}
		\label{fig:duffing:N:2}
	\end{subfigure}
	\caption{Truncation orders $N$ needed to guarantee (solid) and numerically achieve (dashed) an error smaller than $10^{-8}$ for the fundamental solution matrix of the Duffing oscillator.}
		\label{fig:duffing:N}
\end{figure}

\section{Conclusion}\label{sec:discussion}\label{sec:conclusion}
In this paper, we provided a convergence proof and the explicit error bounds given by Equations~\eqref{eq:conv:bound} and~\eqref{eq:subh:bound} for the numerical computation of the fundamental solution matrix of an LTP system using two variants of the Koopman-Hill projection method. The proof relies on expressing both the fundamental solution matrix as well as its approximation as series and then comparing the summands. While the series expressions are well-defined for all time-periodic system matrices where the Fourier coefficients decay exponentially, the error bounds and thus the convergence proof are valid for systems where the decay rate is sufficiently large ($b > \ln 2$).

Theorem~\ref{thm:proof:error} provides valuable theoretical backing for the use of the Hill matrix in the context of the harmonic balance method. 
Up until now, the only available convergence guarantee for any Hill-matrix-based method was in the context of the imaginary-part-based sorting method~\cite[Prop.~3]{Zhou2004}. While applicable to system matrices that are differentiable with piecewise continuous derivative, which is a larger system class than our proof, the proof of~\cite{Zhou2004} does not come with an explicit error bound and gives no indication at all about how to choose $N$ to be sufficiently large in practice. 
Thus, the error bound derived in the present work is the most explicit one available for all Hill matrix approaches and provides a helpful justification for the practitioner with an upper bound on the required truncation order $N$. 

The examples of Section~\ref{sec:examples} show that the error bound is conservative and overestimates the required truncation order $N$. This conservatism arises from the bounding steps employed to enable closed-form expressions of the series. Specifically, to upper-bound the number of index tuples with a given $1$-norm in the series for the error, \emph{all} integer index tuples that do not lie in the largest rhombus within $\cP_{0}^{(m)}$ (or $\cPsub^{(m)}$) were counted, irrespective of whether they actually lie inside $\cP_0^{(m)}$ (resp.\ $\cP^{(m)}_{\mathrm{subh}}$) or not. The ratio of tuples in this rhombus over all tuples in $\cP_0^{(m)}$ or $\cP^{(m)}_{\mathrm{subh}}$ goes to zero as $m$ grows. Hence, especially for large $m$, our error bound includes significantly more summands than necessary, contributing to overly conservative convergence requirements. 
Furthermore, the error bound depends exponentially on the time $t$. However, the numerical results do not indicate any dependence on time at all. Potentially, a tighter bound independent of $t$ could be derived by exploiting the periodicity of the scalar factors or finding a better method to count only those scalar factors that actually lie outside the eligible index set. 
Additionally, algebraic simplifications like $\left( \frac{x}{1-x}\right)^m < 1$ and $\binom{N + m + 1}{m}^{-1} < 1$ in the proof of Lemma~\ref{lem:series_eval} or summing from $0$ to $N+m$ instead of $1$ to $m$ in the final step of the convergence proof become more conservative as $N$ and $m$ grow. 
For these reasons, we postulate that a similar bound to Equation~\eqref{eq:conv:bound}, but independent of $t$ and applicable whenever Assumption~\ref{assu:b} is satisfied, is likely to exist. This would make the result tighter and applicable to all analytic linear periodic differential equations. Further work is needed to derive such a bound.

This paper focused mainly on the error in the fundamental solution matrix, caused by truncation of the Hill matrix. We illustrated in Section~\ref{sec:exanple:mathieu} how the effect of the bound on the Floquet multipliers can be analyzed using the pseudospectrum of the approximation. 

In practice, all involved operations are executed with finite machine precision on a computer. The error due to numerical procedures, in particular the error of evaluating the matrix exponential, is not considered in the error bound of Equation~\eqref{eq:conv:bound}. However, the accurate numerical evaluation of matrix exponentials has been well researched~\cite{Moler2003,AlMohy2009,Ibanez2022}, so the error due to the finite-precision procedures is expected to be negligible compared to the truncation error in practical applications.
\appendix

 \section{Construction of the series expressions of Theorems~\ref{thm:proof:Phi_series:xi} and~\ref{thm:proof:Q}}\label{sec:app:construction}
The proofs of Theorems~\ref{thm:proof:Phi_series:xi} and~\ref{thm:proof:Q} in Section~\ref{sec:series}  rely on differentiation of the proposed series expressions, verifying that they uniquely solve the associated matrix initial value problems. While this approach is concise and rigorous, it sheds no light on the underlying construction of the series representations in Equations~\eqref{eq:proof:Phi_series:xi} and~\eqref{eq:proof:Q:series}. To complement the main developments, this section outlines an alternative procedure for deriving these series. Specifically, we consider the Taylor expansions of $\vPh(t)$ and $\vQ(t)$ about $t = 0$, substitute the Fourier series representation of $\vJ(t)$, and employ inductive arguments to reorganize the resulting expressions into the forms given by Equations~\eqref{eq:proof:Phi_series:xi} and~\eqref{eq:proof:Q:series}. As complete proofs are already provided in Section~\ref{sec:series}, we restrict ourselves here to a conceptual overview, assuming absolute convergence and sufficiently large radius of convergence for all series encountered.

\subsection{Taylor series for \texorpdfstring{$\vPh(t)$}{true fundamental matrix}}\label{sec:app:taylor:Phi}
To write the fundamental solution matrix $\vPh(t)$ given by the matrix initial value problem of Equation~\eqref{eq:background:ode:matrix} as a Taylor series around $0$, all derivatives of $\vPh(t)$ and $\vQ(t)$ are required. The $0$-th and first derivative are immediately clear from the matrix initial value problem~\eqref{eq:background:ode:matrix}. Further derivatives can be computed using the differential equation and the product rule, utilizing the Fourier series expression of $\vJ(t)$. This is formalized in the following lemma.

\begin{lemma}\label{lem:app:Phi:Jn}
	For all natural $l > 0$, the $l$-th derivative of the fundamental solution matrix $\vPh(t)$ defined by Equation~\eqref{eq:background:ode:ext} is
	\begin{align}\label{eq:app:Phi:JPhi}
		\td{^l}{t^l} \vPh(t) := \vJ^{(l)}(t) \vPh(t) 
	\end{align}
	with 
	\begin{align}\label{eq:app:Phi:Jn}
		\vJ^{(l)}(t) = \sum_{m = 1}^l (\ic \omega)^{l-m} \sum_{\vp \in \Zspace^m}
 \left( \sum_{\substack{\val \in \Nspace^m\\ \abs{\val} = l - m}} \left[ p_1, (p_1 + p_2), \dots, (p_1 + \dots + p_m) \right]^{\val}\right) \cJ_{\vp} \ex^{\ic \left( \sum_{k = 1}^m p_k\right)\omega t} \;.
	\end{align}
\end{lemma}
\begin{proof}
	This lemma can be proven by induction. The base case $n = 1$ with $\vJ^{(1)}(t) = \sum_{p \in \Zspace} \vJ_p \ex^{\ic p \omega t} = \vJ(t)$ is immediately clear. For the induction step, differentiation of Equation~\eqref{eq:app:Phi:JPhi} and the use of the IVP reveal
	\begin{align}
	\td{^{(l+1)}}{t^{(l+1)}} \vPh(t) = \td{}{t}\left(\vJ^{(l)}(t)\right) \vPh(t) + \vJ^{(l)}(t) \vJ(t) \vPh(t) \;,
	\end{align}
	i.e., $\vJ^{(l+1)}(t) = \td{}{t}\vJ^{(l)}(t) + \vJ^{(l)}(t) \vJ(t)$ is certainly a $T$-periodic matrix. Note that $\vJ^{(l)}(t)$ is generally not the $l$-th derivative of $\vJ(t)$. Substituting in Equation~\eqref{eq:app:Phi:Jn} and the Taylor series expansion of $\vJ(t)$ yields after some algebraic manipulation and index shifts
	\begin{align}\label{eq:app:Phi:Jn:dot}
		\td{}{t} \vJ^{(l)}(t) &= 
		\sum_{m = 1}^{l} (\ic \omega)^{l-m+1} 
		\sum_{\vp \in \Zspace^m} 
		\left( 
			\sum_{\substack{\val \in \Nspace^{m}\\
					\abs{\val} = l - m + 1 \\
					\alpha_m \geq 1}
				} 
			\left[ 
				p_1, (p_1 + p_2), \dots, (p_1 + \dots + p_m) 
			\right]^{\val} 
		\right) 
		\cJ_{\vp} \ex^{\ic \left( \sum_{k = 1}^m p_k\right) \omega t} \\
		\label{eq:app:Phi:Jn:J}
		\vJ^{(l)}(t) \vJ(t) &= 
		\sum_{m = 2}^{l+1} (\ic \omega)^{l-m+1} 
		\sum_{\vp \in \Zspace^{m} } 
		\left( 
			\sum_{\substack{\val \in \Nspace^{m}\\
					\abs{\val} = l - m + 1 \\
					\alpha_m = 0}
				} 
			\left[ 
				p_1, (p_1 + p_2), \dots, (p_1 + \dots + p_m) 
			\right]^{\val} 
		\right) 
		\cJ_{\vp} \ex^{\ic \left( \sum_{k = 1}^m p_k\right) \omega t}\;.
	\end{align}  
	The summation sets in the round brackets of Equations~\eqref{eq:app:Phi:Jn:dot} and~\eqref{eq:app:Phi:Jn:J} are disjunct. As the summation over~$\val$ in Equation~\eqref{eq:app:Phi:Jn:dot} has no valid summands for $m = l+1$ and, similarly, there are no valid summands for~$\val$ with $m = 1$ in Equation~\eqref{eq:app:Phi:Jn:J}, both series over $m$ in Equations~\eqref{eq:app:Phi:Jn:dot} and~\eqref{eq:app:Phi:Jn:J} can be extended to sum from $m = 1$ to $m = l+1$. Together, Equations~\eqref{eq:app:Phi:Jn:dot} and~\eqref{eq:app:Phi:Jn:J} constitute the complete summation within the round brackets of Equation~\eqref{eq:app:Phi:Jn}. 
\end{proof}
With Lemma~\ref{lem:app:Phi:Jn}, the Taylor series of $\vPh(t)$ is given by $\vPh(t) = \vI + \sum_{l = 1}^{\infty} \vJ^{(l)}(0) \frac{t^l}{l!}$. Substituting in Equation~\eqref{eq:app:Phi:Jn} and changing the order of the $m$- and $l$-summations yields
\begin{align}
	\vPh(t) =\vI + \sum_{m = 1 }^{\infty} \sum_{\vp \in \Zspace^m} \sum_{l = m}^{\infty} \sum_{\substack{\val \in \Nspace^m \\ \abs{\val}  = l - m}} \frac{t^l}{l!} (\ic \omega) ^{l -m} \left[ p_1, (p_1 + p_2), \dots, (p_1 + \dots + p_m) \right]^{\val} \cJ_{\vp} \;.
\end{align}
Noting that every $\val \in \Nspace^m$ fulfills $l = m + \abs{\val}$ for exactly one value of $l \geq m$, the $l$-summation can be eliminated and the desired Equation~\eqref{eq:proof:Phi_series:xi} results.

\subsection{Taylor series for \texorpdfstring{$\vQ(t)$}{Q(t)}}
To express $\vQ(t)$ as a Taylor series around $0$, again all derivatives are needed. Analogous to Lemma~\ref{lem:app:Phi:Jn}, the following lemma provides a series expression for derivatives of $\vQ(t)$.

\begin{lemma}\label{lem:app:Q:An}
	For $l \in\Nspace$, the $l$-th derivative of $\vQ(t)$ as defined in Equation~\eqref{eq:proof:KoopHill:Q_defin} is given by
	\begin{align}\label{eq:app:Q:Qdot}
		\td{^l}{t^l} \vQ(t) = \ex^{\ic \omega \vD t} \vA ^{(l)} \ex^{\vH t} \vW \;.
	\end{align}
	The matrix $\vA^{(l)} \in \Cspace^{n(2N+1) \times n(2N+1)}$ is constant and consists of blocks $\vA^{(l)}_{jk}$, $j, k = -N, \dots, N$, of size $n\times n$. For $l = 0$, it holds that $\vA^{(0)} =\vI$ and for $l \geq 1$
	\begin{align}\label{eq:app:Q:Ajk}
		\vA^{(l)}_{jk} = \sum_{m = 1}^{l} (\ic \omega)^{l-m} \sum_{\vp \in \cA_{jk}^{(m)}} \left(\sum_{\substack{\val \in \Nspace^m \\ \abs{\val} = l -m}} \left[ p_1, (p_1 + p_2), \dots, (p_1 + \dots +  p_m)\right]^{\val} \cJ_{\vp}\right) \;.
	\end{align}
	The index set $\cA_{jk}^{(m)}$ is a subset of $\Zspace^m$ that is independent of $l$ and is given by
	\begin{align}\label{eq:app:Q:Ajk:set}
		\cA_{jk}^{(m)} = \left\{  \vp \in \Zspace^m : k = j - \sum_{i = 1}^m p_i \textrm{~and~} \abs{j - \sum_{i =1}^w p_i} \leq N \textrm{~for~}w= 1, \dots, m \right\} \;.
	\end{align}
\end{lemma}

The integer index set $\cP_j^{(m)}$ which plays an important role in the main part of this work is given by the union of $\cA^{(m)}_{jk}$ over $k = -N, \dots, N$. 
\begin{proof}
	The proof is similar to the one of Lemma~\ref{lem:app:Phi:Jn}. The induction base case $\vA^{(0)} = \vI$ is true by definition. By differentiating Equation~\eqref{eq:app:Q:Qdot}, we obtain
	\begin{align}
		\td{^{l+1}}{t^{l+1}} \vQ(t) = \ex^{\ic \omega \vD t} \left( \ic \omega \vD \vA^{(l)} + \vA^{(l)} \vH \right) \ex^{\vH t} \vW \;,
	\end{align}
	which provides a recursive relationship for $\vA^{(l)}$. A slight reformulation yields
	\begin{align}\label{eq:app:Q:An:recursive}
		\vA^{(l+1)} = \ic \omega \left( \vD \vA^{(l)} - \vA^{(l)} \vD\right) + \vA^{(l)} \left( \vH + \ic \omega \vD\right) \;.
	\end{align}
	In the first summand of Equation~\eqref{eq:app:Q:An:recursive}, the $(jk)$-th $n \times n$ block of $\vA^{(l)}$ is scaled by $\ic \omega (j-k)$. In the second summand, the addition of $\ic \omega \vD$ removes the corresponding terms on the diagonal of $\vH$ (cf.\ Equation~\eqref{eq:background:H:finite}), such that the matrix that is multiplied to $\vA^{(l)}$ is purely block-Toeplitz. 

	For $l = 1$, Equation~\eqref{eq:app:Q:An:recursive} becomes $\vA^{(1)} = \vH + \ic \omega \vD$, whose $(jk)$-th block is the Fourier coefficient $\vJ_{j-k}$ (cf.\ Equation~\eqref{eq:background:Hillmat}). With $l=1$, the summation over $m$ in Equation~\eqref{eq:app:Q:Ajk} only has one summand. It can be verified easily that the corresponding index set $\cA_{jk}^{(1)}$ contains only $p = k-j$ and thus Equation~\eqref{eq:app:Q:Ajk} becomes $\vJ_{j-k}$ for $l = 1$.
	For $l > 1$, it can be shown by an induction procedure, which is analogous to the one in~\ref{sec:app:taylor:Phi} but significantly more tedious due to the sets $\cA_{jk}^{(m)}$, that the blocks $\vA_{jk}$ are indeed given by Equation~\eqref{eq:app:Q:Ajk}. 
\end{proof}
With Lemma~\ref{lem:app:Q:An}, the $l$-th derivative of $\vQ$ at $0$ is $\vA^{(l)} \vW$. As $\vW$ consists of identity matrices, the $j$-th block of $\vA^{(l)} \vW$ can be written as $\sum_{k = -N}^N \vA_{jk}^{(l)}$. Analogously as in~\ref{sec:app:taylor:Phi}, establishing a Taylor series using all these derivatives leads to Equation~\eqref{eq:proof:Q:series}. 

 \section{Periodicity of the scalar factor}\label{sec:app:periodicity}
 An observation that can be drawn from Figure~\ref{fig:proof:xi_p} is that the scalar factor $\xi_{\vp}(t)$ has periodic components, possibly multiplied by a monomial in $t$. We show here that this monomial vanishes for many values of $\vp$, making $\xi_{\vp}(t)$ in these cases $T$-periodic and bounded.
 \begin{theorem}[sufficient criterion for $T$-periodicity of $\xi_{\vp}$] Let $m \in \Nspace \setminus \left\{0\right\}$ and consider an integer index tuple $\vp = [p_1, \dots, p_m] \in \Zspace^m$ which fulfills $\sum_{l = v}^w p_l \neq 0$ for all $v, w = 1, \dots, m$. Then, the function $\xi_{\vp}$ is $T$-periodic and has a Fourier series $\xi_{\vp}(t) =: \sum_{k = -\abs{\vp}}^{\abs{\vp}}  \xi_{\vp}^{(k)} \ex^{\ic k \omega t}$ with finite support. In particular, the $k$-th Fourier coefficient $\xi_{\vp}^{(k)}$ can only be nonzero if $k = 0$ or if there exists a $w \in \left\{1, \dots, m\right\}$ such that $k = \sum_{l = 1}^w p_l$.
 \end{theorem}
 \begin{proof}
 	We prove this by induction. 
 	\paragraph{Base case $m = 1$} Integrating the first statement of Lemma~\ref{lem:proof:xi_p:deriv} for an arbitrary $p \in \Zspace \setminus \left\{ 0 \right\}$ with $\xi_p(0) = 0$ yields
	\begin{align}
		\xi_{p}(t) = \frac{1}{\ic \omega p} \left( \ex^{\ic \omega p t} - 1 \right) \;.
	\end{align}
	In particular, $\xi_p(t)$ is $T$-periodic and only the $0$-th and $p$-th Fourier coefficients are nonzero.
 	\paragraph{Induction assumption} Let $m \geq 2$. Consider an integer index tuple $\vp = [p_1, p_2, \dots, p_m] \in \Zspace^m$ fulfilling the conditions of the theorem. The tuple $[p_2, \dots, p_m] \in \Zspace^{m-1}$ fulfills the conditions of the theorem as well. The induction assumption is that $\xi_{[p_2, \dots, p_m]}(t)$ is $T$-periodic and its Fourier coefficients $\xi_{[p_2, \dots, p_m]}^{(k)}$ are only nonzero if $k = 0$ or if there exists a $w$ such that $k = \sum_{l = 2}^w p_l$.
 	\paragraph{Induction step}
 	Using the induction assumption, the second statement of Lemma~\ref{lem:proof:xi_p:deriv}, and the initial condition $\xi_{\vp}(0) = 0$, $\xi_{\vp}$ can be expressed by
 	\begin{align}
 		\xi_{\vp}(t) = \int_{0}^t \xi_{[p_2, \dots, p_m]}(\tau) \, \ex^{\ic \omega p_1 \tau} \diff \tau 
		= \int_{0}^{t} 
		\sum_{k = -\abs{\vp} + \abs{p_1}}^{\abs{\vp} -\abs{p_1}}  \xi_{[p_2, \dots, p_m]}^{(k)} \ex^{\ic \omega (k + p_1) \tau} \diff \tau \;.
 	\end{align}
 	This integral can be evaluated summand by summand. For $k = -p_1$, the exponential term in the integrand becomes $1$, yielding the non-periodic, linear term
 	\begin{align}\label{eq:openwork:nonperi}
 		\int_{0}^t \xi^{(-p_1)}_{[p_2, \dots, p_m]} \diff \tau = t \, \xi^{(-p_1)}_{[p_2, \dots, p_m]} \;.
 	\end{align}
 	Assume now that $\xi^{(-p_1)}_{[p_2, \dots, p_m]}$ is nonzero. As $p_1 = \sum_{l = 1}^1 p_l \neq 0$, by the induction assumption there must exist a~$w$ such that $-p_1 = \sum_{l = 2}^w p_l$. But this is prohibited by construction of $\vp$ as it would imply $\sum_{l = 1}^w p_l = 0$. We conclude that $\xi^{(-p_1)}_{[p_2, \dots, p_m]}$ must be zero and $\xi_{\vp}(t)$ does not have a non-periodic term of the form of  Equation~\eqref{eq:openwork:nonperi}. 
 	
 	For $k \neq -p_1$, the exponential term in the integrand does not disappear and we obtain
 	\begin{align}\label{eq:openwork:peri}
 		\int_{0}^t \xi^{(k)}_{[p_2, \dots, p_m]} \ex^{\ic \omega (k+p_1) \tau}\diff \tau = \frac{1}{\ic \omega (k + p_1)} \xi^{(k)}_{[p_2, \dots, p_m]} \left( \ex^{\ic \omega (k + p_1) t} - 1\right) \;.
 	\end{align}
 	As Equation~\eqref{eq:openwork:nonperi} vanishes and all other summands are of the form of Equation~\eqref{eq:openwork:peri}, which is $T$-periodic, $\xi_{\vp}$ is again $T$-periodic. In particular, Equation~\eqref{eq:openwork:peri} allows to read off the Fourier coefficients of~$\xi_{\vp}$:
 	\begin{subequations}
 		\begin{align}
 			\xi_{\vp}^{(0)} &= \sum_{k = -\abs{\vp}}^{\abs{\vp}} \frac{-1}{\ic \omega \left(  k + p_1 \right)} \xi^{(k)}_{[p_2, \dots, p_m]}\\
 			\xi_{\vp}^{(k)} &= \frac{1}{\ic \omega k} \xi^{(k - p_1)}_{[p_2, \dots, p_m]} & k &\neq 0 \;.
 		\end{align}
 		By the induction assumption, $\xi_{\vp}^{(k)}$ for $k \neq 0$ can only be nonzero if there is a  $w$ such that $k - p_1 = \sum_{l = 2}^w p_l$, which completes the proof.
 	\end{subequations}
 	
 \end{proof}

 \section{Combinatorics}\label{sec:app:combi}
 This section and the following one collect some particular equivalences that are needed to handle the various series expressions occurring in this work. Some of these results are well-known and revisited here for the sake of completeness, while others are rather specific. We begin by stating a well-known result on the number of multi-indices with a given 1-norm.
\begin{lemma}[Stars and bars \cite{Feller1968}]\label{lem:proof:starsbars}~
	\begin{enumerate}[label=\roman*)]
		\item The set $\left\{ \val \in \Nspace^m :  \abs{\val} = M \right\}$ has exactly $\binom{M + m - 1}{m - 1}$ elements.\label{lem:proof:starsbars:item:alpha}
		\item The set $\left\{ \vp \in \Zspace^m :  \abs{\vp} = M \right\}$ has fewer than $2^m \binom{M + m - 1}{m - 1}$ elements.\label{lem:proof:starsbars:item:p}
	\end{enumerate}
\end{lemma}
\begin{proof}
	Proposition~\ref{lem:proof:starsbars:item:alpha} is a classical combinatorics result that can be found, e.g., in~\cite{Feller1968}. For any nonnegative tuple $\val \in \Nspace$ there are at most $2^m$ tuples $\vp \in \Zspace$ which fulfill $\abs{p_k} = \alpha_k$ for all $k = 1, \dots, m$ as every entry can be either positive or negative (yielding two independent options for each of the $m$ entries), except if it is zero. This immediately proves Proposition~\ref{lem:proof:starsbars:item:p}.
\end{proof}
Next, we provide here a variant of Vandermonde's identity~\cite{Feller1968} with the summation index in the upper entries, which enables a follow-up result about summation over multiple indices.
\begin{lemma}\label{lem:proof:prelim:combin:2sum}
	For arbitrary $n, M, P \in \Nspace$, the following expression holds:
	\begin{align}
		\sum_{\alpha = 0}^M \binom{\alpha + n}{n} \binom{M + P - \alpha}{M-\alpha} = \binom{M + P + n + 1}{M} \;.\label{eq:proof:vandermonde}
	\end{align}
\end{lemma}
\begin{proof}
	This proof follows a classical counting argument. Consider the set $\cD := \left\{ 1, \dots, M + P + n + 1\right\}$, which contains all strictly positive natural numbers up to ${M+P + n + 1}$. The number of subsets of the form $\cA := \left\{ a_1, \dots, a_{P+n+1} \right\} \subset \cD$ with $P + n + 1$ unique elements $a_1 < a_2 < \dots < a_{P+n+1}$ is $\binom{M + P + n + 1}{P+n+1} = \binom{M + P + n + 1}{M}$, which is the right-hand side of Equation~\eqref{eq:proof:vandermonde}. 
	
	For the left-hand side, we construct another way to count these subsets.		
	The $(n+1)$-th element $a_{n+1}$ must have a value between $n+1$ (implying $a_k = k$ for the $n$ elements $a_k$ with $k < n+1$) and $n+1+M$ (implying $a_k = M + k$ for  the $P$ elements with $k > n+1$). 	
	Suppose that $a_{n+1} = n+1+\alpha$ for some $\alpha \in \left\{0, \dots, M\right\}$. As the elements before $a_{n+1}$ must have smaller value and the elements after must have larger value, these subsets of $\cA$ must fulfill
	\begin{align}
		\cA_{-} &:= \left\{ a_1, \dots, a_{n}\right\} \subset \left\{1, \dots, n + \alpha \right\} \\
		\cA_{+} &:= \left\{ a_{n+2}, \dots, a_{n+P+1}\right\} \subset \left\{ n + \alpha + 2, \dots, M + P + n + 1\right\} \;.
	\end{align}
	Hence, for $\cA_{-}$ we choose $n$ out of $n+\alpha$ values, while for $\cA_{+}$ we choose $P$ out of $M+P-\alpha$ values. In summary, for every fixed $\alpha$ there are $\binom{n+\alpha}{n} \binom{M+P-\alpha}{P} = \binom{n+\alpha}{n} \binom{M+P-\alpha}{M-\alpha} $ subsets~$\cA$ where $a_{n+1} = n + 1 + \alpha$, and summing over all possible values of $\alpha$ completes the proof.
\end{proof}
This lemma enables a follow-up statement that deals with multiple sums.
\begin{lemma}\label{cor:proof:prelim:vandermonde:multsums}
	For arbitrary $M \in \Nspace$, $m \in \Nspace \setminus \left\{ 0 \right\}$ and $\vn = [n_1, \dots, n_m] \in \Nspace^m$, the following expression given by $m$ finite-length sums holds:
	\begin{align}
		\sum_{\alpha_1 = 0}^M \sum_{\alpha_2 = 0}^{M - \alpha_1} \dots \sum_{\alpha_m = 0}^{M - \alpha_1 - \dots - \alpha_{m-1}} \binom{\alpha_1 + n_1}{n_1} \dots \binom{\alpha_m + n_m}{n_m} = \binom{m + M + \abs{\vn}}{M} \;. \label{eq:proof:prelim:vandermonde:multsums}
	\end{align}
\end{lemma}
\begin{proof}
	We prove this statement using Lemma~\ref{lem:proof:prelim:combin:2sum} and induction. 
	\paragraph{Base case $m = 1$} The base case is covered immediately by Lemma~\ref{lem:proof:prelim:combin:2sum} with $P = 0$.
	\paragraph{Induction assumption} Assume that Lemma~\ref{cor:proof:prelim:vandermonde:multsums} holds for the case $m-1$, i.e. for $m-1$ summation symbols. By increasing the index of all $\alpha_i$ and $n_i$ by 1 and replacing $M$ by $M-\alpha_1$ for arbitrary $\alpha_1$ and $M \geq \alpha_1$, we can rewrite the induction assumption as
	\begin{align}
		\sum_{\alpha_2 = 0}^{M - \alpha_1} \dots \sum_{\alpha_m = 0}^{M - \alpha_1 - \dots - \alpha_{m-1}} \binom{\alpha_2 + n_2}{n_2} \dots \binom{\alpha_m + n_m}{n_m} = \binom{(m-1) + (M-\alpha_1) + n_2 + \dots + n_m}{M - \alpha_1} \;. \label{eq:proof:prelim:vandermonde:multsums:m-1}
	\end{align}
	\paragraph{Induction step}
	In Equation~\eqref{eq:proof:prelim:vandermonde:multsums}, the first binomial coefficient depends only on $\alpha_1$ and can be pulled outside of the sums over $\alpha_2, \dots, \alpha_m$. Afterwards, the induction assumption Equation~\eqref{eq:proof:prelim:vandermonde:multsums:m-1} can be identified in the inner sums and the statement left to prove is
	\begin{align}
		\sum_{\alpha_1 = 0}^M \binom{\alpha_1 + n_1}{n_1} \binom{m - 1 + M - \alpha_1 + n_2 + \dots + n_m }{M - \alpha_1} = \binom{m + M + \abs{\vn}}{M} \;.
	\end{align}
	As this statement is covered by Lemma~\ref{lem:proof:prelim:combin:2sum} with $P = m-1+ n_2 + \dots + n_m$, the proof is complete.
\end{proof}

\section{Taylor series of \texorpdfstring{$(1-x)^{-(1+k)}$}{(1-x)\textasciicircum(-(1+k))}}\label{sec:app:Taylor}
The following power series plays a central role in the convergence proof.
\begin{lemma}\label{lem:series_eval}
	For arbitrary $k \in \Nspace$, $q \in \Rspace$ with $q  \geq 0$, the power series
	\begin{align}
		\sum_{M = 0}^{\infty} \binom{M+k}{k} M^q x^M \label{eq:lem:series_eval:q}
	\end{align}
	converges absolutely if $\abs{x} < 1$. In the case $q = 0$, the series has the closed-form expression
	\begin{align}
		\sum_{M = 0}^{\infty} \binom{M+k}{k} x^M  =  (1 - x)^{-(1+k)} =: g(x) \;. \label{eq:lem:series_eval:noq}
	\end{align}
	The Taylor remainder for a Taylor polynomial of degree $N$ of this series is given by
	\begin{align}\label{eq:lem:series_eval:remainder}
		R_N(x) := \sum_{M = N+1}^{\infty} \binom{M + k}{k}x^M = x^N \left( \sum_{m = 0}^k \binom{N + k + 1}{N + m + 1} \left( \frac{x}{1-x}\right)^{m + 1} \right) \;.
	\end{align}
\end{lemma}
\begin{proof}
	Denote summands of Equation \eqref{eq:lem:series_eval:q} by $a_M$. The quotient criterion yields
	\begin{align}
		\lim_{M \rightarrow \infty} \abs{\frac{a_{M+1}}{a_M}} = \lim_{M \rightarrow \infty} \frac{M + k + 1}{M + k} \left(\frac{M+1}{M}\right)^q \abs{x} = \abs{x} \;,
	\end{align}
	proving absolute convergence for $\abs{x} < 1$ and $q \geq 0$. 
	For the special case $q = 0$, a straightforward induction procedure shows that the $M$-th derivative of 
	$g(x) = (1-x)^{-(1+k)}$ is
	\begin{align}
		g^{(M)}(x) := \td{^M}{x^M} (1 - x)^{- (1 + k)} = \frac{(M + k )!}{k!} (1 - x)^{-(M + k + 1)} && \text{~for all~} M \in \Nspace \;.
	\end{align}
	Evaluated at zero, these derivatives define the Taylor formula
	\begin{align}
		g(x) = (1-x)^{-(k+1)} = \sum_{M = 0}^N \frac{(M + k )!}{k!} \frac{x^M}{M!} + R_N(x)\;,
	\end{align}
	where $N \in \Nspace$ is the maximum degree of the Taylor polynomial and $R_N(x)$ is the remainder. 
	
	The Taylor remainder in integral form~\cite{Oberguggenberger2018} is given by
	\begin{align}
		R_N(x) &= \int_{0}^x \frac{(x - \tau)^N}{N!} g^{(N+1)}(\tau) \diff \tau = \int_0^x \frac{(x - \tau)^N}{N!} \frac{(N+1 + k)!}{k!} (1 - \tau)^{-(2 + k + N)} \diff \tau\;.
	\end{align}
	The substitution $u = (1-\tau)^{-1}$ of the integration variable simplifies this integral to
	\begin{align}\label{eq:lem:series_eval:remainder_u}
		R_N(x) &= \int_{1}^{(1-x)^{-1}} \frac{u^k}{k!} \, \frac{(N+k+1)!}{N!} \, (1 - (1-x)u)^N \diff u \;,
	\end{align}
	which is an integral over a polynomial in $u$. We evaluate this integral using $(k+1)$-times repeated integration by parts. The integrand of Equation~\eqref{eq:lem:series_eval:remainder_u} is given by the product of the functions 
	\begin{align}
		v^{(0)}(u) &:= \frac{u^k}{k!}\\
		W^{(0)}(u) &:= \frac{(N+k+1)!}{N!} (1 - (1-x)u)^N \;.
	\end{align}
	For $v$, we can immediately compute its $m$-th derivatives
	\begin{align}
		v^{(m)}(u) := \td{^m}{u^m} v^{(0)}(u) = \frac{u^{k-m}}{(k-m)!} && m = 0, \dots, k
	\end{align}
	and the $(k+1)$-th derivative is zero. The function $W^{(0)}$ has the $m$-th antiderivative
	\begin{align}
		W^{(m)}(u) := (x-1)^{-m} \, \frac{(N+k+1)!}{(N + m)!} \, (1 - (1-x) u)^{N+m}
	\end{align}
	such that $\td{^m}{u^m} W^{(m)}(u) = W^{(0)}(u)$ for $m = 0, \dots, k+1$. Performing $k+1$ times the integration by parts on Equation~\eqref{eq:lem:series_eval:remainder_u}, each time using the next derivative of $v$ and the next antiderivative of $W$, yields the formula
	\begin{align}\label{eq:lem:series_eval:partint}
		\int_{1}^{(1-x)^{-1}} v^{(0)} W^{(0)}\diff u = \sum_{m = 0}^{k} (-1)^m \left[ W^{(m+1)} v^{(m)}\right]_{1}^{(1-x)^{-1}} \!\!\!+ (-1)^{k+1} \int_{1}^{(1-x)^{-1}} W^{(k+1)} v^{(k+1)}\diff u \;,
	\end{align}
	where the dependence on $u$ was omitted for the sake of brevity. The integral on the right-hand side of Equation~\eqref{eq:lem:series_eval:partint} vanishes together with $v^{(k+1)}$, and each summand of the remaining boundary terms can be evaluated individually to
	\begin{align}\label{eq:lem:series_eval:partint:summand}
		(-1)^m \left[W^{(m+1)} v^{(m)}\right]_1^{(1\!-\!x)^{-1}} \!\!\! &= \left[ - (1\!-\!x)^{-(m+1)} \tfrac{(N+k+1)!}{(N + m + 1)! (k - m)!} (1 \!-\! (1\!-\!x)u)^{N + m + 1} u^{k - m} \right]_{1}^{(1-x)^{-1}} \nonumber \\
		&= x^N \left( \frac{x}{1-x}\right)^{m+1} \binom{N+k+1}{N+m+1} \;.
	\end{align}
	Substituting Equation~\eqref{eq:lem:series_eval:partint:summand} into Equation~\eqref{eq:lem:series_eval:partint} yields the desired expression, Equation~\eqref{eq:lem:series_eval:remainder}, for the remainder.
	
	To show that the remainder converges to zero as $N \rightarrow \infty$, we will proceed to bound the expression~\eqref{eq:lem:series_eval:remainder} from above. To lighten notation, we assume w.l.o.g that $x > 0$ (otherwise, replace $x$ by $\abs{x}$ and $R_N(x)$ by $\abs{R_N(x)}$ in the developments below). From the factorial expressions it is easy to see that 
	\begin{align}\label{eq:lem:series_eval:factorials}
		\binom{N + k + 1}{N + m + 1} = \binom{N + k + 1}{N + 1} \binom{k} {m} \binom{N + m + 1}{m}^{-1} \;. 
	\end{align}
	Noting that $\binom{N + m + 1}{m}^{-1} \leq 1$ for all $m \in \Nspace$, we substitute Equation~\eqref{eq:lem:series_eval:factorials} into the remainder~\eqref{eq:lem:series_eval:remainder} and use the binomial theorem to obtain
	\begin{align}
		R_N(x) &\leq x^N \frac{x}{1-x}\binom{N + k + 1}{N + 1} \sum_{m = 0}^k  \binom{k} {m} \left( \frac{x}{1-x} \right)^{m} 1 ^{k-m} \nonumber \\
		& = \frac{x^{N+1}}{(1-x)}\binom{N + k + 1}{N + 1} \left(\frac{x}{1-x} + 1\right)^k \nonumber \\
		& = \frac{x^{N+1}}{(1-x)^{k+1}}\binom{N + k + 1}{N + 1} \;.
	\end{align}
	Finally, using the bound $\binom{N+1+k}{k} \leq (N+1)^{k+1}$, we find for any fixed $k \in \Nspace$
	\begin{align}
		0 \leq \lim_{N \rightarrow \infty} R_N(x) \leq \lim_{N \rightarrow \infty} x^{N+1} \left(\frac{N+1}{1-x}\right)^{k+1} = 0 \;,
	\end{align}
	completing the proof. 
\end{proof}


\end{document}